\documentclass[a4paper,11pt,leqno]{article}

\usepackage[latin1]{inputenc}
\usepackage[T1]{fontenc} %only using latex
\usepackage[english]{babel}
\usepackage[notcite,notref]{showkeys}
\usepackage{verbatim}
\usepackage{enumerate}
\usepackage{amsfonts}
\usepackage{amsmath}
\usepackage{amsthm}
\usepackage{amssymb}
\usepackage{color}
\usepackage{graphicx}
\usepackage{bm}

\theoremstyle{definition}
\newtheorem{defin}{Definition}[section]

\newtheorem{rem}[defin]{Remark}

\theoremstyle{plane}
\newtheorem{thm}[defin]{Theorem}
\newtheorem{prop}[defin]{Proposition}
\newtheorem{coroll}[defin]{Corollary}
\newtheorem{lemma}[defin]{Lemma}

\newcommand{\mbb}{\mathbb}

\newcommand{\mc}{\mathcal}
\newcommand{\veps}{\varepsilon}
\newcommand{\what}{\widehat}
\newcommand{\wtilde}{\widetilde}
\newcommand{\vphi}{\varphi}
\newcommand{\oline}{\overline}

\newcommand{\ra}{\rightarrow}

\newcommand{\hra}{\hookrightarrow}

\newcommand{\g}{\gamma}
\newcommand{\lan}{\langle}
\newcommand{\ran}{\rangle}

\newcommand{\R}{\mathbb{R}}
\newcommand{\C}{\mathbb{C}}
\newcommand{\N}{\mathbb{N}}
\newcommand{\Z}{\mathbb{Z}}

\renewcommand{\div}{{\rm div}\,}

\newcommand{\Id}{{\rm Id}\,}

\allowdisplaybreaks

\def\d{\partial}
\def\div{{\rm div}\,}

\title{\textsc{\Large{\textbf{Highly rotating viscous compressible fluids \\ in presence of capillarity
effects}}}}

\author{\textsl{Francesco Fanelli} \vspace{.2cm} \\
{\small \textsc{Univ. Lyon, Universit\'e Claude Bernard Lyon 1}} \vspace{.1cm} \\
{\small  CNRS UMR 5208, \textit{Institut Camille Jordan}} \vspace{.1cm} \\
{\footnotesize 43 blvd. du 11 novembre 1918} \vspace{.1cm} \\
{\footnotesize F-69622 Villeurbanne cedex -- FRANCE} \vspace{.2cm} \\
{\small \ttfamily{fanelli@math.univ-lyon1.fr}} }

\vspace{.3cm}
\date\today

\begin{document}

\maketitle

%%%%%%%%%%%%%%%%%%%%%%%%%%%%%%%%%%%%%%%%%%%%%%%%%%%%%%%%%%%%%%%%%%%%%%%%%%%%%%%%%%%%%%%
\subsubsection*{Abstract}
{\small We study here a singular limit problem for a Navier-Stokes-Korteweg system with Coriolis force, in the domain
$\R^2\times\,]0,1[\,$ and for general ill-prepared initial data.
Taking the Mach and the Rossby numbers proportional to a small parameter $\veps\ra0$,
we perform the incompressible and high rotation limits simultaneously; moreover, we consider both the constant and
vanishing capillarity regimes.
In this last case, the limit problem is identified as a $2$-D incompressible Navier-Stokes equation in the variables
orthogonal to the rotation axis; if the capillarity is constant, instead, the limit equation slightly changes, keeping however a
similar structure, due to the presence of an additional surface tension term.
In the vanishing capillarity regime, various rates at which the capillarity coefficient goes to $0$ are considered: in general,
this produces an anisotropic scaling in the system.
The proof of the results is based on suitable applications of the RAGE theorem, combined with microlocal symmetrization arguments.}

\paragraph*{2010 Mathematics Subject Classification:}{\small 35Q35 % PDEs / Eq. of Math. Phys. / PDEs in connection with fluid mechanics
(primary); 35B25, % PDEs / Qualitative properties / Singular perturbations
35B40, % PDEs / Qualitative properties / Asymptotic behavior of solutions
35P25, % PDEs / Spectral theory and eigenvalue problems / Scattering theory
76U05, % Fluid Mechanics / Rotating fluids / Rotating fluids
47A40, % Operator theory / General theory of linear operators / Scattering theory
28A33, % Measure and integration / Classical measure theory / Spaces of measures, convergence of measures
47A55 % Operator theory / General theory of linear operators / Perturbation theory
(secondary).}

\paragraph*{Keywords:}{\small Navier-Stokes-Korteweg system; rotating fluids; singular perturbation problem; low Mach number;
low Rossby number; vanishing capillarity.}

%%%%%%%%%%%%%%%%%%%%%%%%%%%%%%%%%%%%%%%%%%%%%%%%%%%%%%%%%%%%%%%%%%%%%%%%%%%%%%%%%%%%%%%%%%%%5
%%%%%%%%%%%%%%%%%%%%%%%%%%%%%%%%%%%%%%%%%%%%%%%%%%%%%%%%%%%%%%
\section{Introduction} \label{s:intro}
%%%%%%%%%%%%%%%%%%%%%%%%%%%%%%%%%%%%%%%%%%%%%%%%%%%%%%%%%%%%%%%%%%%%%%%%%%%%%%%%%%%%%%%%%%%%5
%%%%%%%%%%%%%%%%%%%%%%%%%%%%%%%%%%%%%%%%%%%%%%%%%%%%%%%%%%%%%%%%%%%%%%%%%%%%%%%%%%%%%%%%%%%%5

Let us consider, in space dimension $d=3$, the Navier-Stokes-Korteweg system
$$
\begin{cases}
\d_t\rho\,+\,\div\left(\rho\,u\right)\,=\,0 \\[1ex]
\d_t\left(\rho\,u\right)\,+\,\div\bigl(\rho\,u\otimes u\bigr)\,+\,\nabla P(\rho)\,+\,
e^3\times\rho u\,-\,\div\bigl(\mu(\rho)\,Du\bigr)\,-\,
\kappa\,\rho\,\nabla\bigl(\sigma'(\rho)\,\Delta\sigma(\rho)\bigr)\,=\,0\,,
\end{cases}
$$
which describes the evolution of a compressible viscous fluid under the action both of the surface tension
and of the Coriolis force.

In the previous system, the scalar function $\rho=\rho(t,x)\geq0$ represents the density of the fluid, while $u=u(t,x)\in\R^3$ is
its velocity field. The function $P(\rho)$ is the pressure of the fluid, and throughout this paper we will suppose it to be given
by the Boyle law $P(\rho)=\rho^\g/\g$, for some $1<\g\leq2$ (these conditions will be justified in Section \ref{s:result}).
The positive function $\mu(\rho)$ represents the viscosity coefficient, the parameter $\kappa>0$
is the capillarity coefficient and  the term $\sigma(\rho)\geq0$ takes into account the surface tension. Here, we will always assume
(we will motivate this choice below)
$$
\mu(\rho)\,=\,\nu\,\rho\qquad\qquad\mbox{ and }\qquad\qquad \sigma(\rho)\,=\,\rho\,,
$$
for some fixed number $\nu>0$. Finally, the term
$$
e^3\,\times\,\rho u\,:=\,\bigl(-\rho\,u^2\,,\,\rho\,u^1\,,\,0\bigr)
$$
represents the Coriolis force, which acts on the system due to the rotation of the Earth. Here we have supposed that the rotation axis is
parallel to the $x^3$-axis and constant. Notice that this approximation is valid in regions which are very far from the equatorial zone
and from the poles, and which are not too extended: in general, the dependence of the Coriolis force on the latitude should be taken into
account (see e.g. works \cite{G_2008}, \cite{G-SR_2006} and \cite{G-SR_Mem}). On the other hand, for simplicity we are
neglecting the effects of the centrifugal force.

Now, taking a small parameter $\veps\in\,]0,1]$, we perform the scaling $t\mapsto\veps t$, $u\mapsto\veps u$,
$\mu(\rho)\mapsto\veps\mu(\rho)$ and we set $\kappa=\veps^{2\alpha}$, for some $0\leq\alpha\leq1$.
Then, keeping in mind the assumptions we fixed above, we end up with the system
\begin{equation} \label{intro_eq:NSK+rot}
\begin{cases}
\d_t\rho_\veps+\div\left(\rho_\veps u_\veps\right)\,=\,0 \\[2ex]
\d_t\left(\rho_\veps u_\veps\right)+\div\bigl(\rho_\veps u_\veps\otimes u_\veps\bigr)+ \\[1ex]
\qquad\qquad\qquad +\dfrac{1}{\veps^2}\,\nabla P(\rho_\veps)+
\dfrac{1}{\veps}\,e^3\times\rho_\veps u_\veps-\nu\div\bigl(\rho_\veps Du_\veps\bigr)-
\dfrac{1}{\veps^{2(1-\alpha)}}\,\rho_\veps \nabla\Delta\rho_\veps\,=\,0\,.
\end{cases}
\end{equation}
Notice that the previous scaling corresponds to supposing both the Mach number and the Rossby number to be proportional
to $\veps$ (see e.g. paper \cite{J-L-W}, or Chapter 4 of book \cite{F-N}).
We are interested in studying the asymptotic behavior of weak solutions to the previous system, in the regime of small $\veps$,
namely for $\veps\ra0$. In particular, this means that we are performing the incompressible limit, the high rotation limit and,
when $\alpha>0$, the vanishing capillarity limit simultaneously.

Many are the mathematical contributions to the study of the effects of fast rotation on fluid dynamics, under different assumptions
(about e.g. incompressibility of the fluid, about the domain and the boundary conditions\ldots).
We refer e.g. to book \cite{C-D-G-G} and the references therein for an extensive analysis of this problem for incompressible viscous
fluids. In the context of compressible fluids there are, to our knowledge, few works, dealing with different models:
among others that we are going to present more in detail below, we quote here \cite{B-D-GV} and \cite{G-SR_Mem} as
important contributions.
%: among others, we are going to mention in detail
%below, we quote here papers \cite{G-SR_Mem}, \cite{B-D-GV} and \cite{B-D_2003}.
%In particular, let us point out that \cite{B-D_2003} deals with a $2$-D shallow water model, in the case of
%well-prepared initial data and in periodic domains: we will give more details about it later on.

In the recent paper \cite{F-G-N}, Feireisl, Gallagher and Novotn\'y studied the incompressible and high rotation limits together,
for the $3$-D compressible barotropic Navier-Stokes system with Coriolis force,
in the general instance of ill-prepared initial data. Their asymptotic result relies on the spectral analysis of the singular perturbation operator:
by use of the celebrated RAGE theorem (see e.g. books \cite{R-S_III} and \cite{Cyc-Fr-K-S}), they were able to prove some dispersion
properties due to fast rotation, from which they deduced strong convergence of the velocity fields, and this allowed them to pass to
the limit in the  weak formulation of the system.
In paper \cite{F-G-GV-N} by Feireisl, Gallagher, G\'erard-Varet and Novotn\'y, the effect of the centrifugal force was added to the previous
system. Notice that this term scales as $1/\veps^2$; hence, they got interested in both
the isotropic and multi-scale limit: namely, the Mach number was supposed to be proportional to $\veps^m$, for $m=1$ in the former
instance (as in \cite{F-G-N}), $m>1$ in the latter. Let us just point out that, in the analysis
of the isotropic scaling (i.e. $m=1$), they had to resort
to compensated compactness arguments (used for the first time in \cite{G-SR_2006} in the context of rotating fluids)
in order to pass to the limit in the weak formulation: as a matter of fact, the singular perturbation operator had variable coefficients,
and spectral analysis tools were no more available.
In both previous works \cite{F-G-N} and \cite{F-G-GV-N}, it is proved that the limit system is a $2$-D viscous
quasi-geostrophic equation for the limit density (or, better, for the limit $r$ of the quantities
$r_\veps=\veps^{-1}\left(\rho_\veps-1\right)$), which can be interpreted as a sort of stream-function for the limit divergence-free
velocity field. We remark also that the authors were able to establish stability (and then uniqueness) for the limit equation under
an additional regularity hypothesis on the limit-point of the initial velocity fields.

The fact that the limit equation is two-dimensional is a common feature in the context of
fast rotating fluids (see for instance book \cite{C-D-G-G}):
indeed, it is the expression of a well-known physical phenomenon, the Taylor-Proudman theorem (see e.g. \cite{Ped}).
Namely, in the asymptotic regime, the high rotation tends to stabilize the motion, which becomes constant in the direction
parallel to the rotation axis: the fluid moves along vertical coloumns
(the so called ``Taylor-Proudman coloumns''), and the flow is purely horizontal.

\medbreak
Let us come back now to our problem for the Navier-Stokes-Korteweg system \eqref{intro_eq:NSK+rot}.

We point out that the general Navier-Stokes-Korteweg system, that we introduced at the beginning, has been widely studied
(mostly with no Coriolis force), under various choices of the functions $\mu(\rho)$ and $\sigma(\rho)$: one can refer e.g. to papers
\cite{Carles-Danchin-Saut}, \cite{J} and \cite{Li-Marc}.
In fact, this gives rise to many different models, which are relevant, for instance, in the context of quantum hydrodynamics.

For the previous system supplemented with our special assumptions and with no rotation term, in \cite{B-D-L} Bresch,
Desjardins and Lin proved the existence of global in time ``weak'' solutions. Actually, they had to resort to a modified notion
of weak solution: indeed, any information on the velocity field $u$ and its gradient is lost when the density
vanishes; on the other hand, lower bounds for $\rho$ seem not to be available in the
context of weak solutions. This makes it impossible to pass to the limit in the non-linear terms when constructing  a solution
to  the system. The authors overcame such an obstruction choosing test functions whose support is concentrated on the set of
positive density: namely, one has to evaluate the momentum equation not on a classical test function $\vphi$, but rather
on $\rho\,\vphi$, and this leads to a slightly different weak formulation of the system (see also Definition \ref{d:weak} below).
This modified formulation is made possible exploiting additional regularity for $\rho$, which is provided
by the capillarity term: in fact, a fundamental issue in the analysis of \cite{B-D-L} was the proof of the conservation
of a second energy (besides the classical one) for this system, the so called \emph{BD entropy},
which allows to control higher order space derivatives of the density term.

Still using this special energy conservation, in \cite{B-D_2003} Bresch and Desjardins were able to prove existence of global in time
weak solutions (in the classical sense) for a $2$-D viscous shallow water model in a periodic box
(we refer also to \cite{B-D_2006} for the explicit construction of the sequence of approximate solutions).
The system they considered there is very similar to the previous one, but it presents two
additional friction terms: a laminar friction and a turbolent friction. The latter plays a similar role to the capillarity
(one does not need both to prove compactness properties for the sequence of smooth solutions), while the former gives integrability
properties on the velocity field $u$: this is why they did not need to deal with the modified notion of weak solutions.
In the same work \cite{B-D_2003}, the authors were able to consider also the fast rotation limit in the instance of well-prepared initial
data, and to prove the convergence to the viscous quasi-geostrophic equations (as mentioned above for papers
\cite{F-G-N} and \cite{F-G-GV-N}). Let us point out that their argument in passing to the limit relies on the use of the modulated
energy method.

The recent paper \cite{J-L-W}, by J\"ungel, Lin and Wu, deals with a very similar problem: namely, incompressible and high rotation limit
in the two dimensional torus $\mbb{T}^2$ for well-prepared initial data, but combined also with a vanishing capillarity limit
(more specifically, like in system \eqref{intro_eq:NSK+rot}, with $0<\alpha<1$).
On the one hand, the authors were able to treat more general forms of the Navier-Stokes-Korteweg system, with different
functions $\mu(\rho)$ and $\sigma(\rho)$; on the other hand, for doing this they had to work in the framework of
(local in time) strong solutions. Again by use of the modulated energy method, they proved the convergence of the previous system to the
viscous quasi-geostrophic equation: as a matter of fact, due to the vanishing capillarity regime, no surface tension terms enter
into the singular perturbation operator, and the limit system is the same as in works \cite{B-D_2003}, \cite{F-G-N} and \cite{F-G-GV-N}.

\medbreak
In the present paper, we consider system \eqref{intro_eq:NSK+rot} in the infinite slab $\Omega\,=\,\R^2\times\,]0,1[\,$,
supplemented with complete slip boundary conditions, and with ill-prepared initial data.
Our goal is to study the asymptotic behavior of weak solutions in the regime of low Mach number and low Rossby number, possibly combining
these effects with the vanishing capillarity limit.

We stress here the following facts. First of all, we do not deal with general viscosity and surface tension functions:
we fix both $\mu(\rho)$ and $\sigma(\rho)$ as specified above. Second point: we consider a $3$-D domain, but we impose
complete slip boundary conditions, in order to avoid boundary layers effects, which we will not treat in this paper. Finally, our analysis relies on the techniques
used in \cite{F-G-N}, and this allows us to deal with general ill-prepared initial data, i.e. initial densities
$\bigl(\rho_{0,\veps}\bigr)_\veps$ (of the form $\rho_{0,\veps}=1+\veps r_{0,\veps}$) and initial velocities $\bigl(u_{0,\veps}\bigr)_\veps$,
both bounded in suitable spaces, which do not necessarily belong to the kernel of the singular perturbation operator.

For any fixed $\veps\in\,]0,1]$, the existence of global in time ``weak'' solutions to our system can be proved in the same way
as in \cite{B-D-L}  (see the discussion in Subsection \ref{ss:weak} for more
details). As a matter of fact, energy methods still work, due to the skew-symmetry of the Coriolis operator; moreover, a
control of the rotation term (but not uniformly in $\veps$) is easy to get in the BD entropy estimates: this guarantees additional
regularity for the density, and the possibility to prove convergence of the sequence of smooth solutions to a weak one.
However, a uniform control on the higher order derivatives of the densities is fundamental in our study,
in order to pass to the limit for $\veps\ra0$. It can be obtained (see Subsection \ref{ss:unif-bounds})
arguing like in the proof of the BD entropy estimates, but showing a uniform bound for the rotation term. This is the first
delicate point of our analysis: the problem comes from the fact that we have no control on the velocity fields $u_\veps$.
%Let us notice that this was not the case, for instance, in the work \cite{B-D_2003}: there, the presence of the laminar friction
%immediately ensured the uniform control of the rotation term.
For the same reason, we resort to the notion of ``weak'' solution developed by Bresch, Desjardins and Lin (see Definition \ref{d:weak});
notice that this leads to handle new non-linear terms arising in the modified weak formulation of the momentum equation.
%At this point we remark that, even having small initial data and additional regularity (uniformly in $\veps$) for the density functions,
%we are not able to extract from this any integrability property on the $u_\veps$'s. This is why
%we need to resort to the notion of ``weak'' solution developed by Bresch, Desjardins and Lin, and
%to test the momentum equations on $\rho_\veps\,\psi$, with $\psi\in\mc{D}\bigl([0,T]\times\Omega\bigr)$ (see Definition \ref{d:weak}).
%Notice that new non-linear terms arise from the modified weak formulation of the momentum equation.

Let us spend now a few words on the limit system (see the explicit expression in Theorems \ref{th:sing-lim} and \ref{th:alpha=0}),
which will be studied in Subsection \ref{ss:constraint}.
Both for vanishing and constant capillarity regimes, we still find that the limit
velocity field $u$ is divergence-free and horizontal and depends just on the horizontal variables
(in accordance with the Taylor-Proudman theorem); moreover, a relation links $u$ to the limit density $r$, which can still be seen as
a sort of stream-function for $u$. In the instance of constant capillarity, this relation slightly changes, giving rise to a
more complicated equation for $r$ (compare equations \eqref{eq:q-geo_1} and \eqref{eq:q-geo_0} below):
indeed, a surface tension term enters into play in the singular perturbation operator, and hence in the limit equation.

We first focus on the vanishing capillarity case, and more precisely on the choice $\alpha=1$ (treated in Section \ref{s:low-cap}).
Formally, the situation is analogous to the one of paper
\cite{F-G-N}, and the analysis of Feireisl, Gallagher and Novotn\'y still applies, after handling some technical
points in order to adapt their arguments to the modified weak formulation.
The main issue is the analysis of acoustic waves: also in this case, we can apply the RAGE theorem to prove
dispersion of the components which are in the subspace orthogonal to the kernel of the singular perturbation operator.
This is the key point in order to pass to the limit in the non-linear terms, and to get convergence. In the end we find
(as in \cite{F-G-N}) that the weak solutions of our system tend to a weak solution of a $2$-D quasi-geostrophic equation
(see \eqref{eq:q-geo_1} below). %for which we still have the uniqueness criterion established in \cite{F-G-N} and \cite{F-G-GV-N}.

Let us now turn our attention to the case of constant capillarity, i.e. $\alpha=0$ (see Section \ref{s:alpha=0}). Here,
the capillarity term scales as $1/\veps^2$, so it is of the same order as the pressure and the rotation operator. As a consequence,
the singular perturbation operator (say) $\mc{A}_0$ presents an additional term, and it is no more
skew-adjoint with respect to the usual $L^2$ scalar product. However, on the one hand $\mc{A}_0$ has still constant coefficients,
so that spectral analysis is well-adapted: direct computations show that the point spectrum still coincides with the kernel of
the wave propagator, as in the previous case $\alpha=1$. On the other hand, passing in Fourier variables it is easy to find a microlocal
symmetrizer (in the sense of M\'etivier, see \cite{M-2008}) for our system, i.e. a scalar product with respect to
which $\mc{A}_0$ is still skew-adjoint: this allows us to apply again the RAGE theorem with respect
to the new scalar product, and to recover the convergence result by the same techniques as above.
We remark that we use here the additional regularity for the density (the new inner product involves two space derivatives
for $r$).
We also point out that, for $\alpha=0$, the limit equation becomes \eqref{eq:q-geo_0}, which presents a similar structure
to the one of \eqref{eq:q-geo_1}, but where new terms appear.

The case $0<\alpha<1$ (see Section \ref{s:general_a}) is technically more complicated, because this choice
introduces an anisotropy of scaling in the system for acoustic waves. We will need to treat this anisotropy
as a perturbation term in the acoustic propagator: also in this case we can resort to spectral analysis methods, and we can symmetrize
our system, but now both the acoustic propagator and the microlocal symmetrizer depend on $\veps$, via the perturbation term.
So, we first need to prove a RAGE-type theorem (see Theorem \ref{th:RAGE_eps}) for families of operators and symmetrizers:
the main efforts in the analysis are devoted to this. Then, the proof of the convergence can be performed exactly as in the previous
cases: again, the limit system is identified as the $2$-D quasi-geostrophic equation \eqref{eq:q-geo_1}.

\medbreak
Before going on, let us present the organization of the paper.

In the next section we collect our assumptions, and we state our main results. Besides, we recall 
the definition of weak solution we adopt in the sequel and we spend a few words about the existence theory.
Some requirements imposed in this definition are justified by a priori bounds, which we show in Section \ref{s:bounds}; there,
we also identify the weak limits $u$ and $r$ and establish some of their properties.
Section \ref{s:low-cap} is devoted to the proof of the result for $\alpha=1$, while in Section \ref{s:alpha=0}
we deal with the case $\alpha=0$. The anisotropic scaling $0<\alpha<1$ is treated in Section \ref{s:general_a}.
Finally, we collect in the appendix some auxiliary results from Littlewood-Paley theory.

%%%%%%%%%%%%%%%%%%%%%%%%%%%%%%%%%%%%%%%%%%
%%%%%%%%%%%%%%%%%%%%%%%%%%%%%%%%%%%%
\subsubsection*{Acknowledgements}
%%%%%%%%%%%%%%%%%%%%%%%%%%%%%%%%%%%%%
%%%%%%%%%%%%%%%%%%%%%%%%%%%%%%%%%%

The author is deeply grateful to I. Gallagher for proposing him the problem and for enlightening discussions about it.
The most of the work was completed while he was a post-doc at Institut de Math\'ematiques de Jussieu-Paris Rive Gauche,
Universit\'e Paris-Diderot: he wish to express his gratitude also to these institutions.

The author wish to thanks also the anonimous referees for their careful reading and relevant remarks, which greatly helped him to improve
the final version of the paper.

The author was partially supported by the project ``Instabilities in Hydrodynamics'', funded by the Paris city
hall (program ``\'Emergence'') and the Fondation Sciences Math\'ematiques de Paris.
He is also member of the Gruppo Nazionale per l'Analisi Matematica, la Probabilit\`a
e le loro Applicazioni (GNAMPA) of the Istituto Nazionale di Alta Matematica (INdAM).

%%%%%%%%%%%%%%%%%%%%%%%%%%%%%%%%%%%%%%%%%%%%%%%%%%%%%%%%%%%%%%%%%%%%%%%%%%%%%%%%%%%%%%%%%%%%5
%%%%%%%%%%%%%%%%%%%%%%%%%%%%%%%%%%%%%%%%%%%%%%%%%%%%%%%%%%%%%%%%%%%%%%%%%%%%%%%%%%%%%%%%%%%%5
\section{Main hypotheses and results} \label{s:result}
%%%%%%%%%%%%%%%%%%%%%%%%%%%%%%%%%%%%%%%%%%%%%%%%%%%%%%%%%%%%%%%%%%%%%%%%%%%%%%%%%%%%%%%%%%%%5
%%%%%%%%%%%%%%%%%%%%%%%%%%%%%%%%%%%%%%%%%%%%%%%%%%%%%%%%%%%%%%%%%%%%%%%%%%%%%%%%%%%%%%%%%%%%5

This section is devoted to present our main results. First of all, we collect our working assumptions. Then, we give the definition
of weak solutions we will adopt throughout all the paper, and which is based on the one of \cite{B-D-L}.
Finally, we state our main theorems about the asymptotic limit.

%%%%%%%%%%%%%%%%%%%%%%%%%%%%%%%%%%%%%%%%%%%%%%
\subsection{Working setting and main assumptions} \label{ss:hypotheses}
%%%%%%%%%%%%%%%%%%%%%%%%%%%%%%%%%%%%%%%%%%%%%%

We fix the infinite slab
$$
\Omega\,=\,\R^2\,\times\,\,]0,1[
$$
and we consider in $\R_+\times\Omega$ the scaled Navier-Stokes-Korteweg system
\begin{equation} \label{eq:NSK+rot}
\begin{cases}
\d_t\rho+\div\left(\rho u\right)\,=\,0 \\[1ex]
\d_t\left(\rho u\right)+\div\bigl(\rho u\otimes u\bigr)+\dfrac{1}{\veps^2}\,\nabla P(\rho)+
\dfrac{1}{\veps}\,e^3\times\rho u-\nu\div\bigl(\rho Du\bigr)-
\dfrac{1}{\veps^{2(1-\alpha)}}\,\rho \nabla\Delta\rho\,=\,0\,,
\end{cases}
\end{equation}
where $\nu>0$ denotes the viscosity of the fluid, $D$ is the viscous stress tensor defined by
%\begin{equation} \label{eq:stress}
$$
Du\,:=\,\frac{1}{2}\,\bigl(\nabla u\,+\,^t\nabla u\bigr)\,,
$$
%\end{equation}
$e^3=(0,0,1)$ is the unit vector directed along the $x^3$-coordinate, and $0\leq\alpha\leq1$ is a fixed parameter.
Taking different values of $\alpha$, we are interested in performing a low capillarity limit (for $0<\alpha\leq1$),
with capillarity coefficient proportional to $\veps^{2\alpha}$, or in leaving the capillarity constant
(i.e. choosing $\alpha=0$).

We supplement system \eqref{eq:NSK+rot} by complete slip boundary conditions for $u$ and Neumann boundary conditions for $\rho$: this allows us to avoid
boundary layers effects, which go beyond the scopes of the present paper and will not be discussed here. If we denote by $n$
the unitary outward normal to the boundary $\d\Omega$ of the domain (simply, $\d\Omega=\{x^3=0\}\cup\{x^3=1\}$), we impose
\begin{equation} \label{eq:bc}
\left(u\cdot n\right)_{|\d\Omega}\,=\,u^3_{|\d\Omega}\,=\,0\,,\qquad
\left(\nabla\rho\cdot n\right)_{|\d\Omega}\,=\,\d_3\rho_{|\d\Omega}\,=\,0\,,\qquad
\bigl((Du)n\times n\bigr)_{|\d\Omega}\,=\,0\,.
\end{equation}

In the previous system \eqref{eq:NSK+rot}, the scalar function $\rho\geq0$ represents the density of
the fluid, $u\in\R^3$ its velocity field, and $P(\rho)$ its pressure, given by the $\g$-law
\begin{equation} \label{eq:def_P}
P(\rho)\,:=\,\frac{1}{\g}\,\rho^\g\,,\qquad\qquad\mbox{ for some }\quad1\,<\,\g\,\leq\,2\,.
\end{equation}
The requirement $\g\leq2$ is motivated by the fact that we need, roughly speaking, $\rho-1$ of finite energy, both for the existence theory and the asymptotic limit.
In fact, by point (i) of Lemma \ref{l:density} in the Appendix, when $\g>2$ we miss the property $\rho-1\,\in\,L^\infty_T(L^2)$ (recall that
the pressure term gives informations on the low frequencies of the density). 
About the existence theory, notice that no problems of this type appeared in \cite{B-D-L}, since the considered domain was bounded; concerning the singular limit analysis,
we refer to estimate \eqref{est:rho_L^inf-L^2} and Remark \ref{r:rho_L^2}.

\begin{rem} \label{r:period-bc}
Let us point out here that equations \eqref{eq:NSK+rot}, supplemented by boundary conditions \eqref{eq:bc},
can be recast as a periodic problem with respect to the vertical variable, in the new domain
$$
\Omega\,=\,\R^2\,\times\,\mbb{T}^1\,,\qquad\qquad\mbox{ with }\qquad\mbb{T}^1\,:=\,[-1,1]/\sim\,,
$$
where $\sim$ denotes the equivalence relation which identifies $-1$ and $1$. As a matter of fact,
it is enough to extend $\rho$ and $u^h$ as even functions with respect to $x^3$, and
$u^3$ as an odd function: the equations are invariant under such a transformation.

In what follows, we will always assume that such modifications have been performed on the initial data, and
that the respective solutions keep the same symmetry properties.
\end{rem}

Here we will consider the general instance of \emph{ill-prepared} initial data
$\bigl(\rho,u\bigr)_{|t=0}=\bigl(\rho_{0,\veps},u_{0,\veps}\bigr)$. Namely, we will suppose
the following on the family $\bigl(\rho_{0,\veps}\,,\,u_{0,\veps}\bigr)_{\veps>0}$:
\begin{itemize}
\item[(i)] $\rho_{0,\veps}\,=\,1\,+\,\veps\,r_{0,\veps}$, with
$\bigl(r_{0,\veps}\bigr)_\veps\,\subset\,H^1(\Omega)\cap L^\infty(\Omega)$ bounded;
\item[(ii)] $\bigl(u_{0,\veps}\bigr)_\veps\,\subset\,L^2(\Omega)$ bounded.
\end{itemize}
Up to extraction of a subsequence, we can suppose that
\begin{equation} \label{eq:conv-initial}
r_{0,\veps}\,\rightharpoonup\,r_0\quad\mbox{ in }\;H^1(\Omega)\qquad\qquad\mbox{ and }\qquad\qquad
u_{0,\veps}\,\rightharpoonup\,u_0\quad\mbox{ in }\;L^2(\Omega)\,,
\end{equation}
where we denoted by $\rightharpoonup$ the weak convergence in the respective spaces.

%%%%%%%%%%%%%%%%%%%%%%%%%%%%%%%%%%%%%%%%%%%%%%%%%%%%%%%%%%%
\subsection{Weak solutions} \label{ss:weak}
%%%%%%%%%%%%%%%%%%%%%%%%%%%%%%%%%%%%%%%%%%%%%%%%%%%%%%%%%%%%

In the present paragraph, we define the notion of weak solution for our system: it is based on the one given in \cite{B-D-L}. Essentially,
we need to localize the equations on sets where $\rho$ is positive (see also the discussion in the Introduction):  this is achieved, formally,
by testing the momentum equation on functions of the form $\rho\psi$, where $\psi\in\mc{D}$ is a classical test function.

First of all, we introduce the internal energy, i.e. the scalar function $h=h(\rho)$ such that
$$
h''(\rho)\,=\,\frac{P'(\rho)}{\rho}\,=\,\rho^{\g-2}\qquad\qquad\mbox{ and }\qquad\qquad
h(1)\,=\,h'(1)\,=\,0\,;
$$
let us define then the energies
\begin{eqnarray}
E_\veps[\rho,u](t) & := & \int_\Omega\left(\dfrac{1}{\veps^2}\,h(\rho)\,+\,\dfrac{1}{2}\,\rho\,|u|^2\,+\,
\dfrac{1}{2\,\veps^{2(1-\alpha)}}\,|\nabla\rho|^2\right)dx \label{eq:def_E} \\
F_\veps[\rho](t) & := & 
%\int_\Omega\left(\dfrac{1}{\veps^2}\,h(\rho)\,+\,\dfrac{1}{2}\,\rho\,|u+\nu\nabla\log\rho|^2
%\,+\,\dfrac{1}{2\,\veps^2}\,|\nabla\rho|^2\right)dx\,.
\frac{\nu^2}{2}\int_\Omega\rho\,|\nabla\log\rho|^2\,dx\;=\;
2\,\nu^2\int_\Omega\left|\nabla\sqrt{\rho}\right|^2\,dx\,.\label{eq:def_F}
\end{eqnarray}
We will denote by $E_\veps[\rho_0,u_0]\,\equiv\,E_\veps[\rho,u](0)$ and by
$F_\veps[\rho_0]\,\equiv\,F_\veps[\rho](0)$ the same quantities, when computed on the initial data
$\bigl(\rho_0,u_0\bigr)$.

Here we present the definition. The integrability properties required in points (i)-(ii)-(v)
will be justified by the computations of Subsections \ref{ss:energies} and \ref{ss:unif-bounds}.

\begin{defin} \label{d:weak}
Fix $(\rho_0,u_0)$ such that $E_\veps[\rho_0,u_0]+F_\veps[\rho_{0}]\,<\,+\infty$.

We say that $\bigl(\rho,u\bigr)$ is a \emph{weak solution} to system \eqref{eq:NSK+rot}-\eqref{eq:bc}
in $[0,T[\,\times\Omega$ (for some $T>0$) with initial data $(\rho_0,u_0)$ if the following conditions are fulfilled:
\begin{itemize}
 \item[(i)] $\rho\geq0$ almost everywhere, $\rho-1\in L^\infty\bigl([0,T[\,;L^\g(\Omega)\bigr)$, $\nabla\rho$ and
$\nabla\sqrt{\rho}\;\in L^\infty\bigl([0,T[\,;L^2(\Omega)\bigr)$ and $\nabla^2\rho\in L^2\bigl([0,T[\,;L^2(\Omega)\bigr)$;
\item[(ii)] $\sqrt{\rho}\,u\,\in L^\infty\bigl([0,T[\,;L^2(\Omega)\bigr)$ and $\sqrt{\rho}\,Du\,\in L^2\bigl([0,T[\,;L^2(\Omega)\bigr)$;
\item[(iii)] the mass equation is satisfied in the weak sense: for any $\phi\in\mc{D}\bigl([0,T[\,\times\Omega\bigr)$ one has
\begin{equation} \label{eq:weak-mass}
-\int^T_0\int_\Omega\biggl(\rho\,\d_t\phi\,+\,\rho\,u\,\cdot\,\nabla\phi\biggr)\,dx\,dt\,=\,\int_\Omega\rho_0\,\phi(0)\,dx\,;
\end{equation}
\item[(iv)] the momentum equation is verified in the following sense: for any $\psi\in\mc{D}\bigl([0,T[\times\Omega\bigr)$ one has
\begin{eqnarray}
& & \hspace{-0.5cm}
\int_\Omega\rho_0^2u_0\cdot\psi(0)\,dx\,=\,\int^T_0\!\!\int_\Omega\biggl(-\rho^2u\cdot\d_t\psi\,-\,\rho u\otimes\rho u:\nabla\psi\,+\,
\rho^2\left(u\cdot\psi\right)\div u\,- \label{eq:weak-momentum} \\
& & \qquad -\,\frac{\g}{\veps^2(\g+1)}\,P(\rho)\rho\,\div\psi\,+\,\frac{1}{\veps}\,e^3\times\rho^2u\cdot\psi\,+\,
\nu\rho Du:\rho\nabla\psi\,+ \nonumber \\
& & \qquad\qquad
\,+\,\nu\rho Du:\left(\psi\otimes\nabla\rho\right)\,+\,\frac{1}{\veps^{2(1-\alpha)}}\,\rho^2\Delta\rho\,\div\psi\,+\,
\frac{2}{\veps^{2(1-\alpha)}}\,\rho\Delta\rho\nabla\rho\cdot\psi\biggr)\,dx\,dt\,; \nonumber
\end{eqnarray}
\item[(v)] for almost every $t\in\,]0,T[\,$, the following energy inequalities hold true:
\begin{eqnarray*} 
E_\veps[\rho,u](t)\,+\,\nu\int^t_0\int_\Omega\rho\,\left|Du\right|^2\,dx\,d\tau & \leq & E_\veps[\rho_0,u_0] \\ %\label{est:en-ineq} \\
F_\veps[\rho](t)\,+\,\dfrac{\nu}{\veps^2}\int^t_0\int_\Omega P'(\rho)\,|\nabla\sqrt{\rho}|^2\,dx\,d\tau\,+\,
\dfrac{\nu}{\veps^{2(1-\alpha)}}\int^t_0\int_\Omega\left|\nabla^2\rho\right|^2\,dx\,d\tau & \leq & C\,(1+T)\,,
\end{eqnarray*}
for some constant $C$ depending just on $\bigl(E_\veps[\rho_{0},u_0],F_\veps[\rho_{0}],\nu\bigr)$.
\end{itemize}
\end{defin}

\begin{rem} \label{r:en-bounds}
Notice that, under our hypotheses (recall points (i)-(ii) in Subsection \ref{ss:hypotheses}),
the energies of the initial data are uniformly bounded with respect to $\veps$:
$$
E_\veps[\rho_{0,\veps},u_{0,\veps}]\,+\,F_\veps[\rho_{0,\veps}]\,\leq\,K_0\,,
$$
where the constant $K_0>0$ is independent of $\veps$.
\end{rem}

Let us spend a few words on existence of weak solutions to our equations.

\paragraph{Existence of weak solutions.}

We sketch here the proof of existence of global in time weak solutions to our system, in the sense specified by Definition \ref{d:weak}.
The main point is the construction of a sequence of smooth approximate solutions, which respect the energy and BD entropy estimates we will establish in
Propositions \ref{p:E}, \ref{p:F} and \ref{p:F-unif}. This can be done arguing as in paper \cite{B-D_2006} by Bresch and Desjardins (see also \cite{Mu-Po-Za}
and \cite{Mu-Po-Za_2015}): namely we have to introduce regularizing terms, both for $u$ and $\rho$, in the momentum equations, depending on small parameters
that we will let vanish in a second moment.

Before presenting the details, let us remark that, in \cite{B-D_2006}, the construction is performed for a shallow water system with two additional drag terms:
a laminar friction and a turbulent friction. As remarked there, the former item is needed to have integrability properties for the velocity field,
while the latter plays a role analogous to capillarity in stability analysis, and one do not need both of them to recover existence of weak solutions.
In any case, precisely due to the laminar friction (which could be also replaced by a so-called ``cold pressure'' term), one can get back to the classical
notion of weak solutions; on the contrary, in our case, since we miss integrability properties
for $u$, we have to resort to the modified notion of Definition \ref{d:weak}, the problem relying in proving compactness of the sequence of approximate solutions.

We also point out that this construction works in the simplest case when the density dependent viscosity coefficient and the surface tension term are
taken linear in $\rho$. For the general equations written at the beginning of the Introduction, the explicit construction of smooth
approximate solutions, which respect the BD entropy structure of the system, is much more complicated, and deep advances have been obtained just recently.
In the matter of this, we refer e.g. to works \cite{B-D-Zat}, \cite{Li-Xin}, \cite{V-Yu} and \cite{V-Yu_2}.

\medbreak
Let us come back to the explicit construction. The parameter $\veps>0$ is kept fixed at this level: for any $\veps$, we want to construct a weak solution
$(\rho_\veps,u_\veps)$ to system \eqref{eq:NSK+rot}. Let us take two additional small parameters $\delta>0$ and $\eta>0$, and let us consider the system
\begin{equation} \label{eq:NSK_app}
\begin{cases}
\d_t\rho+\div\left(\rho u\right)\,=\,0 \\[1ex]
\d_t\left(\rho u\right)+\div\bigl(\rho u\otimes u\bigr)+\dfrac{1}{\veps^2}\,\nabla P(\rho)+
\dfrac{1}{\veps}\,e^3\times\rho u- \\[1ex]
\qquad\qquad\qquad
-\nu\div\bigl(\rho Du\bigr)-\dfrac{1}{\veps^{2(1-\alpha)}}\,\rho \nabla\Delta\rho\,=\,
\delta\rho\nabla\Delta^{2s+1}\rho-\delta\nabla\pi(\rho)-\eta\Delta^2u\,.
\end{cases}
\end{equation}
The hypercapillarity $\rho\nabla\Delta^{2s+1}\rho$ (for some $s>2$ large enough) is needed to smooth out the density function; the artificial ``cold pressure''
$\pi(\rho)=-\rho^{-3}$ serves to keep it bounded away from $0$. Finally, the hyperviscosity term $\Delta^2u$ helps to regularize the velocity field.

Therefore, for any fixed positive $\delta$ and $\eta$, system \eqref{eq:NSK_app} becomes parabolic in $u$, and then it admits a unique global solution
$(\rho_{\veps,\delta,\eta}\,,\,u_{\veps,\delta,\eta})$.
Furthermore, as observed in \cite{B-D_2006}, the new terms are compatible with the mathematical structure of the original equations, so that one can find uniform
(with respect to $(\delta,\eta)$ parameters) energy and BD entropy estimates. For the explicit estimates and a proof of them, we refer to Subsections
\ref{ss:energies} and \ref{ss:unif-bounds}.

Now, the first step is to pass to the limit for $\eta\ra0$: this can be done as in Subsection 2.3 of \cite{B-D_2006}. Notice that we can still recover
compactness of the velocity fields thanks to the uniform (in $\eta$ but not in $\delta$) upper and lower bounds for $\rho$. In this way, we get a solution
$(\rho_{\veps,\delta}\,,\,u_{\veps,\delta})$ to equations \eqref{eq:NSK_app} supplemented with the choice $\eta=0$.
Then, the final step consists in passing to the limit also for $\delta\ra0$: as remarked above, we have to work with the modified weak formulation of
Definition \ref{d:weak}, for which the stability analysis performed in \cite{B-D-L} applies.
We refer to \cite{B-D_2006} and \cite{Mu-Po-Za} for more details.

\medbreak
In the end we have proved that, under the hypotheses fixed in Subsection \ref{ss:hypotheses}, for any initial datum $\bigl(\rho_{0,\veps},u_{0,\veps}\bigr)$
there exists a global weak solution $\bigl(\rho_\veps,u_\veps\bigr)$ to problem \eqref{eq:NSK+rot}-\eqref{eq:bc} in $\R_+\times\Omega$,
in the sense specified by Definition \ref{d:weak} above.

%%%%%%%%%%%%%%%%%%%%%%%%%%%%%%%%%%%%%%%%%%%%%%
\subsection{Main results} \label{ss:results}
%%%%%%%%%%%%%%%%%%%%%%%%%%%%%%%%%%%%%%%%%%%%%%

We are interested in studying the asymptotic behaviour of a family of weak solutions $\bigl(\rho_\veps\,,\,u_\veps\bigr)_\veps$
to system \eqref{eq:NSK+rot} for the parameter $\veps\ra0$.
As we will see in a while (see Theorems \ref{th:sing-lim} and \ref{th:alpha=0}), one of the main features is that the limit-flow
will be \emph{two-dimensional} and \emph{horizontal} along the plane orthogonal to the rotation axis.

Then, let us introduce some notations to describe better this phenomenon. We will always decompose $x\in\Omega$
into $x=(x^h,x^3)$, with $x^h\in\R^2$ denoting its horizontal component. Analogously,
for a vector-field $v=(v^1,v^2,v^3)\in\R^3$ we set $v^h=(v^1,v^2)$, and we define the differential operators
$\nabla_h$ and $\div_{\!h}$ as the usual operators, but acting just with respect to $x^h$.
Finally, we define the operator $\nabla^\perp_h\,:=\,\bigl(-\d_2\,,\,\d_1\bigr)$.

We can now state our main result in the vanishing capillarity case. The particular choice of the pressure law, i.e. $\g=2$,
will be commented in Remark \ref{r:rho_L^2}.
\begin{thm} \label{th:sing-lim}
Let us take $0<\alpha\leq1$ in \eqref{eq:NSK+rot} and $\g=2$ in \eqref{eq:def_P}.

Let $\bigl(\rho_\veps\,,\,u_\veps\bigr)_\veps$ be a family of weak solutions to system
\eqref{eq:NSK+rot}-\eqref{eq:bc} in $[0,T]\times\Omega$, in the sense of Definition \ref{d:weak}, related to initial data
$\bigl(\rho_{0,\veps},u_{0,\veps}\bigr)_\veps$ satisfying the hypotheses ${\rm (i)-(ii)}$ and \eqref{eq:conv-initial}, and the symmetry
assumptions of Remark \ref{r:period-bc}.
Let us define the scalar quantity $r_\veps\,:=\,\veps^{-1}\left(\rho_\veps-1\right)$.

Then, up to the extraction of a subsequence, one has the following properties:
\begin{itemize}
 \item[(a)] $r_\veps\,\rightharpoonup\,r$ in $L^\infty\bigl([0,T];L^2(\Omega)\bigr)\,\cap\,L^2\bigl([0,T];H^1(\Omega)\bigr)$;
 \item[(b)] $\sqrt{\rho_\veps}\,u_\veps\,\rightharpoonup\,u$ in $L^\infty\bigl([0,T];L^2(\Omega)\bigr)$ and
 $\sqrt{\rho_\veps}\,Du_\veps\,\rightharpoonup\,Du$ in $L^2\bigl([0,T];L^2(\Omega)\bigr)$;
 \item[(c)] $r_\veps\,\ra\,r$ and $\rho_\veps^{3/2}\,u_\veps\,\ra\,u$ (strong convergence) in $L^2\bigl([0,T];L^2_{loc}(\Omega)\bigr)$,
\end{itemize}
where $r=r(x^h)$ and $u=\bigl(u^h(x^h),0\bigr)$ are linked by the relation $u^h\,=\,\nabla^\perp_hr$. Moreover,
$r$ satisfies (in the weak sense) the quasi-geostrophic type equation
\begin{equation} \label{eq:q-geo_1}
\d_t\bigl(r\,-\,\Delta_hr\bigr)\,+\,\nabla^\perp_hr\,\cdot\,\nabla_h\Delta_hr\,+\,\frac{\nu}{2}\,\Delta_h^2r\,=\,0
\end{equation}
supplemented with the initial condition $r_{|t=0}\,=\,\oline{r}_0$, where $\oline{r}_0\,\in\,H^1(\R^2)$ is the unique solution of
$$
\bigl(\Id-\Delta_h\bigr)\,\oline{r}_0\,=\,\int_0^1\bigl(\omega^3_0\,+\,r_0\bigr)\,dx^3\,,
$$
with $r_0$ and $u_0$ which are defined in \eqref{eq:conv-initial}, $\omega_0\,=\,\nabla\times u_0$ which is the vorticity of $u_0$ and
$\omega_0^3$ its third component.
\end{thm}

\begin{rem} \label{r:uniqueness}
Let us point out that, the limit equation being the same, the uniqueness criterion given in Theorem 1.3 of \cite{F-G-N}
still holds true under our hypothesis.

Then, if $\omega^3_0\in L^2(\Omega)$, the solution $r$ to equation \eqref{eq:q-geo_1} is uniquely determined by the initial condition
in the subspace of distributions such that $\nabla_hr\in L^\infty\bigl(\R_+;H^1(\R^2)\bigr)\,\cap\,L^2\bigl(\R_+;\dot{H}^2(\R^2)\bigr)$,
and the whole sequence of weak solutions converges to it.
\end{rem}

Let us now turn our attention to the case $\alpha=0$, i.e. when the capillarity coefficient is taken to be constant.
\begin{thm} \label{th:alpha=0}
Let us take $\alpha=0$ in \eqref{eq:NSK+rot} and $1<\g\leq2$ in \eqref{eq:def_P}.

Let $\bigl(\rho_\veps\,,\,u_\veps\bigr)_\veps$ be a family of weak solutions to system
\eqref{eq:NSK+rot}-\eqref{eq:bc} in $[0,T]\times\Omega$, in the sense of Definition \ref{d:weak}, related to initial data
$\bigl(\rho_{0,\veps},u_{0,\veps}\bigr)_\veps$ satisfying the hypotheses ${\rm (i)-(ii)}$ and \eqref{eq:conv-initial}, and the symmetry
assumptions of Remark \ref{r:period-bc}..
We define $r_\veps\,:=\,\veps^{-1}\left(\rho_\veps-1\right)$, as before.

Then, up to the extraction of a subsequence, one has the convergence properties
\begin{itemize}
 \item[(a*)] $r_\veps\,\rightharpoonup\,r$ in $L^\infty\bigl([0,T];H^1(\Omega)\bigr)\,\cap\,L^2\bigl([0,T];H^2(\Omega)\bigr)$
\end{itemize}
and the same (b) and (c) stated in Theorem \ref{th:sing-lim}, where,
this time, $r=r(x^h)$ and $u=\bigl(u^h(x^h),0\bigr)$ are linked by the relation
$u^h\,=\,\nabla^\perp_h\left(\Id-\Delta_h\right)r$. Moreover,
$r$ solves (in the weak sense) the modified quasi-geostrophic equation
\begin{equation} \label{eq:q-geo_0}
\d_t\Bigl(r\,-\,\Delta_hr\,+\,\Delta_h^2r\Bigr)\,+\,\nabla^\perp_h\bigl(\Id-\Delta_h\bigr)r\,\cdot\,\nabla_h\Delta_h^2r\,+\,
\frac{\nu}{2}\,\Delta_h^2\bigl(\Id-\Delta_h\bigr)r\,=\,0
\end{equation}
supplemented with the initial condition $r_{|t=0}\,=\,\wtilde{r}_0$, where $\wtilde{r}_0\,\in\,H^3(\R^2)$ is the unique solution of
$$
\bigl(\Id-\Delta_h+\Delta_h^2\bigr)\,\wtilde{r}_0\,=\,\int_0^1\bigl(\omega^3_0\,+\,r_0\bigr)\,dx^3\,.
$$
%with $\omega_0\,=\,\nabla\times u_0$ the vorticity of the initial velocity field $u_0$.
\end{thm}

\begin{comment}
\begin{rem} \label{r:q-geo_better}
 Notice that, by orthogonality, equation \eqref{eq:q-geo_0} can be also rewritten as
$$
\d_t\Bigl(r-\Delta_hX(r)\Bigr)\,+\,\nabla^\perp_hX(r)\,\cdot\,\nabla_h\Bigl(r-\Delta_hX(r)\Bigr)\,-\,
\frac{\nu}{2}\,\Delta_h^2X(r)\,=\,0\,,
$$
where we have defined $X(r)\,:=\,\bigl(\Id-\Delta_h\bigr)r$ (which can be interpreted as a sort of stream-function for the limit flow $u$).
In this form, the analogy with equation \eqref{eq:q-geo_1} is clear.

%In an alternative form, setting $\Gamma\,:=\,\bigl(\Id-\Delta_h+\Delta_h^2\bigr)r\,=\,X+\Phi$, we can also write
% $$
%\d_t\Gamma\,+\,\nabla^\perp_hX\,\cdot\,\nabla_h\Gamma\,+\,\frac{\nu}{2}\,\Delta_h\Gamma\,=\,
%-\,\frac{\nu}{2}\Delta_hr\,=\,\frac{\nu}{2}\bigl(-\Delta_h\bigr)^{-1}\Phi\,.
% $$
\end{rem}
\end{comment}

%%%%%%%%%%%%%%%%%%%%%%%%%%%%%%%%%%%%%%%%%%%%%%%%%%%%%%%%%%%%%%%%%%%%%%%%%%%%%%%%%%%%%%%%%%%%%%
%%%%%%%%%%%%%%%%%%%%%%%%%%%%%%%%%%%%%%%%%%%%%%%%%%%%%%%%%%%%%%%%%%%%%%%%%%%%%%%%%%%%%%%%%%%%%%%%%
\section{Preliminaries and uniform bounds} \label{s:bounds}
%%%%%%%%%%%%%%%%%%%%%%%%%%%%%%%%%%%%%%%%%%%%%%%%%%%%%%%%%%%%%%%%%%%%%%%%%%%%%%%%%%%%%%%%%%%%%%%%%
%%%%%%%%%%%%%%%%%%%%%%%%%%%%%%%%%%%%%%%%%%%%%%%%%%%%%%%%%%%%%%%%%%%%%%%%%%%%%%%%%%%%%%%%%%%%%%%%%

The present section is devoted to stating the main properties of the family of weak solutions of Theorems \ref{th:sing-lim}
and \ref{th:alpha=0}.

First of all, we prove energy and BD entropy estimates for our system, uniformly with respect to the parameter $\veps$. This will justify
the properties required in Definition \ref{d:weak}.
From them, we will infer additional uniform bounds and further properties the family $\bigl(\rho_\veps,u_\veps\bigr)_\veps$ enjoys.
Finally, we will derive some constraints on its weak limit.

%%%%%%%%%%
\subsection{Energy and BD entropy estimates} \label{ss:energies}
%%%%%%%%%%%%

Suppose that $(\rho,u)$ is a smooth solution to system \eqref{eq:NSK+rot} in $[0,T[\,\times\Omega$ (for some
$T>0$), related to the smooth initial datum $\bigl(\rho_{0},u_{0}\bigr)$.

The first energy estimate, involving $E_\veps$, is obtained in a standard way.
\begin{prop} \label{p:E}
Let $(\rho,u)$ be a smooth solution to system \eqref{eq:NSK+rot} in $[0,T[\,\times\Omega$, with
initial datum $\bigl(\rho_0,u_0\bigr)$, for some positive time $T>0$.

Then, for all $\veps>0$ and all $t\in[0,T[\,$, one has
$$
\frac{d}{dt}E_\veps[\rho,u]\,+\,\nu\int_\Omega\rho\,|Du|^2\,dx\,=\,0\,.
$$
\end{prop}

\begin{proof}
First of all, we multiply the second relation in system \eqref{eq:NSK+rot} by $u$: by use of the mass equation and
due to the fact that $e^3\times\rho u$ is orthogonal to $u$, we arrive at the identity:
$$
\frac{1}{2}\,\frac{d}{dt}\int_\Omega\rho\,|u|^2\,dx\,+\,
\frac{1}{\veps^2}\int_\Omega P'(\rho)\,\nabla\rho\cdot u\,dx\,+\,
\nu\,\int_\Omega\rho\,Du\,:\,\nabla u\,dx\,+\,
\frac{1}{2\,\veps^{2(1-\alpha)}}\,\frac{d}{dt}\int_\Omega|\nabla\rho|^2\,dx\,=\,0\,.
$$
On the one hand, we have the identity $Du:\nabla u=|Du|^2$; on the other hand, multiplying the equation for
$\rho$ by $h'(\rho)/\veps^2$ gives
$$
\frac{1}{\veps^2}\int_\Omega P'(\rho)\,\nabla\rho\cdot u\,dx\,=\,
\frac{1}{\veps^2}\,\frac{d}{dt}\int_\Omega h(\rho)\,dx\,.
$$
Putting this relation into the previous one concludes the proof of the proposition.
\end{proof}

Let us now consider the function $F_\veps$: we have the following estimate.
\begin{prop} \label{p:F}
Let $(\rho,u)$ be a smooth solution to system \eqref{eq:NSK+rot} in $[0,T[\,\times\Omega$, with
initial datum $\bigl(\rho_0,u_0\bigr)$, for some positive time $T>0$.

Then there exists a ``universal'' constant $C>0$ such that, for all $t\in[0,T[\,$, one has
\begin{eqnarray}
& & \hspace{-0.2cm}
\frac{1}{2}\int_\Omega\rho(t)\,\left|u(t)\,+\,\nu\,\nabla\log\rho(t)\right|^2\,dx\,+\,
\frac{\nu}{\veps^{2(1-\alpha)}}\int^t_0\!\!\int_\Omega\left|\nabla^2\rho\right|^2\,dx\,d\tau\,+ \label{est:F} \\
& & \quad +\,\frac{4\nu}{\veps^2}\int^t_0\!\!\int_\Omega P'(\rho)\,\left|\nabla\sqrt{\rho}\right|^2\,dx\,d\tau
\,\leq\,C\bigl(F_\veps[\rho_0]+E_\veps[\rho_0,u_0]\bigr)\,+\,
\frac{\nu}{\veps}\left|\int^t_0\!\!\int_\Omega e^3\times u\cdot\nabla\rho\,dx\,d\tau\right|.  \nonumber
\end{eqnarray}
\end{prop}

\begin{proof}
We will argue as in Section 3 of \cite{B-D-L}. First, by Lemma 2 of that paper we have the identity
\begin{equation} \label{eq:lemma2}
\frac{1}{2}\,\frac{d}{dt}\int_\Omega\rho\,\left|\nabla\log\rho\right|^2\,+\,\int_\Omega\nabla\div u\cdot
\nabla\rho\,+\,\int_\Omega\rho\,Du:\nabla\log\rho\otimes\nabla\log\rho\,=\,0\,.
\end{equation}

Next, we multiply the momentum equation by $\nu\,\nabla\rho/\rho$ and we integrate over $\Omega$: we find
\begin{eqnarray*}
& & \hspace{-1.5cm}
\nu\int_\Omega\left(\d_tu+u\cdot\nabla u\right)\cdot\nabla\rho\,+\,
\nu^2\int_\Omega Du:\left(\nabla^2\rho-\frac{1}{\rho}\nabla\rho\otimes\nabla\rho\right)\,+ \\
& & \qquad\qquad
+\,\frac{\nu}{\veps}\int_\Omega e^3\times u\cdot\nabla\rho\,+\,
\frac{\nu}{\veps^{2(1-\alpha)}}\int_\Omega\left|\nabla^2\rho\right|^2\,+\,
\frac{4\,\nu}{\veps^2}\int_\Omega P'(\rho)\left|\nabla\sqrt{\rho}\right|^2\,=\,0\,.
\end{eqnarray*}

Now we add \eqref{eq:lemma2}, multiplied by $\nu^2$, to this last relation, getting
\begin{eqnarray*}
& & \hspace{-1cm}
\frac{\nu^2}{2}\,\frac{d}{dt}\int_\Omega\rho\left|\nabla\log\rho\right|^2\,+\,
\frac{\nu}{\veps^{2(1-\alpha)}}\int_\Omega\left|\nabla^2\rho\right|^2\,+\,
\frac{4\nu}{\veps^2}\int_\Omega P'(\rho)\left|\nabla\sqrt{\rho}\right|^2\,+\,
\frac{\nu}{\veps}\int_\Omega e^3\times u\cdot\nabla\rho \\
& & \qquad =\,-\,\nu\int_\Omega\d_tu\cdot\nabla\rho\,-\,\nu^2\int_\Omega\nabla\div u\cdot\nabla\rho\,-\,
\nu\int_\Omega\left(u\cdot\nabla u\right)\cdot\nabla\rho\,-\,\nu^2\int_\Omega Du:\nabla^2\rho\,.
\end{eqnarray*}
Using the mass equation and the identities
\begin{eqnarray*}
-\,\int_\Omega u\cdot\nabla\div(\rho u)\,-\,\int_\Omega\left(u\cdot\nabla u\right)\cdot\nabla\rho & = & 
\int_\Omega\rho\nabla u:\,^t\nabla u \\
-\,\int_\Omega\nabla\div u\cdot\nabla\rho\,-\,\int_\Omega Du:\nabla^2\rho & = & 0
\end{eqnarray*}
we end up with the equality
\begin{eqnarray*}
& & \hspace{-1.5cm}\frac{d}{dt}F_\veps\,+\,
\frac{\nu}{\veps^{2(1-\alpha)}}\int_\Omega\left|\nabla^2\rho\right|^2\,dx\,+\,
\frac{4\nu}{\veps^2}\int_\Omega P'(\rho)\,\left|\nabla\sqrt{\rho}\right|^2\,dx\,+ \\
& & \qquad\qquad +\,\frac{\nu}{\veps}\int_\Omega e^3\times u\cdot\nabla\rho\,dx\,=\,
-\,\nu\,\frac{d}{dt}\int_\Omega u\cdot\nabla\rho\,dx\,+\,\nu\int_\Omega\rho\,\nabla u\,:\,^t\nabla u\,dx\,.
\end{eqnarray*}
We notice that this relation can be rewritten in the following way:
\begin{eqnarray*}
& & \hspace{-1cm}
\frac{1}{2}\,\frac{d}{dt}\int_\Omega\rho\,\left|u\,+\,\nu\,\nabla\log\rho\right|^2\,dx\,+\,
\frac{\nu}{\veps^{2(1-\alpha)}}\int_\Omega\left|\nabla^2\rho\right|^2\,dx\,+\,
\frac{4\nu}{\veps^2}\int_\Omega P'(\rho)\,\left|\nabla\sqrt{\rho}\right|^2\,dx\,+ \\
& & \qquad\qquad\qquad\qquad +\,\frac{\nu}{\veps}\int_\Omega e^3\times u\cdot\nabla\rho\,dx\,=\,
\frac{1}{2}\,\frac{d}{dt}\int_\Omega \rho\,|u|^2\,dx\,+\,\nu\int_\Omega\rho\,\nabla u\,:\,^t\nabla u\,dx\,.
\end{eqnarray*}

Now we integrate with respect to time and we use Proposition \ref{p:E}.
\end{proof}

Observe that, writing $e^3\times u\cdot\nabla\rho\,=\,2\,e^3\times\left(\sqrt{\rho}u\right)\cdot\nabla\sqrt{\rho}$ and using Young's inequality
and Proposition \ref{p:E}, one can control the last term in \eqref{est:F} and bound the quantity
$$
F_\veps[\rho](t)\,+\,\frac{\nu}{\veps^{2(1-\alpha)}}\int^t_0\int_\Omega\left|\nabla^2\rho\right|^2\,dx\,d\tau\,+\,
\frac{\nu}{\veps^2}\int^t_0\int_\Omega P'(\rho)\,\left|\nabla\sqrt{\rho}\right|^2\,dx\,d\tau\,.
$$
Such a bound is enough to get additional regularity for the sequence of smooth approximate densities when constructing a weak solution,
but it is not uniform with respect to $\veps$: so it is not suitable to fully exploit the BD entropy structure of the system
in our study.

Nonetheless, in Proposition \ref{p:F-unif} below we are going to show that, under our assumptions, it is possible to control
the right-hand side of \eqref{est:F} in a uniform way with respect to $\veps$. This is a key point in order to prove
our results.

%%%%%%%%%%
\subsection{Uniform bounds} \label{ss:unif-bounds}
%%%%%%%%%%%%

Now, we are going to establish uniform properties the family $\bigl(\rho_\veps,u_\veps\bigr)$ satisfies.

First of all, by Proposition \ref{p:E} and Remark \ref{r:en-bounds}, we infer the following properties:
\begin{itemize}
 \item $\left(\veps^{-2}\,h(\rho_\veps)\right)_\veps\,\subset\,L^\infty\bigl(\R_+;L^1(\Omega)\bigr)$ is bounded;
 \item $\left(\sqrt{\rho_\veps}\,u_\veps\right)_\veps$ is bounded in $L^\infty\bigl(\R_+;L^2(\Omega)\bigr)$;
 \item $\left(\sqrt{\rho_\veps}\,Du_\veps\right)_\veps$ is a bounded subset of $L^2\bigl(\R_+;L^2(\Omega)\bigr)$;
 \item $\left(\nabla\rho_\veps\right)_\veps\,\subset\,L^\infty\bigl(\R_+;L^2(\Omega)\bigr)$, with
$$
\left\|\nabla\rho_\veps\right\|_{L^\infty(\R_+;L^2(\Omega))}\,\leq\,C\,\veps^{1-\alpha}\,,
$$
for some positive constant $C$.
\end{itemize}

Furthermore, arguing as in the proof to Lemma 2 of \cite{J-L-W}, from the uniform bound on the internal energy $h$ we infer the control
$$
\left\|\rho_\veps\,-\,1\right\|_{L^\infty(\R_+;L^\g(\Omega))}\,\leq\,C\,\veps\,.
$$
In particular, under our assumptions on $\alpha$ and $\g$, we always find (see point (i) of Lemma \ref{l:density})
\begin{equation} \label{est:rho_L^inf-L^2}
\left\|\rho_\veps\,-\,1\right\|_{L^\infty(\R_+;L^2(\Omega))}\,\leq\,C\,\veps\,.
\end{equation}

Now, as announced above, we are going to show how to derive, under our assumptions, BD entropy bounds which are uniform in $\veps$.
Of course, all the computations are justified for smooth functions: however, by a standard approximation procedure,
the final bounds will be fulfilled also by the family of weak solutions.
\begin{prop} \label{p:F-unif}
Let $\bigl(\rho_{0,\veps},u_{0,\veps}\bigr)_\veps$ be a family of initial data satisfying the assumptions (i)-(ii) of Subsection
\ref{ss:hypotheses}, and let $\bigl(\rho_\veps,u_\veps\bigr)_\veps$ be a family of corresponding weak solutions.

Then there exist an $\veps_0>0$ and a constant $C>0$ (depending just on %$E_\veps[\rho_{0,\veps},u_{0,\veps}]$, $F_\veps[\rho_{0,\veps}]$
the constant $K_0$ of Remark \ref{r:en-bounds} and on the viscosity coefficient $\nu$) such that the inequality
$$
F_\veps[\rho_\veps](t)\,+\,\dfrac{\nu}{\veps^2}\int^t_0\int_\Omega P'(\rho_\veps)\,|\nabla\sqrt{\rho_\veps}|^2\,dx\,d\tau\,+\,
\dfrac{\nu}{\veps^{2(1-\alpha)}}\int^t_0\int_\Omega\left|\nabla^2\rho_\veps\right|^2\,dx\,d\tau\,\leq\,C\,(1\,+\,t)
$$
holds true for any $t>0$ and for all $0<\veps\leq\veps_0$.
\end{prop}

\begin{proof}
Our starting point is the inequality stated in Proposition \ref{p:F}: we have to control the last term in its right-hand side. For
convenience, let us omit for a while the index $\veps$ in the notation.

First of all, we can write
\begin{eqnarray*}
\int_\Omega e^3\times u\cdot\nabla\rho & = & \int_\Omega\rho^{(\g-1)/2}\,e^3\times u\cdot\nabla\rho\,+\,
\int_\Omega\left(1-\rho^{(\g-1)/2}\right)\,e^3\times u\cdot\nabla\rho \\
& = & 2\int_\Omega e^3\times\left(\sqrt{\rho}\,u\right)\cdot\nabla\sqrt{\rho}\;\rho^{(\g-1)/2}\,+\,
\int_\Omega\left(1-\rho^{(\g-1)/2}\right)\,e^3\times u\cdot\nabla\rho\,.
\end{eqnarray*}
Now we focus on the last term: integrating by parts
%: since $e^3\times u$ is orthogonal to the outward normal $n$ at $\d\Omega$, boundary terms disappear. Hence
we get
$$
\int_\Omega\left(1-\rho^{(\g-1)/2}\right)\,e^3\times u\cdot\nabla\rho\,=\,\int_\Omega\rho\,\omega^3\,\left(1-\rho^{(\g-1)/2}\right)\,+\,
\frac{\g-1}{2}\int_\Omega\rho^{(\g-1)/2}\,e^3\times u\cdot\nabla\rho\,,
$$
where we denoted by $\omega\,=\,\nabla\times u$ the vorticity of the fluid. Therefore, in the end we find
$$
\int_\Omega e^3\times u\cdot\nabla\rho\,=\,
(\g+1)\int_\Omega e^3\times\left(\sqrt{\rho}\,u\right)\cdot\nabla\sqrt{\rho}\;\rho^{(\g-1)/2}\,+\,
\int_\Omega\rho\,\omega^3\,\left(1-\rho^{(\g-1)/2}\right)\,.
$$

Let us deal with the first term: we have
\begin{eqnarray}
\frac{\nu}{\veps}\left|\int^t_0\int_\Omega e^3\times\left(\sqrt{\rho}\,u\right)\cdot\nabla\sqrt{\rho}\;\rho^{(\g-1)/2}\right| & \leq &
\frac{\nu}{\veps}\int^t_0\left\|\sqrt{\rho}\,u\right\|_{L^2}\,\left\|\rho^{(\g-1)/2}\,\nabla\sqrt{\rho}\right\|_{L^2} \label{est:F_1} \\
& \leq & C\,\nu\,t\,+\frac{\nu}{2\,\veps^2}\int^t_0\left\|\rho^{(\g-1)/2}\,\nabla\sqrt{\rho}\right\|^2_{L^2}\,, \nonumber
\end{eqnarray}
where we have used also the uniform bounds for $\bigl(\sqrt{\rho_\veps}\,u_\veps\bigr)_\veps$ and Young's inequality.
Notice that, as $P'(\rho)=\rho^{\g-1}$, the last term can be absorbed in the left-hand side of \eqref{est:F}.

Now we consider the term involving the vorticity. Notice that, since $0<(\g-1)/2\leq1/2$, we can bound $|\rho^{(\g-1)/2}-1|$
with $|\rho-1|$; then, using also the established uniform bounds, we get
\begin{eqnarray*}
\frac{\nu}{\veps}\left|\int^t_0\int_\Omega\rho\,\omega^3\,\left(1-\rho^{(\g-1)/2}\right)\right| & \leq & 
\frac{\nu}{\veps}\int^t_0\|\rho-1\|_{L^2}\,
\left\|\sqrt{\rho}\,Du\right\|_{L^2}\,\|\rho\|^{1/2}_{L^\infty} \\
& \leq & C\,\nu\left(\int^t_0\|\rho\|_{L^\infty}\right)^{1/2}\,.
\end{eqnarray*}
In order to control the $L^\infty$ norm of the density, we write $\rho=1+(\rho-1)$: for the second term we use Lemma \ref{l:density}
with $p=2$ and $\delta=1/2$. Keeping in mind also estimate \eqref{est:rho_L^inf-L^2}, we get
\begin{eqnarray*}
C\,\nu\left(\int^t_0\|\rho\|_{L^\infty}\right)^{1/2} & \leq & 
C\,\nu\left(\int^t_0\left(1+\left\|\nabla^2\rho\right\|_{L^2}\right)\right)^{1/2} \\
& \leq & \frac{C\,\nu}{2}\,(1+t)\,+\,\frac{C\,\nu}{2}\,\int^t_0\left\|\nabla^2\rho\right\|_{L^2} \\
& \leq & C'\,\nu\,(1+t)\,+\,\frac{\nu}{4}\,\int^t_0\left\|\nabla^2\rho\right\|^2_{L^2}\,,
\end{eqnarray*}
where we used twice Young's inequality. Hence, in the end we obtain
\begin{equation} \label{est:F_2}
\frac{\nu}{\veps}\left|\int^t_0\int_\Omega\rho\,\omega^3\,\left(1-\rho^{(\g-1)/2}\right)\right|\,\leq\,C\,\nu\,(1+t)\,+\,
C_\veps\,\frac{\nu}{\veps^{2(1-\alpha)}}\,\left\|\nabla^2\rho\right\|^2_{L^2_t(L^2)}\,,
\end{equation}
with $C_\veps\,=\,\veps^{2(1-\alpha)}/4$.
Then, for any $\alpha\in[0,1]$ we can absorb the last term of this estimate into the left-hand side of \eqref{est:F}.

Therefore, thanks to inequalities \eqref{est:F_1} and \eqref{est:F_2}, combined with \eqref{est:F}, we get the result.
\end{proof}

%\begin{rem} \label{r:global_t}
%Notice that we cannot get global in time estimates in the previous proposition, essentially due to the presence of
%the term $C\,\nu\,t$ in estimates \eqref{est:F_1} and \eqref{est:F_2}.
%\end{rem}

\begin{rem} \label{r:rho_L^2}
The approach we followed seems to suggest that having $\|\rho_\veps-1\|_{L^\infty_T(L^2)}\,\sim\,O(\veps)$ is necessary to control
the term coming from rotation in \eqref{est:F} and so to close the estimates (see in particular the bounds for the term
involving vorticity).

This is the only (technical) reason for which we assumed $\g=2$ when $0<\alpha\leq1$ (low capillarity limit),
while for $\alpha=0$ (constant capillarity case) we can take more general pressure laws, namely any $1<\g\leq2$,
since we still have inequality \eqref{est:rho_L^inf-L^2}.
\end{rem}

By the bounds established in Proposition \ref{p:F-unif}, we infer also the following properties:
\begin{itemize}
\item $\left(\sqrt{\rho_\veps}\,\,\nabla\log\rho_\veps\right)_\veps$ is bounded in the space $L^\infty_{loc}\bigl(\R_+;L^2(\Omega)\bigr)$;
\item $\left(\veps^{-1}\,\rho_\veps^{(\g-1)/2}\,\nabla\sqrt\rho_\veps\right)_\veps\,\subset\,L^2_{loc}\bigl(\R_+;L^2(\Omega)\bigr)$ bounded;
\item $\left(\veps^{-(1-\alpha)}\,\nabla^2\rho_\veps\right)_\veps$ is bounded in $L^2_{loc}\bigl(\R_+;L^2(\Omega)\bigr)$.
\end{itemize}
In particular, from the last fact combined with estimate \eqref{est:rho_L^inf-L^2} and Lemma \ref{l:density}, we also gather that,
for any fixed positive time $T$,
\begin{equation} \label{est:rho_L^2-L^inf}
\left\|\rho_\veps\,-\,1\right\|_{L^2([0,T];L^\infty(\Omega))}\,\leq\,C_T\,\veps^{1-\alpha}\,,
\end{equation}
where we denote by $C_T$ a quantity proportional (for some ``universal'' constant) to $1+T$.

Note that, thanks to the equality $\sqrt{\rho}\,\,\nabla\log\rho\,=\,2\,\nabla\sqrt{\rho}$, from the previous bounds we get also that
$\left(\nabla\sqrt{\rho_\veps}\right)_\veps$ is bounded in $L^\infty\bigl([0,T];L^2(\Omega)\bigr)$, for any $T>0$ fixed.

Let us also remark that we have a nice decay of the first derivatives of $\rho_\veps$ even in the low capillarity
regime: namely, for $0<\alpha\leq1$ (and then $\g=2$), one has
\begin{equation} \label{est:nabla-rho_L^2-L^2}
\left\|\nabla\rho_\veps\right\|_{L^2_T(L^2)}\,\leq\,C_T\,\veps
\end{equation}
for any $T>0$ fixed (see the second point of the previous list of bounds). Notice that the constant $C_T$
does not depend on $\alpha$ (recall Proposition \ref{p:F-unif}).

Finally, let us state an important property on the quantity $D\bigl(\rho^{3/2}_\veps\,u_\veps\bigr)$. First of all, we write
\begin{eqnarray}
D\bigl(\rho^{3/2}_\veps\,u_\veps\bigr) & = & \rho_\veps\,\sqrt{\rho_\veps}\,Du_\veps\,+\,\frac{3}{2}\,\sqrt{\rho_\veps}\,u_\veps\,
D\rho_\veps \label{eq:D(rho-u)} \\
& = & \sqrt{\rho_\veps}\,Du_\veps\,+\,\left(\rho_\veps\,-\,1\right)\sqrt{\rho_\veps}\,Du_\veps\,+\,
\frac{3}{2}\,\sqrt{\rho_\veps}\,u_\veps\,D\rho_\veps\,. \nonumber
\end{eqnarray}
The first term in the right-hand side clearly belongs to $L^2_T(L^2)$, while, by uniform bounds and Sobolev embeddings,
the second and the third ones are uniformly bounded in $L^2_T(L^{3/2})$.
Therefore, we infer that 
%\begin{equation} \label{est:D(rho-u)}
%\Bigl(D\bigl(\rho^{3/2}_\veps\,u_\veps\bigr)\Bigr)_\veps\;\subset\;L^2\bigl(L^2(\Omega)\,+\,L^{3/2}(\Omega)\bigr)
%\end{equation}
$\Bigl(D\bigl(\rho^{3/2}_\veps\,u_\veps\bigr)\Bigr)_\veps$ is a uniformly bounded family in $L^2_T(L^2+L^{3/2})$.

%%%%%%%%%%%%%%%%%%%
\subsection{Constraints on the limit} \label{ss:constraint}
%%%%%%%%%%%%%%%%%%%%%

As specified in the introduction, we want to study the weak limit of the family $\bigl(\rho_\veps,u_\veps\bigr)_\veps$, i.e.
we want to pass to the limit for $\veps\ra0$ in equations \eqref{eq:weak-mass}-\eqref{eq:weak-momentum} when computed
on $\bigl(\rho_\veps,u_\veps\bigr)$.

The present paragraph is devoted to establishing some properties the weak limit has to satisfy.

\medbreak
By uniform bounds, seeing $L^\infty$ as the dual of $L^1$, we infer, up to extraction of subsequences,
the weak convergences
%the existence of $u\in L^\infty\bigl(\R_+;L^2(\Omega)\bigr)$
%and $U\in L^2\bigl(\R_+;L^2(\Omega)\bigr)$ such that, up to extraction of subsequences,
\begin{eqnarray*}
\sqrt{\rho_\veps}\,u_\veps\,\stackrel{*}{\rightharpoonup}\,u & 	\qquad\qquad\mbox{ in }\quad &  L^\infty\bigl(\R_+;L^2(\Omega)\bigr) \\[1ex]
%\qquad\qquad\mbox{ and }\qquad\qquad
\sqrt{\rho_\veps}\,Du_\veps\,\rightharpoonup\,U & \qquad\qquad\mbox{ in }\quad &  L^2\bigl(\R_+;L^2(\Omega)\bigr)\,.
\end{eqnarray*}
Here $\stackrel{*}{\rightharpoonup}$ denotes the weak-$*$ convergence in $L^\infty\bigl(\R_+;L^2(\Omega)\bigr)$.

On the other hand, thanks to the estimates for the density, we immediately deduce that $\rho_\veps\,\ra\,1$ (strong convergence)
in $L^\infty\bigl(\R_+;L^2(\Omega)\bigr)$, with convergence rate of order $\veps$.
So, we can write $\rho_\veps\,=\,1\,+\,\veps\,r_\veps$, with the family $\bigl(r_\veps\bigr)_\veps$ bounded in
$L^\infty\bigl(\R_+;L^2(\Omega)\bigr)$, and then (up to an extraction) weakly convergent to some $r$ in this space.

Notice that, in the case $\alpha=0$, we know that actually $\bigl(\rho_\veps\bigr)_\veps$ strongly converges
to $1$ in the space $L^\infty\bigl(\R_+;H^1(\Omega)\bigr)\,\cap\,L^2_{loc}\bigl(\R_+;H^2(\Omega)\bigr)$, still
with rate $O(\veps)$. Then we infer also that
\begin{equation} \label{eq:conv-r_a=0}
\qquad\qquad\qquad
r_\veps\,\rightharpoonup\,r\qquad\qquad\qquad\mbox{ in }\quad
L^\infty\bigl(\R_+;H^1(\Omega)\bigr)\,\cap\,L^2_{loc}\bigl(\R_+;H^2(\Omega)\bigr)\,.
\end{equation}
In the case $0<\alpha\leq1$, thanks to \eqref{est:nabla-rho_L^2-L^2} we gather instead that
\begin{equation} \label{eq:conv-r_a}
\qquad\qquad\qquad
r_\veps\,\rightharpoonup\,r\qquad\qquad\qquad\mbox{ in }\quad L^2_{loc}\bigl(\R_+;H^1(\Omega)\bigr)\,.
\end{equation}

Notice also that, as expected, one has $U\,=\,Du$, and then $u\in L^2\bigl(\R_+;H^1(\Omega)\bigr)$.
As a matter of fact, consider equation \eqref{eq:D(rho-u)}: using again the trick
$\rho_\veps\,=\,1+\left(\rho_\veps-1\right)$ together with \eqref{est:rho_L^inf-L^2}, it is easy to check that
$$
D\bigl(\rho_\veps^{3/2}\,u_\veps\bigr)\,\longrightarrow\,Du\qquad\qquad\mbox{ in }\qquad \mc{D}'\,.
$$
On the other hand, the bounds \eqref{est:rho_L^inf-L^2} and \eqref{est:nabla-rho_L^2-L^2} imply that the right-hand side
of \eqref{eq:D(rho-u)} weakly converges to $U$, and this proves our claim.

\begin{comment}
As a matter of fact, fix $T>0$ and
take $\phi\in\mc{D}\bigl([0,T[\,\times\Omega\bigr)$: on the one hand, we can write
$$
\int^T_0\int_\Omega\sqrt{\rho_\veps}\,Du_\veps\cdot\rho_\veps\,\phi\,dx\,dt\,=\,
\int^T_0\int_\Omega\sqrt{\rho_\veps}\,Du_\veps\cdot\phi\,dx\,dt\,+\,
\int^T_0\int_\Omega\left(\rho_\veps-1\right)\,\sqrt{\rho_\veps}\,Du_\veps\cdot\phi\,dx\,dt\,,
%\,\ra\,\int^T_0\int_\Omega U\cdot\phi\,dx\,dt\,;
$$
which clearly converges to $\int^T_0\int_\Omega U\cdot\phi\,dx\,dt$.
On the other hand, integrating by parts leads us to
\begin{eqnarray*}
\int^T_0\int_\Omega\sqrt{\rho_\veps}\,Du_\veps\cdot\rho_\veps\,\phi\,dx\,dt & = & -\int^T_0\int_\Omega D\left(\rho_\veps^{3/2}\right)
\,u_\veps\cdot\phi\,dx\,dt\,-\,\int^T_0\int_\Omega\rho_\veps\,\sqrt{\rho_\veps}\,u_\veps\cdot D\phi\,dx\,dt \\
& = & -\,\frac{3}{2}\int^T_0\int_\Omega\sqrt{\rho_\veps}\,u_\veps\,D\rho_\veps\cdot\phi\,dx\,dt\,-\,
\int^T_0\int_\Omega\rho_\veps\,\sqrt{\rho_\veps}\,u_\veps\cdot D\phi\,dx\,dt\,:
\end{eqnarray*}
due to the strong convergence of $\rho_\veps$, it is easy to see that the former term in the right-hand side converges to $0$,
while the latter to $-\int^T_0\int_\Omega u\cdot D\phi\,dx\,dt$. This proves our claim.
\end{comment}

Let us also point out that
\begin{equation} \label{eq:conv-rho-u}
\qquad\qquad\qquad
\rho_\veps\,u_\veps\,\rightharpoonup\,u\qquad\qquad\qquad\mbox{ in }\quad L^2\bigl([0,T];L^2(\Omega)\bigr)\,.
\end{equation}
In fact, we can write $\rho_\veps u_\veps\,=\,\sqrt{\rho_\veps}u_\veps\,+\,\left(\sqrt{\rho_\veps}-1\right)\sqrt{\rho_\veps}u_\veps$.
By $\left|\sqrt{\rho_\veps}-1\right|\leq\left|\rho_\veps-1\right|$ and Sobolev embeddings, we get that the second term in the right-hand
side converges strongly to $0$ in $L^\infty\bigl([0,T];L^1(\Omega)\cap L^{3/2}(\Omega)\bigr)\,\cap\,L^2\bigl([0,T];L^2(\Omega)\bigr)$.

Exactly in the same way, we find that
\begin{equation} \label{eq:conv-rho-Du}
\qquad\quad
\rho_\veps\,Du_\veps\,\rightharpoonup\,Du\qquad\qquad\qquad\mbox{ in }\quad
L^1\bigl([0,T];L^2(\Omega)\bigr)\,\cap\,L^2\bigl([0,T];L^1(\Omega)\cap L^{3/2}(\Omega)\bigr)\,.
\end{equation}

We conclude this part by proving the following proposition, which can be seen as the analogue of the
Taylor-Proudman theorem in our context.
\begin{prop} \label{p:weak-limit}
Let $\bigl(\rho_\veps,u_\veps\bigr)_\veps$ be a family of weak solutions (in the sense of Definition \ref{d:weak}  above)
to system \eqref{eq:NSK+rot}-\eqref{eq:bc}, with initial data $\bigl(\rho_{0,\veps},u_{0,\veps}\bigr)$ satisfying
the hypotheses fixed in Section \ref{s:result}. \\
Let us define $r_\veps:=\veps^{-1}\left(\rho_\veps-1\right)$, and let $(r,u)$ be a limit point of the sequence
$\bigl(r_\veps,u_\veps\bigr)_\veps$.

Then $r=r(x^h)$ and $u=\bigl(u^h(x^h),0\bigr)$, with $\div_{\!h}u^h=0$. Moreover, $r$ and $u$ are linked by the relation
$$
\left\{\begin{array}{ll}
u^h\;=\,\nabla_h^\perp r & \qquad\mbox{ if }\quad0\,<\,\alpha\,\leq\,1 \\[2ex]
u^h\;=\,\nabla_h^\perp\bigl(\Id-\Delta\bigr)r & \qquad\mbox{ if }\quad\alpha\,=\,0\,.
\end{array}\right.
$$
\end{prop}

\begin{proof}
Let us consider first the mass equation in the (classical) weak formulation, i.e. \eqref{eq:weak-mass}:
writing $\rho_\veps\,=\,1\,+\,\veps\,r_\veps$ as above, for any $\phi\in\mc{D}\bigl([0,T[\,\times\Omega\bigr)$ we have
$$
-\,\veps\int^T_0\int_\Omega r_\veps\,\d_t\phi\,-\,\int^T_0\int_\Omega\rho_\veps\,u_\veps\cdot\nabla\phi\,=\,
\veps\int_\Omega r_{0,\veps}\,\phi(0)\,.
$$
Letting $\veps\ra0$, we deduce that $\int^T_0\int_\Omega u\cdot\nabla\phi\,=\,0$, which implies
\begin{equation} \label{constr:div}
\div\,u\;\equiv\:0\qquad\qquad\mbox{ almost everywhere in }\qquad [0,T]\times\Omega\,.
\end{equation}

After that, we turn our attention to the (modified) weak formulation of the momentum equation, given by
\eqref{eq:weak-momentum}: we multiply it by $\veps$ and we pass to the limit $\veps\ra0$.
By uniform bounds, it is easy to see that the only integrals which do not go to $0$ are the ones involving the pressure, the rotation
and the capillarity: let us analyse them carefully.

First of all, let us deal with the pressure term: we rewrite it as
\begin{eqnarray*}
\frac{1}{\veps}\int^T_0\int_\Omega\nabla P(\rho_\veps)\cdot\rho_\veps\,\psi & = & 
\frac{1}{\veps}\int^T_0\int_\Omega\nabla P(\rho_\veps)\cdot\left(\rho_\veps-1\right)\,\psi\,+\,
\frac{1}{\veps}\int^T_0\int_\Omega\nabla P(\rho_\veps)\cdot\psi \\
& = & \int^T_0\int_\Omega r_\veps\,\rho_\veps^{\g-1}\,\nabla\rho_\veps\cdot\psi\,+\,
\frac{1}{\veps}\int^T_0\int_\Omega\nabla P(\rho_\veps)\cdot\psi\,.
\end{eqnarray*}
Using the boundedness of $\bigl(r_\veps\bigr)_\veps$ in $L^\infty_T(L^2)$ and the strong convergence of $\nabla\rho_\veps\ra0$
in $L^\infty_T(L^2)$ and (as $0<\g-1\leq1$) of $\rho_\veps^{\g-1}\ra1$ in $L^{2}_T(L^\infty)$, one infers that the former term
of the last equality goes to $0$. The latter, instead, can be rewritten in the following way:
$$
\frac{1}{\veps}\int^T_0\!\!\!\int_\Omega\nabla P(\rho_\veps)\cdot\psi\,=\,-\frac{1}{\veps}\int^T_0\!\!\!\int_\Omega
\biggl(P(\rho_\veps)-P(1)-P'(1)\left(\rho_\veps-1\right)\biggr)\div\psi\,+\,
\frac{1}{\veps}\,P'(1)\int^T_0\!\!\!\int_\Omega\nabla\rho_\veps\cdot\psi\,.
$$
Notice that the quantity $P(\rho_\veps)-P(1)-P'(1)\left(\rho_\veps-1\right)$ coincides, up to a factor $1/(\g-1)$, with the internal
energy $h(\rho_\veps)$: since $h(\rho_\veps)/\veps^2$ is bounded in $L^\infty_T(L^1)$
(see Subsection \ref{ss:unif-bounds}), the first integral tends to $0$ for $\veps\ra0$. Finally, thanks also to bounds 
\eqref{eq:conv-r_a=0} and \eqref{eq:conv-r_a}, we find
$$
\frac{1}{\veps}\int^T_0\int_\Omega\nabla P(\rho_\veps)\cdot\rho_\veps\,\psi\;\longrightarrow\;
\int^T_0\int_\Omega\nabla r\cdot\psi\,.
$$

We now consider the rotation term: by \eqref{eq:conv-rho-u} and the strong convergence $\rho_\veps\ra1$ in $L^\infty_T(L^2)$,
we immediately get
$$
\int^T_0\int_\Omega e^3\times\rho_\veps^2\,u_\veps\cdot\psi\;\longrightarrow\;\int^T_0\int_\Omega e^3\times u\cdot\psi\,.
$$

Finally, we deal with the capillarity terms: on the one hand, thanks to the uniform bounds for 
$\bigl(\veps^{-(1-\alpha)}\,\nabla\rho_\veps\bigr)_\veps$ and $\bigl(\veps^{-(1-\alpha)}\,\nabla^2\rho_\veps\bigr)_\veps$,
we get
$$
\frac{2\,\veps}{\veps^{2(1-\alpha)}}\int^T_0\int_\Omega\rho_\veps\,\Delta\rho_\veps\,\nabla\rho_\veps\cdot\psi
%\,=\,2\,\veps\int^T_0\int_\Omega\rho_\veps\,\Delta r_\veps\,\nabla r_\veps\cdot\psi
\;\longrightarrow\;0\,.
$$
On the other hand, splitting $\rho^2_\veps\,=\,1+\left(\rho_\veps-1\right)\left(\rho_\veps+1\right)$ and using uniform bounds again,
one easily gets that the quantity
$$
\frac{\veps}{\veps^{2(1-\alpha)}}\int^T_0\int_\Omega\rho_\veps^2\,\Delta\rho_\veps\,\div\psi\,=\,
\frac{\veps^\alpha}{\veps^{1-\alpha}}\int^T_0\int_\Omega\rho_\veps^2\,\Delta\rho_\veps\,\div\psi
%\;\longrightarrow\;
%\int^T_0\int_\Omega\Delta r\,\div\psi\,.
$$
converges to $0$ in the case $0<\alpha\leq1$, while it converges to $\int^T_0\int_\Omega\Delta r\,\div\psi$
in the case $\alpha=0$.

Let us now restrict for a while to the case $\alpha=0$. To sum up, in the limit $\veps\ra0$, the equation
for the velocity field (tested against $\veps\,\rho_\veps\,\psi$) gives the constraint
\begin{equation} \label{constr:u-r}
e^3\,\times\,u\,+\,\nabla\wtilde{r}\,=\,0\,,\qquad\qquad\qquad\mbox{ with }\qquad \wtilde{r}\,:=\,r\,-\,\Delta r\,.
\end{equation}
This means that
$$
\begin{cases}
 \d_1\wtilde{r}\,=\,\,u^2 \\
 \d_2\wtilde{r}\,=\,-\,u^1 \\
 \d_3\wtilde{r}\,=\,0\,,
\end{cases}
$$
which immediately implies that $\wtilde{r}\,=\,\wtilde{r}(x^h)$ depends just on the horizontal variables. From this, it follows that also
$u^h\,=\,u^h(x^h)$.

Moreover, from the previous system we easily deduce that
\begin{equation} \label{constr:div_h}
\div_{\!h}\,u^h\,=\,0\,,
\end{equation}
which, together with \eqref{constr:div}, entails that $\d_3u^3\equiv0$. Due to the complete slip boundary conditions, we then infer
that $u^3\equiv0$ almost everywhere in $[0,T]\times\Omega$.
%\begin{equation} \label{constr:u^3}
%u^3\,\equiv\,0\qquad\qquad\mbox{ almost everywhere in }\qquad [0,T]\times\Omega\,.
%\end{equation}
In the end, we have proved that the limit velocity field $u$ is two-dimensional, horizontal and divergence-free.

Finally, let us come back to $r$: by what we have said before, $\d_3r$ fulfills the elliptic equation
$$
-\,\Delta\d_3r\,+\,\d_3r\,=\,0\qquad\qquad\mbox{ in }\qquad \Omega\,.
$$
By passing to Fourier transform in $\R^2\times\mbb{T}^1$, or by energy methods (because $\d_3r\in L^\infty_T(H^1)$), or
by spectral theory (since the Laplace operator has only positive eigenvalues), we find that
\begin{equation} \label{constr:r}
\d_3r\,\equiv\,0\qquad\qquad\Longrightarrow\qquad\qquad r\,=\,r(x^h)\,.
\end{equation}

The same arguments as above also apply when $0<\alpha\leq1$, working with $r$ itself instead of $\wtilde{r}$. Notice that
the property $r=r(x^h)$ is then straightforward, because of the third equation in \eqref{constr:u-r}.

The proposition is now completely proved.
\end{proof}

%%%%%%%%%%%%%%%%%%%%%%%%%%%%%%%%%%%%%%%%%%%%%%%%%%%%%%%%%%%%%%%%%%%%%%%%%%%%%%%%%%%%%%%%%%%%%%%%%
%%%%%%%%%%%%%%%%%%%%%%%%%%%%%%%%%%%%%%%%%%%%%%%%%%%%%%%%%%%%%%%%%%%%%%%%%%%%%%%%%%%%%%%%%%%%%%%%%
\section{Vanishing capillarity limit: the case $\alpha=1$} \label{s:low-cap}
%%%%%%%%%%%%%%%%%%%%%%%%%%%%%%%%%%%%%%%%%%%%%%%%%%%%%%%%%%%%%%%%%%%%%%%%%%%%%%%%%%%%%%%%%%%%%%%%%
%%%%%%%%%%%%%%%%%%%%%%%%%%%%%%%%%%%%%%%%%%%%%%%%%%%%%%%%%%%%%%%%%%%%%%%%%%%%%%%%%%%%%%%%%%%%%%%%%

In this section we restrict our attention to the vanishing capillarity limit, and we prove
Theorem \ref{th:sing-lim} in the special (and simpler) case $\alpha=1$. In fact, when $0<\alpha<1$ the system
presents an anisotropy in $\veps$, which requires a modification of the arguments of the proof: we refer to
Section \ref{s:general_a} for the analysis.

We first study the propagation of acoustic waves, from which we infer (by use of the RAGE theorem)
the strong convergence of the quantities $\bigl(r_\veps\bigr)_\veps$ and 
$\bigl(\rho^{3/2}_\veps\,u_\veps\bigr)_\veps$ in $L^2_T\bigl(L^2_{loc}(\Omega)\bigr)$. We are then able to pass to the limit
in the weak formulation \eqref{eq:weak-mass}-\eqref{eq:weak-momentum}, and to identify the limit system.

%%%%%%%%%%%%%%%%%%
\subsection{Analysis of the acoustic waves} \label{ss:acoustic}
%%%%%%%%%%%%%%%%%

The present paragraph is devoted to the analysis of the acoustic waves.
The main goal is to apply the well-known RAGE theorem to prove dispersion of the components of the solutions which are
orthogonal to the kernel of the  singular perturbation operator.

We shall follow the analysis performed in \cite{F-G-N}.

\subsubsection{The acoustic propagator} \label{sss:propagator}

First of all, we rewrite system \eqref{eq:NSK+rot} in the form
\begin{equation} \label{eq:acoust-waves}
\begin{cases}
\veps\,\d_tr_\veps\,+\,\div\,V_\veps\,=\,0 \\[1ex]
\veps\,\d_tV_\veps\,+\,\Bigl(e^3\times V_\veps\,+\,\nabla r_\veps\Bigr)\,=\,\veps\,f_\veps\,,
\end{cases}
\end{equation}
where we have defined $V_\veps\,:=\,\rho_\veps\,u_\veps$ and
\begin{eqnarray}
f_\veps & := & -\,\div\left(\rho_\veps u_\veps\otimes u_\veps\right)\,+\,\nu\,\div\left(\rho_\veps Du_\veps\right)\,- 
\label{eq:f_veps} \\
& & \qquad -\,\frac{1}{\veps^2}\nabla\Bigl(P(\rho_\veps)-P(1)-P'(1)\left(\rho_\veps-1\right)\Bigr)\,+\,
\rho_\veps\,\nabla\Delta\rho_\veps\,. \nonumber
\end{eqnarray}

Of course, system \eqref{eq:acoust-waves} has to be read in the weak sense specified by Definition \ref{d:weak}: for any scalar
$\phi\in\mc{D}\bigl([0,T[\,\times\Omega\bigr)$ one has
$$
-\,\veps\int^T_0\int_\Omega r_\veps\,\d_t\phi\,dx\,dt\,-\,\int^T_0\int_\Omega V_\veps\cdot\nabla\phi\,dx\,dt\,=\,
\veps\int_\Omega r_{0,\veps}\,\phi(0)\,dx\,,
$$
and, for any $\psi\in\mc{D}\bigl([0,T[\,\times\Omega\bigr)$ with values in $\R^3$,
%\begin{eqnarray*}
$$%& & \hspace{-0.5cm}
\int^T_0\!\!\int_\Omega\biggl(-\,\veps\,V_\veps\cdot\d_t\bigl(\rho_\veps\psi\bigr)\,+\,\rho_\veps e^3\times V_\veps\cdot\psi\,-\,
r_\veps\,\div\!\bigl(\rho_\veps\psi\bigr)\biggr)\,=\,\veps\int_\Omega\rho_{0,\veps}^2\,u_{0,\veps}\,\psi(0)\,+\,
\veps\int^T_0\langle f_\veps,\rho_\veps\psi\rangle\,,
%& & \qquad =\,
$$%\end{eqnarray*}
where we have set
\begin{eqnarray*}
%\hspace{-0.5cm}
\langle f_\veps,\zeta\rangle & := & \int_\Omega\biggl(\rho_\veps u_\veps\otimes u_\veps:\nabla\zeta\,-\,
\nu\,\rho_\veps Du_\veps:\nabla\zeta\,-\,\Delta\rho_\veps\,\nabla\rho_\veps\cdot\zeta\,- \\
& & \qquad\qquad\qquad -\,\rho_\veps\,\Delta\rho_\veps\,\div\zeta\,+\,
\frac{1}{\veps^2}\Bigl(P(\rho_\veps)-P(1)-P'(1)\left(\rho_\veps-1\right)\Bigr)\div\zeta\biggr)dx \\
& = & \int_\Omega\Bigl(f^1_\veps:\nabla\zeta\,+\,f^2_\veps:\nabla\zeta\,+\,f^3_\veps\cdot\zeta\,+\,f^4_\veps\,\div\zeta\,+\,
f^5_\veps\,\div\zeta\Bigr)\,dx\,.
\end{eqnarray*}

Since $\bigl(\sqrt{\rho_\veps}u_\veps\bigr)_\veps\subset L^\infty_T(L^2)$ and $f^5_\veps\sim h(\rho_\veps)$,
uniform bounds imply that $\bigl(f^1_\veps\bigr)_\veps$ and $\bigl(f^5_\veps\bigr)_\veps$ are
uniformly bounded in $L^\infty_T(L^1)$.
Since $\nabla\rho_\veps\in L^\infty_T(L^2)$ and $\nabla^2\rho_\veps\in L^2_T(L^2)$ are uniformly bounded,
we get that $\bigl(f^3_\veps\bigr)_\veps\subset L^2_T(L^1)$ is bounded, and so is $\bigl(f^4_\veps\bigr)_\veps$  in
$L^2_T(L^2+L^1)$ (write $\rho_\veps=1+(\rho_\veps-1)$ and use \eqref{est:rho_L^inf-L^2} for the second term). Finally, by writing
$\rho_\veps=\sqrt{\rho_\veps}+(\sqrt{\rho_\veps}-1)\sqrt{\rho_\veps}$ and arguing similarly as for \eqref{eq:conv-rho-Du},
we discover that $\bigl(f^2_\veps\bigr)_\veps$ is bounded in $L^2_T(L^2+L^1)$.

Then we get that $\bigl(f_\veps\bigr)_\veps$ is bounded in $L^2_T\bigl(W^{-1,2}(\Omega)+W^{-1,1}(\Omega)\bigr)$.

\medbreak
This having been established, let us turn our attention to the acoustic propagator, i.e. the operator $\mc{A}$ defined by
$$
\begin{array}{lccc}
\mc{A}\,: & L^2(\Omega)\;\times\;L^2(\Omega) & \longrightarrow & H^{-1}(\Omega)\;\times\;H^{-1}(\Omega) \\
& \bigl(r\;,\;V\bigr) & \mapsto & \Bigl(\div V\;,\;e^3\times V\,+\nabla r\Bigr)\,.
\end{array}
$$
We remark that $\mc{A}$ is skew-adjoint with respect to the $L^2(\Omega)\times L^2(\Omega)$ scalar product.

Notice that, by Proposition \ref{p:weak-limit}, any limit point $\left(r,u\right)$ of the sequence of weak
solutions has to belong to ${\rm Ker}\,\mc{A}$.

Moreover, the following proposition holds true. For the proof, see Subsection 3.1 of \cite{F-G-N}.

\begin{prop} \label{p:A-spec}
Let us denote by $\sigma_p(\mc{A})$ the point spectrum of $\mc{A}$.
Then $\sigma_p(\mc{A})\,=\,\{0\}$.

In particular, if we define by ${\rm Eigen}\,\mc{A}$ the space spanned by the eigenvectors of $\mc{A}$, we have
${\rm Eigen}\,\mc{A}\;\equiv\;{\rm Ker}\,\mc{A}$.
\end{prop}

%%%
\subsubsection{Application of the RAGE theorem} \label{sss:RAGE}
%%%

Let us first recall the RAGE theorem and some of its consequences. The present form is the same used in \cite{F-G-N}
(see \cite{Cyc-Fr-K-S}, Theorem 5.8).
\begin{thm} \label{th:RAGE}
Let $\mc H$ be a Hilbert space and $\mc{B}\,:\,D(\mc{B})\subset \mc H\,\longrightarrow\,\mc H$ a self-adjoint operator.
Denote by $\Pi_{\rm cont}$ the orthogonal projection onto the subspace $\mc H_{\rm cont}$, where
$$
\mc H\;=\;\mc H_{\rm cont}\,\oplus\,\oline{{\rm Eigen}\,(\mc{B})}
$$
and $\oline{\Theta}$  is the closure of a subset $\Theta$ in $\mc H$. Finally, let $\mc{K}\,:\,\mc H\,\longrightarrow\,\mc H$
be a compact operator.

Then, in the limit for $T\ra+\infty$ one has
$$
\left\|\frac{1}{T}\,\int^T_0e^{-i\,t\,\mc{B}}\;\mc{K}\;\Pi_{\rm cont}\;e^{i\,t\,\mc{B}}\;dt\right\|_{\mc{L}(\mc H)}\,
\longrightarrow\,0\,.
$$
\end{thm}

Exactly as in \cite{F-G-N}, from the previous theorem we infer the following properties.
\begin{coroll} \label{c:RAGE}
Under the hypothesis of Theorem \ref{th:RAGE}, suppose moreover that $\mc{K}$ is self-adjoint, with $\mc{K}\geq0$.

Then there exists a function $\mu$, with $\mu(\veps)\ra0$ for $\veps\ra0$, such that:
\begin{itemize}
 \item[1)] for any $Y\in\mc H$ and any $T>0$, one has
 $$
 \frac{1}{T}\int^T_0\left\|\mc{K}^{1/2}\,e^{i\,t\,\mc{B}/\veps}\,\Pi_{\rm cont}Y\right\|^2_{\mc H}\,dt\,\leq\,
 \mu(\veps)\,\|Y\|_{\mc H}^2\,;
 $$
 \item[2)] for any $T>0$ and any $X\in L^2\bigl([0,T];\mc H\bigr)$, one has
 $$
 \frac{1}{T^2}\left\|\mc{K}^{1/2}\,\Pi_{\rm cont}\int^t_0e^{i\,(t-\tau)\,\mc{B}/\veps}\,X(\tau)\,d\tau\right\|^2_{L^2([0,T];\mc H)}\,\leq\,
 \mu(\veps)\,\left\|X\right\|^2_{L^2([0,T];\mc H)}\,.
 $$
\end{itemize}
\end{coroll}

We now come back to our problem.
For any fixed $M>0$, define the Hilbert space $H_M$ by
\begin{equation} \label{eq:def-H_M}
H_M\,:=\,\left\{(r,V)\in L^2(\Omega)\times L^2(\Omega)\;\bigl|\;
\what{r}(\xi^h,k)\equiv0\;\mbox{ and }\;\what{V}(\xi^h,k)\equiv0\quad\mbox{if}\quad\bigl|\xi^h\bigr|+|k|>M\right\}\,,
\end{equation}
and let $P_M\,:\,L^2(\Omega)\times L^2(\Omega)\,\longrightarrow\,H_M$ be the orthogonal projection onto $H_M$. For a fixed
$\theta\in\mc{D}(\Omega)$ such that $0\leq\theta\leq1$, we also define the operator
$$
\mc{K}_{M,\theta}(r,V)\,:=\,P_M\bigl(\theta\,P_M(r,V)\bigr)
$$
acting on $H_M$. Note that $\mc{K}_{M,\theta}$ is self-adjoint and positive; moreover, it is also compact by Rellich-Kondrachov
theorem, since its range is included in the set of functions having compact spectrum.

We want to apply the RAGE theorem to
$$
\mc H\,=\,H_M\;,\quad \mc{B}\,=\,i\,\mc{A}\;,\quad \mc{K}\,=\,\mc{K}_{M,\theta}\quad\mbox{ and }\quad\Pi_{\rm cont}\,=\,Q^\perp\,,
$$
where $Q$ and $Q^\perp$ are the orthogonal projections onto respectively ${\rm Ker}\,\mc{A}$ and
$\bigl({\rm Ker}\,\mc{A}\bigr)^\perp$.

Let us set $\left(r_{\veps,M},V_{\veps,M}\right)\,:=\,P_M(r_\veps,V_\veps)$, and note that, thanks to a priori bounds,
for any $M$ it makes sense to apply the term $f_\veps$ to any element of $H_M$. Hence, from system \eqref{eq:acoust-waves} we get
\begin{equation} \label{eq:acoustic-M}
 \veps\,\frac{d}{dt}\bigl(r_{\veps,M}\,,\,V_{\veps,M}\bigr)\,+\,\mc{A}\bigl(r_{\veps,M}\,,\,V_{\veps,M}\bigr)\,=\,
 \veps\,\bigl(0\,,\,f_{\veps,M}\bigr)\,,
\end{equation}
where $\bigl(0\,,\,f_{\veps,M}\bigr)\,\in\,H_M^*\cong H_M$ is defined by
\begin{eqnarray*}
\langle\bigl(0,f_{\veps,M}\bigr)\,,\,\left(s,P_M\bigl(\rho_\veps\,W\bigr)\right) \rangle_{H_M} & := & 
\int_\Omega\Bigl(f^1_\veps:\nabla P_M\bigl(\rho_\veps W\bigr)+
f^2_\veps:\nabla P_M\bigl(\rho_\veps W\bigr)+ \\
& & \; +f^3_\veps\cdot P_M\bigl(\rho_\veps W\bigr)+
f^4_\veps\,\div P_M\bigl(\rho_\veps W\bigr)+f^5_\veps\,\div P_M\bigl(\rho_\veps W\bigr)\Bigr)dx
\end{eqnarray*}
for any $(s,W)\in H_M$. Moreover, by Bernstein inequalities (due to the localization in the phase space)
it is easy to see that, for any $T>0$ fixed and any $W\in H_M$,
\begin{eqnarray*}
\left\|P_M\bigl(\rho_\veps W\bigr)\right\|_{L^2_T(W^{1,\infty}\cap H^1)} & \leq & C(M)\,\left\|\rho_\veps W\right\|_{L^2_T(L^2)} \\
& \leq &
C(M)\,\left(\|W\|_{L^2_T(L^2)}\,+\,\|\rho_\veps-1\|_{L^\infty_T(L^2)}\,\|W\|_{L^2_T(L^\infty)}\right)\,,
\end{eqnarray*}
for some constant $C(M)$ depending only on $M$. This fact, combined with the uniform bounds we established on $f_\veps$, entails
\begin{equation} \label{est:f_eps-M}
\left\|\bigl(0\,,\,f_{\veps,M}\bigr)\right\|_{L^2_T(H_M)}\,\leq\,C(M)\,.
\end{equation}

By use of Duhamel's formula, solutions to equation \eqref{eq:acoustic-M} can be written as
\begin{equation} \label{eq:acoust-int}
\bigl(r_{\veps,M}\,,\,V_{\veps,M}\bigr)(t)\,=\,e^{i\,t\,\mc{B}/\veps}\bigl(r_{\veps,M}\,,\,V_{\veps,M}\bigr)(0)\,+\,
\int^t_0e^{i\,(t-\tau)\,\mc{B}/\veps}\,\bigl(0\,,\,f_{\veps,M}\bigr)\,d\tau\,.
\end{equation}
Note that, by definition (and since $[P_M,Q]=0$),
$$
\left\|\left(\mc{K}_{M,\theta}\right)^{1/2}\,Q^\perp\bigl(r_{\veps,M}\,,\,V_{\veps,M}\bigr)\right\|^2_{H_M}\,=\,
\int_\Omega\theta\,\left|Q^\perp\bigl(r_{\veps,M}\,,\,V_{\veps,M}\bigr)\right|^2\,dx\,.
$$
Therefore, a straightforward application of Corollary \ref{c:RAGE} (recalling also Proposition \ref{p:A-spec}) gives that, for $T>0$ fixed
and for $\veps$ going to $0$,
\begin{equation} \label{conv:ker-ort}
Q^\perp\bigl(r_{\veps,M}\,,\,V_{\veps,M}\bigr)\,\longrightarrow\,0\qquad\mbox{ in }\quad L^2\bigl([0,T]\times K\bigr)
\end{equation}
for any fixed $M>0$ and any compact set $K\subset\Omega$.

On the other hand, applying operator $Q$ to equation \eqref{eq:acoust-int} and differentiating in time, by use also
of bounds \eqref{est:f_eps-M} we discover that (for any $M>0$ fixed) the family
$\bigl(\d_tQ\bigl(r_{\veps,M}\,,\,V_{\veps,M}\bigr)\bigr)_\veps$ is uniformly bounded (with respect to $\veps$)
in the space $L^2_T(H_M)$. Moreover, as $H_M\hookrightarrow H^m$ for any $m\in\N$, we infer also that
it is compactly embedded in $L^2(K)$ for any $M>0$ and any compact subset $K\subset\Omega$. Hence, Ascoli-Arzel\`a theorem
implies that, for $\veps\ra0$,
\begin{equation} \label{conv:ker}
Q\bigl(r_{\veps,M}\,,\,V_{\veps,M}\bigr)\,\longrightarrow\,\bigl(r_M\,,\,u_M\bigr)\qquad\mbox{ in }\quad L^2\bigl([0,T]\times K\bigr)\,.
\end{equation}
%where the  functions $r$ and $u$ are the limits which have been identified in Subsection \ref{ss:constraint}, and which have to satisfy
%the constraints given in Proposition \ref{p:weak-limit}.

%%%%%%%%%%%%%%%%%%%%%%%%%%%%%%%%%%%%%%%%%%%%%%%%%%%%%%%%%%%%%%%%%%%%%%%%%%%%%%%%%%%%%%%%%%%%%%%%%
%%%%%%%%%%%%%%%%%%%%%%%%%%%%%%%%%%%%%%%%%%%%%%%%%%%%%%%%%%%%%%%%%%%%%%%%%%%%%%%%%%%%%%%%%%%%%%%%%
\subsection{Passing to the limit} \label{ss:limit}
%%%%%%%%%%%%%%%%%%%%%%%%%%%%%%%%%%%%%%%%%%%%%%%%%%%%%%%%%%%%%%%%%%%%%%%%%%%%%%%%%%%%%%%%%%%%%%%%%
%%%%%%%%%%%%%%%%%%%%%%%%%%%%%%%%%%%%%%%%%%%%%%%%%%%%%%%%%%%%%%%%%%%%%%%%%%%%%%%%%%%%%%%%%%%%%%%%%

In the present subsection we conclude the proof of Theorem \ref{th:sing-lim} when $\alpha=1$.
First of all, we show strong convergence of the $r_\veps$'s and the
velocity fields; then we pass to the limit in the weak formulation of the equations, and we identify the limit system.

%%%%
\subsubsection{Strong convergence of the velocity fields} \label{sss:strong}
%%%%

The goal of the present paragraph is to prove the following proposition, which will allow us to pass to the limit in the weak
formulation \eqref{eq:weak-mass}-\eqref{eq:weak-momentum} of our system.
\begin{prop} \label{p:strong}
Let $\alpha=1$ and $\g=2$.
For any $T>0$, for $\veps\ra0$ one has, up to extraction of a subsequence, the strong convergences
$$
r_\veps\,\longrightarrow\,r\qquad\mbox{ and }\qquad
\rho_\veps^{3/2}\,u_\veps\,\longrightarrow\,u\qquad\quad\mbox{ in }\quad L^2\bigl([0,T];L^2_{loc}(\Omega)\bigr)\,.
$$
\end{prop}

\begin{proof}
We start by decomposing $\rho^{3/2}_\veps\,u_\veps$ into low and high frequencies: namely, for any $M>0$ fixed, we can write
$$
\rho^{3/2}_\veps\,u_\veps\,=\,P_M\bigl(\rho^{3/2}_\veps\,u_\veps\bigr)\,+\,\left(\Id-P_M\right)\bigl(\rho^{3/2}_\veps\,u_\veps\bigr)\,.
$$

Let us consider the low frequencies term first: again, it can be separated into the sum of two pieces, namely
$$
P_M\bigl(\rho^{3/2}_\veps\,u_\veps\bigr)\,=\,\veps\,P_M\bigl(\veps^{-1}\left(\sqrt{\rho_\veps}-1\right)\,\rho_\veps\,u_\veps\bigr)\,+\,
P_M\bigl(\rho_\veps\,u_\veps\bigr)\,.
$$
By uniform bounds (recall also Subsection \ref{ss:constraint}), we have that $\bigl(\rho_\veps\,u_\veps\bigr)_\veps$ is bounded in
$L^2_T(L^2)$, while $\bigl(\veps^{-1}\left(\sqrt{\rho_\veps}-1\right)\bigr)_\veps$ is clearly bounded in $L^\infty_T(L^2)$. Then,
using also Bernstein's inequalities, we infer that the former item in the previous equality goes to $0$ in $L^2_T(L^2)$,
in the limit for $\veps\ra0$.

On the other hand, by properties \eqref{conv:ker-ort} and \eqref{conv:ker} we immediately get that
$P_M\bigl(\rho_\veps\,u_\veps\bigr)$ converges to $u_M=P_M(u)$ strongly $L^2_T(L^2_{loc})$.
Recall that $u$ is the limit velocity field identified in Subsection \ref{ss:constraint}, and which has to satisfy, together with
$r$, the constraints given in Proposition \ref{p:weak-limit}.

We deal now with the high frequencies term. Recall that, by decomposition \eqref{eq:D(rho-u)}, we have already deduced
the uniform inclusion $\Bigl(D\bigl(\rho_\veps^{3/2}\,u_\veps\bigr)\Bigr)_\veps\,\subset\,L^2_T(L^2+L^{3/2})$.
Then by Lemma \ref{l:density} and Proposition \ref{p:emb_hom-besov}, we get
$$
\left\|\left(\Id-P_M\right)\bigl(\rho^{3/2}_\veps\,u_\veps\bigr)\right\|_{L^2_T(L^2)}\,\leq\,c_M\,,
$$
for some constant $c_M$, depending just on $M$ (and not on $\veps$) and which tends to $0$ for $M\ra+\infty$.

For the convergence of $r_\veps$ to $r$ one can argue in an analogous way. The control of the high frequency part is actually
easier, thanks to \eqref{eq:conv-r_a}. For the low frequencies, we decompose again $P_Mr_\veps=QP_Mr_\veps+Q^\perp P_Mr_\veps$,
for which we use \eqref{conv:ker} and \eqref{conv:ker-ort} respectively .

The proposition is then proved.
\end{proof}

%%%
\subsubsection{The limit system} \label{sss:limit-system}
%%%

Thanks to the convergence properties established in Subsection \ref{ss:constraint} and by Proposition \ref{p:strong},
we can pass to the limit in the weak formulation \eqref{eq:weak-mass}-\eqref{eq:weak-momentum}.
For this, we evaluate the equations on an element which already belongs to ${\rm Ker}\,\mc{A}$.

So, let us take $\phi\in\mc{D}\bigl([0,T[\,\times\Omega\bigr)$, with $\phi=\phi(x^h)$, and use
$\psi\,=\,\left(\nabla^\perp_h\phi,0\right)$ as a test function in equation \eqref{eq:weak-momentum}: since $\div\psi=0$, we get
\begin{eqnarray}
& & \hspace{-0.1cm} \int^T_0\int_\Omega\biggl(-\rho_\veps^2\,u_\veps\cdot\d_t\psi\,-\,\rho_\veps u_\veps\otimes\rho_\veps u_\veps:\nabla\psi\,+\,
\rho_\veps^2\left(u_\veps\cdot\psi\right)\div u_\veps\,+\,\frac{1}{\veps}\,e^3\times\rho_\veps^2u_\veps\cdot\psi\,+ \label{eq:weak-limit} \\
& & +\,\nu\rho_\veps Du_\veps:\rho_\veps\nabla\psi\,+\,\nu\rho_\veps Du_\veps:\left(\psi\otimes\nabla\rho_\veps\right)\,+\,
2\,\rho_\veps\Delta\rho_\veps\nabla\rho_\veps\cdot\psi\biggr)\,dx\,dt\,=\,
\int_\Omega\rho_{0,\veps}^2\,u_{0,\veps}\cdot\psi(0)\,dx\,. \nonumber
\end{eqnarray}

Now, we rewrite the rotation term in the following way:
\begin{eqnarray*}
\frac{1}{\veps}\int^T_0\!\!\int_\Omega e^3\times\rho_\veps^2u_\veps\cdot\psi & = & 
\frac{1}{\veps}\int^T_0\!\!\int_\Omega\rho_\veps^2u^h_\veps\cdot\nabla_h\phi\;=\;\frac{1}{\veps}\int^T_0\!\!\int_\Omega\rho_\veps
u^h_\veps\cdot\nabla_h\phi\,+\,\int^T_0\!\!\int_\Omega r_\veps \rho_\veps u^h_\veps\cdot\nabla_h\phi \\
& = & -\,\int_\Omega r_{0,\veps}\,\phi(0)\,-\,\int^T_0\!\!\int_\Omega r_\veps\,\d_t\phi\,+\,\int^T_0\!\!\int_\Omega r_\veps\,\rho_\veps\,
u^h_\veps\cdot\nabla_h\phi\,,
\end{eqnarray*}
where the last equality comes from the mass equation \eqref{eq:weak-mass} tested on $\phi$.
Notice that the last term in the right-hand side converges, due to \eqref{eq:conv-rho-u} and
the strong convergence of $r_\veps$ in $L^2_T(L^2)$ (which is guaranteed by Proposition \ref{p:strong}).

Using again the trick $\rho_\veps=1+(\rho_\veps-1)$, we can also write
$$
2\int^T_0\int_\Omega\rho_\veps\Delta\rho_\veps\nabla\rho_\veps\cdot\psi\,=\,2\int^T_0\int_\Omega\Delta\rho_\veps\,
\nabla\rho_\veps\cdot\psi\,+\,2\int^T_0\int_\Omega\left(\rho_\veps-1\right)\,\Delta\rho_\veps\,\nabla\rho_\veps\cdot\psi\,:
$$
by uniform bounds and \eqref{est:nabla-rho_L^2-L^2}, it is easy to see that both terms goes to $0$ for $\veps\ra0$.

Putting these last two relations into \eqref{eq:weak-limit} and using convergence properties established above in order
to pass to the limit, we arrive at the equation
\begin{eqnarray}
& & \hspace{-0.2cm} \int^T_0\int_\Omega\biggl(-\,u\cdot\d_t\psi\,-\,u\otimes u:\nabla\psi\,-\,r\,\d_t\phi\,+ \label{eq:weak} \\
& & \qquad\qquad\qquad\qquad\qquad
+\,r\,u^h\cdot\nabla_h\phi\,+\,\nu Du:\nabla\psi\biggr)\,dx\,dt\,=\,
\int_\Omega\bigl(u_0\cdot\psi(0)\,+\,r_0\,\phi(0)\bigr)\,dx\,. \nonumber
\end{eqnarray}

Now we use that $\psi=\bigl(\nabla_h^\perp\phi,0\bigr)$ and that, by Proposition \ref{p:weak-limit}, $u=\bigl(\nabla^\perp_hr,0\bigr)$.
Keeping in mind that all these functions don't depend on $x^3$, by integration by parts we get
\begin{eqnarray*}
-\int^T_0\int_\Omega u\cdot\d_t\psi\,dx\,dt & = & \int^T_0\int_{\R^2}\Delta_hr\,\d_t\phi\,dx^h\,dt \\
-\int^T_0\int_\Omega u\otimes u:\nabla\psi\,dx\,dt & = & -\int^T_0\int_{\R^2}\nabla^\perp_hr\cdot\nabla_h\Delta_hr\,
\phi\,dx^h\,dt \\
\nu\int^T_0\int_\Omega Du:\nabla\psi\,dx\,dt & = & \frac{\nu}{2}\int^T_0\int_{\R^2}\Delta_h^2r\,\phi\,dx^h\,dt\,.
\end{eqnarray*}
In the same way, one can show that the following identity holds true:
$$
\int^T_0\int_\Omega r\,u^h\cdot\nabla_h\phi\,dx\,dt\,=\,\int^T_0\int_{\R^2}\nabla_hr\cdot\nabla^\perp_hr\,\phi\,dx^h\,dt\,=\,0\,.
$$

Putting all these equalities together completes the proof of Theorem \ref{th:sing-lim} in the case $\alpha=1$.

%%%%%%%%%%%%%%%%%%%%%%%%%%%%%%%%%%%%%%%%%%%%%%%%%%%%%%%%%%%%%%%%%%%%%%%%%%%%%%%%%%%%%%%%%%%%%%%%%
%%%%%%%%%%%%%%%%%%%%%%%%%%%%%%%%%%%%%%%%%%%%%%%%%%%%%%%%%%%%%%%%%%%%%%%%%%%%%%%%%%%%%%%%%%%%%%%%%
\section{The case $\alpha=0$} \label{s:alpha=0}
%%%%%%%%%%%%%%%%%%%%%%%%%%%%%%%%%%%%%%%%%%%%%%%%%%%%%%%%%%%%%%%%%%%%%%%%%%%%%%%%%%%%%%%%%%%%%%%%%
%%%%%%%%%%%%%%%%%%%%%%%%%%%%%%%%%%%%%%%%%%%%%%%%%%%%%%%%%%%%%%%%%%%%%%%%%%%%%%%%%%%%%%%%%%%%%%%%%

We consider now the case of constant capillarity coefficient, i.e. $\alpha=0$: the present section is devoted to the
proof of Theorem \ref{th:alpha=0}.

The main issue of the analysis here is that, now, the singular perturbation operator becomes
\begin{equation} \label{eq:A_0}
\begin{array}{lccc}
\mc{A}_0\,: & L^2(\Omega)\;\times\;L^2(\Omega) & \longrightarrow & H^{-1}(\Omega)\;\times\;H^{-3}(\Omega) \\
& \bigl(\,r\;,\;V\,\bigr) & \mapsto & \Bigl(\div V\;,\;e^3\times V\,+\nabla\bigl(\Id-\Delta\bigr)r\Bigr)\,,
\end{array}
\end{equation}
which is no more skew-adjoint with respect to the usual $L^2$ scalar product. 
Nonetheless, it is possible to symmetrize our system, i.e. it is possible to find a scalar product
on the space $H_M$, defined in \eqref{eq:def-H_M}, with respect to which the operator $\mc{A}_0$ becomes
skew-adjoint.

Indeed, passing in Fourier variables, one can easily compute a positive self-adjoint $4\times4$
matrix $\mc S$ such that $\mc S\,\mc{A}_0\,=\,-\,\mc{A}_0^*\,\mc S$, which is exactly the condition for $\mc{A}_0$ to be skew-adjoint
with respect to the scalar product defined by $\mc S$. Hence, one can apply the RAGE theorem machinery to $\mc{A}_0$, acting on $H_M$
endowed with the scalar product $\mc S$.

\begin{comment}
Let us also remark that one could try to resort to some compensated compactness arguments in order to pass to the limit in
the non-linear terms, but actually we were not able to make this strategy work here.
Indeed, the main problems come from the non-linear term
\begin{equation} \label{eq:bad-term}
\int_0^T\int_\Omega\rho_\veps^2u_\veps\cdot\psi\;\div u_\veps\,dx\,dt\,,
\end{equation}
which comes out from the new weak formulation of our equations.
In order to pass to the limit in it, we would need some compactness for the sequence $\bigl(\rho_\veps^{3/2}\,u_\veps\bigr)_\veps$.
Actually, it is possible to prove that $D\bigl(\rho_\veps^{3/2}\,u_\veps\bigr)$ is uniformly bounded in
$L^2_T\bigl(L^2+L^{3/2}\bigr)$, but we miss uniform properties in time, so that Ascoli-Arzel\`a arguments do not apply.
On the other hand, we would like to smooth out the functions, and try to exploit the structure of the equation.
But we don't have informations on the derivatives of $u_\veps$, and neither on $u_\veps$ itself
(so we cannot separate $u_\veps$ and $\div u_\veps$ from $\rho_\veps$);
hence, it's not clear even how to smooth out the items in order to start the compensated compactness argument.

\end{comment}

\medbreak
After this brief introduction, let us go back to the proof of Theorem \ref{th:alpha=0}.
As before, we first analyse the acoustic waves, proving strong convergence for the velocity fields, and then we will
pass to the limit.

%%%
\subsection{Propagation of acoustic waves} \label{ss:acoustic_0}
%%%

In the case $\alpha=0$, the equation for acoustic waves is the same as \eqref{eq:acoust-waves}, with operator $\mc{A}$
replaced by $\mc{A}_0$: namely, we have
\begin{equation} \label{eq:ac-waves_0}
\begin{cases}
\veps\,\d_tr_\veps\,+\,\div\,V_\veps\,=\,0 \\[1ex]
\veps\,\d_tV_\veps\,+\,\Bigl(e^3\times V_\veps\,+\,\nabla\bigl(\Id-\Delta\bigr)r_\veps\Bigr)\,=\,\veps\,\wtilde{f}_\veps\,,
\end{cases}
\end{equation}
where $\wtilde{f}_\veps$ is analogous to $f_\veps$, which was defined in Paragraph \ref{sss:propagator},
but with the last term of formula \eqref{eq:f_veps} replaced by
$$
\frac{1}{\veps^2}\,\bigl(\rho_\veps-1\bigr)\,\nabla\Delta\rho_\veps\,.
$$
In particular, we have
\begin{eqnarray*}
\langle\wtilde{f}_\veps,\phi\rangle & := & \int_\Omega\biggl(\rho_\veps u_\veps\otimes u_\veps:\nabla\phi\,-\,
\nu\,\rho_\veps Du_\veps:\nabla\phi\,-\,\frac{1}{\veps^2}\,\Delta\rho_\veps\,\nabla\rho_\veps\cdot\phi\,- \\
& & \qquad\quad -\,\frac{1}{\veps^2}\,\bigl(\rho_\veps-1\bigr)\,\Delta\rho_\veps\,\div\phi\,+\,
\frac{1}{\veps^2}\Bigl(P(\rho_\veps)-P(1)-P'(1)\left(\rho_\veps-1\right)\Bigr)\div\phi\biggr)dx\,.
%& = & \int_\Omega\Bigl(f^1_\veps:\nabla\zeta\,+\,f^2_\veps:\nabla\zeta\,+\,f^3_\veps\cdot\zeta\,+\,f^4_\veps\,\div\zeta\,+\,
%f^5_\veps\,\div\zeta\Bigr)\,dx\,.
\end{eqnarray*}
Exactly as before, by uniform bounds we get that $\bigl(\wtilde{f}_\veps\bigr)_\veps$ is bounded in
$L^2_T\bigl(W^{-1,2}(\Omega)+W^{-1,1}(\Omega)\bigr)$.

Recall also that, as in the previous case, equations \eqref{eq:ac-waves_0} hold true when computed on test
functions of the form $(\vphi\,,\,\rho_\veps\,\psi)$.

Let us turn our attention to the acoustic propagator $\mc{A}_0$, defined in \eqref{eq:A_0}. We have the following statement,
which is the analogous of Proposition \ref{p:A-spec}.
\begin{prop} \label{p:A-spec_0}
One has $\sigma_p(\mc{A}_0)\,=\,\{0\}$.
In particular, ${\rm Eigen}\,\mc{A}_0\;\equiv\;{\rm Ker}\,\mc{A}_0$.
\end{prop}

\begin{proof}
We have to look for $\lambda\in\C$ for which the following system
$$
\begin{cases}
\div V\;=\;\lambda\,r \\[1ex]
e^3\times V\,+\,\nabla(r\,-\,\Delta r)\;=\;\lambda\,V
\end{cases}
$$
has non-trivial solutions $(r,V)\neq(0,0)$.

Denoting by $\what{v}$ the Fourier transform of a function $v$ in the domain $\R^2\times\mbb{T}^1$, defined
for any $(\xi^h,k)\in\R^2\times\Z$ by the formula
$$
\what{v}(\xi^h,k)\,:=\,\frac{1}{\sqrt{2}}\int_{-1}^{1}\int_{\R^2}e^{-i\,x^h\cdot\xi^h}\,v(x^h,x^3)\,dx^h\;e^{-i\,x^3k}\,dx^3\,,
$$
we can write the previous system in the equivalent way
$$
\begin{cases}
i\left(\xi^h\cdot\what{V}^h\,+\,k\,\what{V}^3\right)\;=\;\lambda\,\what{r} \\[1ex]
-\what{V}^2\,+\,i\,\xi^1\left(1\,+\,\left|\xi^h\right|^2\,+\,k^2\right)\what{r}\;=\;\lambda\,\what{V}^1 \\[1ex]
\what{V}^1\,+\,i\,\xi^2\left(1\,+\,\left|\xi^h\right|^2\,+\,k^2\right)\what{r}\;=\;\lambda\,\what{V}^2 \\[1ex]
\,i\,k\left(1\,+\,\left|\xi^h\right|^2\,+\,k^2\right)\what{r}\;=\;\lambda\,\what{V}^3\,,
\end{cases}
$$
where $\left|\xi^h\right|^2\,=\,\left|\xi^1\right|^2\,+\,\left|\xi^2\right|^2$.
For notation convenience, let us set $\zeta\bigl(\xi^h,k\bigr)=\left|\xi^h\right|^2+k^2$: after easy computations,
we arrive to the following equation for $\lambda$,
$$
\lambda^4\,+\,\left(1\,+\,\zeta\bigl(\xi^h,k\bigr)\,+\,\zeta^2\bigl(\xi^h,k\bigr)\right)\,\lambda^2\,+\,
k^2\,\left(1\,+\,\zeta\bigl(\xi^h,k\bigr)\right)\,=\,0\,,
$$
from which we immediately infer that
$$
\lambda^2\,=\,-\,\frac{1}{2}\left(1+\zeta+\zeta^2\pm\sqrt{\left(1+\zeta+\zeta^2\right)^2\,-\,4\,k^2\,(1+\zeta)}\right)\,.
$$
To have $\lambda$ in the discrete spectrum of $\mc{A}_0$, we need to delete its dependence on $\xi^h$: since $1+\zeta>0$, the
only way to do it is to have $k=0$, for which $\lambda=0$.
\end{proof}

Now, for any fixed $M>0$, we consider the space $H_M$, which was defined in \eqref{eq:def-H_M}, endowed with the scalar product
\begin{equation} \label{eq:scalar-prod}
\langle (r_1,V_1)\,,\,(r_2,V_2)\rangle_{H_M}\,:=\,\langle r_1\,,\,(\Id-\Delta)r_2\rangle_{L^2}\,+\,\langle V_1\,,\,V_2\rangle_{L^2}\,.
\end{equation}
In fact, it is easy to verify that the previous bilinear form is symmetric and positive definite. Moreover, we observe that
$\|(r,V)\|^2_{H_M}\,=\,\left\|(\Id-\Delta)^{1/2}r\right\|^2_{L^2}+\left\|V\right\|^2_{L^2}$.
Straightforward computations also show that $\mc{A}_0$ is skew-adjoint with respect to this scalar product, namely
$$
\langle \mc{A}_0(r_1,V_1)\,,\,(r_2,V_2)\rangle_{H_M}\,=\,-\,\langle (r_1,V_1)\,,\,\mc{A}_0(r_2,V_2)\rangle_{H_M}\,.
$$
%Notice that one can also verify this property by working directly on the matrix symbols of the previous operators.
%Indeed, if one denotes by $\mc S$ the symmetric operator such that
%$\langle\,\cdot\,,\,\cdot\,\rangle_{H_M}\,=\,\langle\,\cdot\,,\,\mc S(\,\cdot\,)\,\rangle_{L^2}$,
%passing in Fourier transform, one can compute directly $\mc S\,\mc{A}_0\,=\,-\,\mc{A}_0^*\,\mc S$.

Now, let us set $P_M\,:\,L^2(\Omega)\times L^2(\Omega)\,\longrightarrow\,H_M$ to  be the orthogonal projection onto $H_M$, as in
Paragraph \ref{sss:RAGE}. This time, for a fixed $\theta\in\mc{D}(\Omega)$ such that $0\leq\theta\leq1$, we define the operator
$$
\wtilde{\mc{K}}_{M,\theta}(r,V)\,:=\,\Bigl(\bigl(\Id-\Delta\bigr)^{-1}P_M\bigl(\theta\,P_Mr\bigr)\,,\,P_M\bigl(\theta\,P_MV\bigr)\Bigr)\,.
$$
Note that $\wtilde{\mc{K}}_{M,\theta}$ is self-adjoint and positive with respect to the scalar product
$\langle\,\cdot\,,\,\cdot\,\rangle_{H_M}$. Moreover, as before, it is compact by Rellich-Kondrachov
theorem.

Now, exactly as done in Paragraph \ref{sss:RAGE}, we apply the RAGE theorem to
$$
\mc H\,=\,H_M\;,\quad \mc{B}\,=\,i\,\mc{A}_0\;,\quad \mc{K}\,=\,\wtilde{\mc{K}}_{M,\theta}\quad\mbox{ and }\quad
\Pi_{\rm cont}\,=\,Q^\perp\,,
$$
where $Q$ and $Q^\perp$ are still the orthogonal projections onto respectively ${\rm Ker}\,\mc{A}_0$ and
$\bigl({\rm Ker}\,\mc{A}_0\bigr)^\perp$.

Since, with our definitions, we still have
$$
\left\|\left(\wtilde{\mc{K}}_{M,\theta}\right)^{1/2}\,Q^\perp\bigl(r_{\veps,M}\,,\,V_{\veps,M}\bigr)\right\|^2_{H_M}\,=\,
\int_\Omega\theta\,\left|Q^\perp\bigl(r_{\veps,M}\,,\,V_{\veps,M}\bigr)\right|^2\,dx\,,
$$
a direct application of the RAGE theorem (or better of Corollary \ref{c:RAGE}) immediately gives us
\begin{equation} \label{conv:ker-ort_0}
Q^\perp\bigl(r_{\veps,M}\,,\,V_{\veps,M}\bigr)\,\longrightarrow\,0\qquad\mbox{ in }\quad L^2\bigl([0,T]\times K\bigr)
\end{equation}
for any fixed $M>0$ and any compact $K\subset\Omega$.

On the other hand, exactly as we did in Paragraph \ref{sss:RAGE}, by Ascoli-Arzel\`a theorem we can deduce the 
strong convergence
\begin{equation} \label{conv:ker_0}
Q\bigl(r_{\veps,M}\,,\,V_{\veps,M}\bigr)\,\longrightarrow\,\bigl(r_M\,,\,u_M\bigr)\qquad\mbox{ in }\quad L^2\bigl([0,T]\times K\bigr)\,.
\end{equation}

%%%
\subsection{Passing to the limit} \label{ss:limit_0}
%%%

Thanks to relations \eqref{conv:ker-ort_0} and \eqref{conv:ker_0}, Proposition \ref{p:strong} still holds true: namely,
we have the strong convergences of
$$
r_\veps\,\longrightarrow\,r\qquad\mbox{ and }\qquad
\rho_\veps^{3/2}\,u_\veps\;\longrightarrow\;u\qquad\qquad\mbox{ in }\qquad L^2\bigl([0,T];L^2_{loc}(\Omega)\bigr)\,,
$$
and this allows us to pass to the limit in the non-linear terms. Note that we get in particular the strong convergence
of $\bigl(\nabla r_\veps\bigr)_\veps$ in $L^2_T\bigl(H^{-1}_{loc}\bigr)$ (up to extraction of a subsequence);
on the other hand, by uniform bounds we know that this family is bounded in $L^2_T\bigl(H^1\bigr)$. Then, by interpolation
we have also the strong convergence in all the intermediate spaces, and especially
\begin{equation} \label{conv:strong_delta-rho}
\nabla r_\veps\,\longrightarrow\,\nabla r\qquad\qquad\mbox{ in }\qquad L^2\bigl([0,T];L^2_{loc}(\Omega)\bigr)\,.
\end{equation}

In order to compute the limit system, let us take $\phi\in\mc{D}\bigl([0,T[\,\times\Omega\bigr)$, with $\phi=\phi(x^h)$, and use
$\psi\,=\,\left(\nabla^\perp_h\phi,0\right)$ as a test function in equation \eqref{eq:weak-momentum}. Since $\div\psi=0$, as before we get
\begin{eqnarray}
& & %\hspace{-0.1cm}
\int^T_0\!\!\!\int_\Omega\biggl(-\rho_\veps^2\,u_\veps\cdot\d_t\psi\,-\,\rho_\veps u_\veps\otimes\rho_\veps u_\veps:\nabla\psi\,+\,
\rho_\veps^2\left(u_\veps\cdot\psi\right)\div u_\veps\,+\,\frac{1}{\veps}\,e^3\times\rho_\veps^2u_\veps\cdot\psi\,+ \label{eq:weak-limit_0} \\
& & +\,\nu\rho_\veps Du_\veps:\rho_\veps\nabla\psi\,+\,\nu\rho_\veps Du_\veps:\left(\psi\otimes\nabla\rho_\veps\right)\,+\,
\frac{2}{\veps^2}\,\rho_\veps\Delta\rho_\veps\nabla\rho_\veps\cdot\psi\biggr)\,dx\,dt\,=\,
\int_\Omega\rho_{0,\veps}^2\,u_{0,\veps}\cdot\psi(0)\,dx\,. \nonumber
\end{eqnarray}

Also in this case, we rewrite the rotation term by using the mass equation \eqref{eq:weak-mass}:
\begin{eqnarray*}
\frac{1}{\veps}\int^T_0\!\!\int_\Omega e^3\times\rho_\veps^2u_\veps\cdot\psi & = & 
\frac{1}{\veps}\int^T_0\!\!\int_\Omega\rho_\veps
u^h_\veps\cdot\nabla_h\phi\,+\,\frac{1}{\veps}\int^T_0\!\!\int_\Omega\left(\rho_\veps-1\right)\rho_\veps u^h_\veps\cdot\nabla_h\phi \\
& = & -\,\int_\Omega r_{0,\veps}\,\phi(0)\,-\,\int^T_0\!\!\int_\Omega r_\veps\,\d_t\phi\,+\,\int^T_0\!\!\int_\Omega r_\veps\,\rho_\veps\,
u^h_\veps\cdot\nabla_h\phi\,.
\end{eqnarray*}
Again, the last term in the right-hand side converges, due to \eqref{eq:conv-rho-u} and
the strong convergence of $r_\veps$ in $L^2_T(L^2)$.

For analysing the capillarity term, we write
$$
\frac{2}{\veps^2}\int^T_0\int_\Omega\rho_\veps\Delta\rho_\veps\nabla\rho_\veps\cdot\psi\,=\,
\frac{2}{\veps^2}\int^T_0\int_\Omega\Delta\rho_\veps\,\nabla\rho_\veps\cdot\psi\,+\,
\frac{2}{\veps^2}\int^T_0\int_\Omega\left(\rho_\veps-1\right)\,\Delta\rho_\veps\,\nabla\rho_\veps\cdot\psi\,.
$$
By uniform bounds, we gather that the second term goes to $0$; on the other hand, combining \eqref{conv:strong_delta-rho} with the
weak convergence of $\Delta r_\veps$ in $L^2_T(L^2)$ implies that also the first term converges for $\veps\ra0$.

Putting these last two relations into \eqref{eq:weak-limit_0} and using convergence properties established above in order to pass to the limit,
we arrive at the equation
\begin{eqnarray}
& & \hspace{-0.2cm} \int^T_0\int_\Omega\biggl(-\,u\cdot\d_t\psi\,-\,u\otimes u:\nabla\psi\,-\,r\,\d_t\phi\,+ \label{eq:weak_0} \\
& & \qquad\qquad
+\,r\,u^h\cdot\nabla_h\phi\,+\,\nu Du:\nabla\psi\,+\,2\,\Delta r\,\nabla r\cdot\psi\biggr)\,dx\,dt\,=\,
\int_\Omega\bigl(u_0\cdot\psi(0)\,+\,r_0\,\phi(0)\bigr)\,dx\,. \nonumber
\end{eqnarray}

Now we use that $\psi=\bigl(\nabla_h^\perp\phi,0\bigr)$ and that, by Proposition \ref{p:weak-limit},
$u=\bigl(\nabla^\perp_h\wtilde{r},0\bigr)$, where we have set $\wtilde{r}\,:=\,\left(\Id-\Delta\right)r$. Keeping in mind that all
these functions do not depend on $x^3$, by integration by parts we get
\begin{eqnarray*}
-\int^T_0\int_\Omega u\cdot\d_t\psi\,dx\,dt & = & \int^T_0\int_{\R^2}\Delta_h\wtilde{r}\,\d_t\phi\,dx^h\,dt \\
-\int^T_0\int_\Omega u\otimes u:\nabla\psi\,dx\,dt & = & -\int^T_0\int_{\R^2}\nabla^\perp_h\wtilde{r}\cdot\nabla_h\Delta_h\wtilde{r}\,
\phi\,dx^h\,dt \\
& = & -\int^T_0\int_{\R^2}\nabla^\perp_hr\cdot\nabla_h\Delta_hr\,\phi\,dx^h\,dt\,+ \\
& & \hspace{-1cm} +\,\int^T_0\int_{\R^2}\nabla^\perp_hr\cdot\nabla_h\Delta^2_hr\,\phi\,dx^h\,dt\,-\,
\int^T_0\int_{\R^2}\nabla^\perp_h\Delta_hr\cdot\nabla_h\Delta^2_hr\,\phi\,dx^h\,dt \\
\nu\int^T_0\int_\Omega Du:\nabla\psi\,dx\,dt & = & \frac{\nu}{2}\int^T_0\int_{\R^2}\Delta_h^2\wtilde{r}\,\phi\,dx^h\,dt\,.
\end{eqnarray*}
Moreover, it is also easy to see that the following identities hold true:
\begin{eqnarray*}
\int^T_0\int_\Omega r\,u^h\cdot\nabla_h\phi\,dx\,dt & = & -\int^T_0\int_{\R^2}\nabla^\perp_hr\cdot\nabla_h\Delta_h r\,\phi\,dx^h\,dt \\
2\int^T_0\int_\Omega \Delta r\,\nabla r\cdot\psi\,dx\,dt & = & 2\int^T_0\int_{\R^2}\nabla^\perp_hr\cdot\nabla_h\Delta_h r\,\phi\,dx^h\,dt\,.
\end{eqnarray*}
Then, these terms, together with the first one coming from the transport part $u\otimes u$, cancel out.

Hence, putting all these equalities together gives us the quasi-geostrophic type equation stated in Theorem \ref{th:alpha=0},
which is now completely proved.

%%%%%%%%%%%%%%%%%%%%%%%%%%%%%%%%%%%%%%%%%%%%%%%%%%%%%%%%%%%%%%%%%%%%%%%%%%%%%%%%%%%%%%%%%%%%%%%%%
%%%%%%%%%%%%%%%%%%%%%%%%%%%%%%%%%%%%%%%%%%%%%%%%%%%%%%%%%%%%%%%%%%%%%%%%%%%%%%%%%%%%%%%%%%%%%%%%%
\section{Vanishing capillarity limit: anisotropic scaling} \label{s:general_a}
%%%%%%%%%%%%%%%%%%%%%%%%%%%%%%%%%%%%%%%%%%%%%%%%%%%%%%%%%%%%%%%%%%%%%%%%%%%%%%%%%%%%%%%%%%%%%%%%%
%%%%%%%%%%%%%%%%%%%%%%%%%%%%%%%%%%%%%%%%%%%%%%%%%%%%%%%%%%%%%%%%%%%%%%%%%%%%%%%%%%%%%%%%%%%%%%%%%

In this section we complete the proof of Theorem \ref{th:sing-lim}, focusing on the remaining cases $0<\alpha<1$.
The results of Section \ref{s:bounds} still holding true, we just have to analyse the propagation of acoustic waves
and to prove strong convergence of the velocity fields.

First of all, let us write system \eqref{eq:NSK+rot} in the form
$$
\begin{cases}
\veps\,\d_tr_\veps\,+\,\div\,V_\veps\,=\,0 \\[1ex]
\veps\,\d_tV_\veps\,+\,\Bigl(e^3\times V_\veps\,+\,\nabla r_\veps\Bigr)\,=\,\veps\,f_{\veps,\alpha}\,+\,
\veps^{\alpha}\,g_\veps\,.
\end{cases}
$$
Here, like in the previous section, $f_{\veps,\alpha}$ is obtained from $f_\veps$ of Paragraph \ref{sss:propagator},
by replacing the last term of formula \eqref{eq:f_veps} with
$$
\frac{1}{\veps^{2(1-\alpha)}}\,\bigl(\rho_\veps-1\bigr)\,\nabla\Delta\rho_\veps\,;
$$
moreover, we have defined
$$
g_\veps\,:=\,\frac{1}{\veps^{1-\alpha}}\,\nabla\Delta\bigl(\rho_\veps-1\bigr)\,.
$$
Notice that $\bigl(f_{\veps,\alpha}\bigr)_\veps$ is bounded in $L^2_T\bigl(W^{-1,2}(\Omega)+W^{-1,1}(\Omega)\bigr)$ for any $\alpha$,
while uniform bounds imply that $\bigl(g_\veps\bigr)_\veps$ is bounded in $L^2_T\bigl(W^{-1,2}(\Omega)\bigr)$.

For $0<\alpha<1$, we remark that the term $g_\veps$ is of higher order than $f_{\veps,\alpha}$: then, we cannot treat it
as a remainder. Then, the first step is to put it on the left-hand side of
the equation, and to read it as a small perturbation of the acoustic propagator $\mc{A}$ (defined in Paragraph \ref{sss:propagator}).
Hence, we are led to consider a one-parameter continuous family of operators, each one of which admits a symmetrizer
on the Hilbert space $H_M$.

Roughly speaking, all these operators have the same point spectrum (we will be much more precise below, see Subsection \ref{ss:appl}):
the idea is then to apply a sort of RAGE theorem for one-parameter family of operators and metrics, in order to prove
dispersion of the components of the solutions orthogonal to the kernels of the acoustic propagators.

In what follows, first of all we will set up the problem in an abstract way, showing a RAGE-type theorem
for families of operators and metrics. This having been done, we will apply the general theory to our particular case, and this will
complete the proof of Theorem \ref{th:sing-lim} for $\alpha\,\in\,]0,1[\,$.

%%%%%%%%%%%%%%%%%%%%%%%%%%%%%%%%%%%%%%%%%%%%%%%%%%%%%%%%%%%%%%%%%%%
\subsection{RAGE theorem depending on a parameter} \label{ss:RAGE_param}
%%%%%%%%%%%%%%%%%%%%%%%%%%%%%%%%%%%%%%%%%%%%%%%%%%%%%%%%%%%

As just said, we want to extend the RAGE theorem to the case when both operators and metrics depend on a small parameter
$\eta$ (for us, $\eta=\veps^{2\alpha}$).

For the sake of completeness, we start by presenting some variants of the Wiener theorem, which is the basis to prove
the RAGE theorem.

%%%
\subsubsection{Variants of the Wiener theorem} \label{sss:wiener}
%%%

First of all, some definitions are in order.
\begin{defin} \label{d:measure_ineq}
 Given two positive measures $\mu$ and $\nu$ defined on a measurable space $(X,\Sigma)$, we say $\;\mu\,\leq\,\nu\;$ if
$\mu(A)\leq\nu(A)$ for any measurable set $A\in\Sigma$.
\end{defin}

\begin{defin} \label{d:measure-cont}
Let $\bigl(\mu_{\eta}\bigr)_{\eta}$ be a one-parameter family of positive measures on a measurable space $(X,\Sigma)$.
We say that it is a \emph{continuous family} (with respect to $\eta$) if, for any $A\in\Sigma$, the map
$\eta\,\mapsto\,\mu_\eta(A)$ is continuous from $[0,1]$ to $\R_+$.
\end{defin}
The notion of continuity we adopt corresponds then to the strong (also called setwise) topology in the space of measures on $(X,\Sigma)$. Notice  that
this notion requires no uniformity with respect to $A\in\Sigma$.

The first result is a very simple adaptation of the original Wiener theorem, which can be found e.g. in \cite{R-S_III}
(see the Appendix to Section XI.17). Its proof goes along the lines of the original one: for later use, however,
we give the most of the details.

\begin{prop} \label{p:wiener_eta}
Let $\bigl(\mu_\eta\bigr)_{\eta\in[0,1]}$ be a family of finite Baire measures on $\R$, such that
\begin{equation} \label{eq:monot-meas}
\mu_{\eta_1}\,\leq\,\mu_{\eta_2}\qquad\qquad\qquad\forall\quad0\,\leq\,\eta_1\,\leq\,\eta_2\,\leq\,1\,.
\end{equation}
For any $\eta\in[0,1]$, let us define the Fourier transform of $\mu_\eta$ by the formula
$$
F_\eta(t)\,:=\,\int_\R e^{-ixt}\,d\mu_\eta(x)\,.
$$

Then one has
$$
\lim_{\eta\ra0}\,\lim_{T\ra+\infty}\,\frac{1}{2T}\int^T_{-T}\left|F_\eta(t)\right|^2\,dt\,=\,
\sum_{x\in\R}\bigl|\mu_0\left(\{x\}\right)\bigr|^2\,.
$$
In particular, if $\mu_0$ has no pure points, then the limit is $0$.
\end{prop}

\begin{proof}
Like in the proof of the original statement, for any fixed $\eta$ we can write
\begin{eqnarray*}
\frac{1}{2T}\int^T_{-T}\left|F_\eta(t)\right|^2\,dt & = & \int_\R d\mu_\eta(x)\left(\int_\R d\mu_\eta(y)
\left(\frac{1}{2T}\int^T_{-T}e^{-i(x-y)t}\,dt\right)\right) \\
& = & \int_\R d\mu_\eta(x)\left(\int_\R d\mu_\eta(y)\,\frac{\sin\bigl(T(x-y)\bigr)}{T(x-y)}\right)
\end{eqnarray*}
by Fubini's theorem. Let us now define
$$
H_\eta(T,x)\,:=\,\int_\R\frac{\sin\bigl(T(x-y)\bigr)}{T(x-y)}\,d\mu_\eta(y)\,:
$$
the integrand in $H_\eta$ is pointwise bounded by $1$; moreover, for $T\longrightarrow+\infty$, it converges to $0$ if $y\neq x$,
and to $1$ if $y=x$. Hence, by dominated convergence theorem we have
$$
\lim_{T\ra+\infty}H_\eta(T,x)\,=\,\mu_\eta\left(\{x\}\right)\,.
$$
Moreover, $\bigl|H_\eta(T,x)\bigr|\,\leq\,\mu_\eta(\R)$; then, by dominated convergence theorem again we infer that
$$
\lim_{T\ra+\infty}\int_\R d\mu_\eta(x)\left(\int_\R d\mu_\eta(y)\,\frac{\sin\bigl(T(x-y)\bigr)}{T(x-y)}\right)\,=\,
\sum_{x\in\R}\bigl|\mu_\eta\left(\{x\}\right)\bigr|^2\,.
$$
Finally, we take the limit for $\eta\longrightarrow0$ and we apply the monotone convergence theorem.
\end{proof}

\begin{rem} \label{r:Wiener}
 Note that, if monotonicity hypothesis \eqref{eq:monot-meas} is not fulfilled, then one gets
$$
\lim_{\eta\ra0}\,\lim_{T\ra+\infty}\,\frac{1}{2T}\int^T_{-T}\left|F_\eta(t)\right|^2\,dt\,=\,
\lim_{\eta\ra0}\,\sum_{x\in\R}\bigl|\mu_\eta\left(\{x\}\right)\bigr|^2\,.
$$
In particular, if $\mu_\eta$ has no pure points for any $\eta$, then still the limit is $0$.
\end{rem}

We are now interested in linking the parameters $\eta$ and $T$ together, and in performing the two limits at the same time.
In this case, we can no more apply the dominated convergence theorem, as the measures themselves change
when $T$ increases. However, the next statement says that the previous result still holds true.

\begin{thm} \label{th:wiener}
Let $\sigma:[0,1]\longrightarrow[0,1]$ be a continuous increasing function, such that $\sigma(0)=0$ and $\sigma(1)=1$.
Let $\bigl(\mu_{\sigma(\veps)}\bigr)_{\veps\in[0,1]}$ be a family of finite Baire measures on $\R$, such that one of the two following
conditions holds true:
\begin{itemize}
 \item $\bigl(\mu_{\eta}\bigr)_{\eta\in[0,1]}$ is monotone increasing in the sense of inequality \eqref{eq:monot-meas};
 \item $\bigl(\mu_{\eta}\bigr)_{\eta\in[0,1]}$ is a continuous family, in the sense of Definition \ref{d:measure-cont}.
\end{itemize}
%which is monotone decreasing in the sense of inequality \eqref{eq:monot-meas}.
For any $\veps\in[0,1]$, let us denote by $F_\veps$ the Fourier transform of the measure $\mu_{\sigma(\veps)}$.

Then we have
$$
\lim_{\veps\ra0}\,\frac{1}{2T}\int^T_{-T}\left|F_\veps\bigl(t/\veps\bigr)\right|^2\,dt\,=\,
\sum_{x\in\R}\bigl|\mu_0\left(\{x\}\right)\bigr|^2\,.
$$
In particular, the limit is $0$ if $\mu_0$ has no pure points.
\end{thm}

\begin{proof}
 First of all, by the change of variable $\tau=t/\veps$, we are reconducted to prove that
$$
 \lim_{\veps\ra0}\,\frac{\veps}{2T}\int^{T/\veps}_{-T/\veps}\left|F_\veps\bigl(t)\right|^2\,dt\,=\,
\sum_{x\in\R}\bigl|\mu_0\left(\{x\}\right)\bigr|^2\,.
$$

Next, as done in the previous proof, the following equalities hold true:
\begin{eqnarray*}
\frac{\veps}{2T}\int^{T/\veps}_{-T/\veps}\left|F_\veps(t)\right|^2\,dt & = &
\int d\mu_{\sigma(\veps)}(x)\left(\int d\mu_{\sigma(\veps)}(y)\,\frac{\sin\bigl(T(x-y)/\veps\bigr)}{T(x-y)/\veps}\right) \\
& = & \int d\mu_{\sigma(\veps)}(x)\left(\mu_{\sigma(\veps)}\bigl(\{x\}\bigr)\,+\,
\int_{y\neq x} d\mu_{\sigma(\veps)}(y)\,\frac{\sin\bigl(T(x-y)/\veps\bigr)}{T(x-y)/\veps}\right) \\
& = & \sum_{x\in\R}\bigl|\mu_{\sigma(\veps)}\left(\{x\}\right)\bigr|^2\,+\,
\int d\mu_{\sigma(\veps)}(x)\left(\int_{y\neq x} d\mu_{\sigma(\veps)}(y)\,\frac{\sin\bigl(T(x-y)/\veps\bigr)}{T(x-y)/\veps}\right)\,.
\end{eqnarray*}
The first term on the right-hand side converges to the same quantity computed in $\veps=0$. Indeed, in the case of a monotone family this follows from monotone convergence theorem;
in the case of a continuous family, instead, we denote $f_\veps(x)\,:=\,\mu_{\sigma(\veps)}\bigl(\{x\}\bigr)$ and we apply Proposition \ref{p:Areskin}.
So, we have only to prove that
\begin{equation} \label{eq:to-prove}
\lim_{\veps\ra0}\,\int d\mu_{\sigma(\veps)}(x)\left(\int_{y\neq x}d\mu_{\sigma(\veps)}(y)\,
\left|\frac{\sin\bigl(T(x-y)/\veps\bigr)}{T(x-y)/\veps}\right|\right)\,=\,0\,.
\end{equation}

We first consider the case when the family of measures is monotone decreasing.

Let us fix a $\delta>0$, and let $\veps_\delta$ be such that $\veps/T\,\leq\,\delta$ for all $\veps\leq\veps_\delta$.
Moreover, let us define the sets $Y_\leq\,:=\,\left\{y\neq x\,\bigl|\,|x-y|\leq g(\delta)\right\}$ and
$Y_\geq\,:=\,\left\{y\neq x\,\bigl|\,|x-y|> g(\delta)\right\}$, for a suitable continuous function $g(\delta)$,
going to $0$ for $\delta\ra0$, to be determined later.

Then we can split the second integral in \eqref{eq:to-prove} into the sum of the integrals on $Y_\leq$ and $Y_\geq$,
and elementary inequalitites give us
$$
\int_{y\neq x}d\mu_{\sigma(\veps)}(y)\,\left|\frac{\sin\bigl(T(x-y)/\veps\bigr)}{T(x-y)/\veps}\right|\,\leq\,
\int_{Y_\leq}d\mu_{\sigma(\veps)}(y)\,+\,
\int_{Y_\geq}\frac{\veps}{T\,|x-y|}\,d\mu_{\sigma(\veps)}(y)\,.
$$

We first focus on the former term: we have
\begin{eqnarray*}
\int d\mu_{\sigma(\veps)}(x)\int_{Y_\leq}d\mu_{\sigma(\veps)}(y) & \leq &
\int\mu_{\sigma(\veps)}\bigl([x-g(\delta),x+g(\delta)]\setminus\{x\}\bigr)\,d\mu_{\sigma(\veps)}(x) \\
& \leq & \int\mu_{1}\bigl([x-g(\delta),x+g(\delta)]\setminus\{x\}\bigr)\,d\mu_{\sigma(\veps)}(x) \\
& \leq & \int\mu_{1}\bigl([x-g(\delta),x+g(\delta)]\setminus\{x\}\bigr)\,d\mu_{1}(x)\,,
\end{eqnarray*}
where the last inequality follow from the monotonicity property of the family of measures (and from the fact that the integrand does
not depend on $\veps$ anymore). From dominated convergence theorem, one can show that
$$
\lim_{\delta\ra0}\int\mu_{1}\bigl([x-g(\delta),x+g(\delta)]\setminus\{x\}\bigr)\,d\mu_{1}(x)\,=\,0
$$
(recall that $g(\delta)\ra0$ for $\delta\ra0$), and from this fact we deduce
\begin{equation} \label{est:Y_-}
\int d\mu_{\sigma(\veps)}(x)\int_{Y_\leq}d\mu_{\sigma(\veps)}(y)\,\leq\,C_\delta\,,
\end{equation}
for some suitable $C_\delta$ converging to $0$ for $\delta\ra0$.

We now consider the integral over $Y_\geq$. By definition of $\delta$, for any $\veps\leq\veps_\delta$ one has
\begin{eqnarray*}
\int d\mu_{\sigma(\veps)}(x)\left(\int_{Y_\geq}\frac{\veps}{T\,|x-y|}\,d\mu_{\sigma(\veps)}(y)\right) & \leq & \frac{\delta}{g(\delta)}\,
\int\mu_{\sigma(\veps)}\bigl(\R\bigr)\,d\mu_{\sigma(\veps)}(x) \\
& \leq & \frac{\delta}{g(\delta)}\,\bigl|\mu_{\sigma(\veps)}\left(\R\right)\bigr|^2\,.
\end{eqnarray*}
The term on the right-hand side can be bounded by the same quantity computed in $\mu_1$;
therefore, if we take for instance $g(\delta)=\sqrt\delta$, we find, for all $\veps\leq\veps_\delta$,
\begin{equation} \label{est:Y_+}
\int d\mu_{\sigma(\veps)}(x)\left(\int_{Y_\geq}\frac{\veps}{T\,|x-y|}\,d\mu_{\sigma(\veps)}(y)\right)\,\leq\,C\,\sqrt\delta\,.
\end{equation}

In the end, putting inequalities \eqref{est:Y_-} and \eqref{est:Y_+} together gives us relation \eqref{eq:to-prove},
and this completes the proof of the theorem.

\medbreak
Let us now prove \eqref{eq:to-prove} in the case of a continuous family of measures. As before, we split the domain of the
second integral into $Y_\leq$ and $Y_\geq$.

The control of the integral over $Y_\geq$ can be performed exactly as done above: we have
$$
\int d\mu_{\sigma(\veps)}(x)\left(\int_{Y_\geq}\frac{\veps}{T\,|x-y|}\,d\mu_{\sigma(\veps)}(y)\right)\,\leq\,
\frac{\delta}{g(\delta)}\,\bigl|\mu_{\sigma(\veps)}\left(\R\right)\bigr|^2\;\leq\;C\,\frac{\delta}{g(\delta)}
$$
for any $\veps\leq\veps_\delta$, where the last inequality follows from the hypothesis of setwise convergence. Again, the choice $g(\delta)=\sqrt{\delta}$ gives us \eqref{est:Y_+}.

For the integral over $Y_\leq$, we still have
$$
\int d\mu_{\sigma(\veps)}(x)\int_{Y_\leq}d\mu_{\sigma(\veps)}(y)\,\leq\,
\int\mu_{\sigma(\veps)}\bigl([x-g(\delta),x+g(\delta)]\setminus\{x\}\bigr)\,d\mu_{\sigma(\veps)}(x)\,.
$$
The idea is to resort once again to Proposition \ref{p:Areskin}. To this end, for all $\veps\geq0$ (recall that we have fixed $\delta>0$ above), we define
$$
\psi_\veps(x)\,=\,\mu_{\sigma(\veps)}\bigl([x-g(\delta),x+g(\delta)]\setminus\{x\}\bigr)\,.
$$
By hypothesis of setwise convergence, it immediately follows that $\psi_\veps(x)\,\longrightarrow\,\psi_0(x)$ in the limit $\veps\ra0$, for all $x\in\R$.
Moreover, $|\psi_\veps(x)|\,\leq\,\mu_{\sigma(\veps)}(\R)\,\leq\,C$ for all $\veps\in[0,1]$ and all $x\in\R$.
Then, we can apply Proposition \ref{p:Areskin} to deduce that
$$
\int d\mu_{\sigma(\veps)}(x)\int_{Y_\leq}d\mu_{\sigma(\veps)}(y)\;\longrightarrow\;
\int\mu_{0}\bigl([x-g(\delta),x+g(\delta)]\setminus\{x\}\bigr)\,d\mu_{0}(x)
$$
for $\veps\ra0$.
On the other hand, by dominated convergence theorem, one easily checks that, for $\delta\ra0$, the integral on the right-hand side
of the previous relation converges to $0$. Hence, for the $\delta>0$ fixed above, there exists a $\veps_\delta'$ such that, for all $0\leq\veps\leq\veps_\delta'$,
\begin{equation} \label{est:Y_-_cont}
\int d\mu_{\sigma(\veps)}(x)\int_{Y_\leq}d\mu_{\sigma(\veps)}(y)\,\leq\,C_\delta\,,
\end{equation}
for some suitable constant $C_\delta$ independent of $\veps$ and going to $0$ when $\delta\ra0$.

In the end, putting estimates \eqref{est:Y_+} and \eqref{est:Y_-_cont} together, we have proved that, for any small $\delta>0$, there exists an $\wtilde{\veps}_\delta\,=\,
\min\left\{\veps_\delta,\veps_\delta'\right\}$ such that,
for all $0<\veps\leq\wtilde{\veps}_\delta$,
$$
\int d\mu_{\sigma(\veps)}(x)\left(\int_{y\neq x}d\mu_{\sigma(\veps)}(y)\,
\left|\frac{\sin\bigl(T(x-y)/\veps\bigr)}{T(x-y)/\veps}\right|\right)\,\leq\,C_\delta\,+\,C\,\sqrt{\delta}\,,
$$
where the right-hand side converges to $0$ for $\delta\ra0$. This property completes the proof of \eqref{eq:to-prove} and of the whole Theorem \ref{th:wiener}.
\end{proof}

%%%
\subsubsection{RAGE-type theorems} \label{sss:RAGE_eta}
%%%

We are now ready to prove some results in the same spirit as the RAGE theorem, for families of operators and metrics.

Despite our attempt of generality, we have to make very precise assumptions for such families, which are modelled
on our problem issued from the Navier-Stokes-Korteweg system. On the other hand, these hypothesis seem to us to be
important in order to prove our result: we will point out where they will be used.

\medbreak
First of all, let us introduce some notations.

We are going to work in a fixed space $\mc{H}$; we will consider in $\mc H$ a continuous family of scalar products
$\bigl(\mc S_\eta\bigr)_{\eta\in[0,1]}$, each one of which induces a Hilbert structure on $\mc H$.
In general, we will write $(\mc H,\mc S_\eta)$ if we consider the Hilbert structure on $\mc H$ induced by the scalar product
$\mc S_\eta$; if we do not specify the scalar product (for instance,
in speaking of a self-adjoint operator), we mean we are referring to $\mc S_0$.
In fact, $\mc S_0$ will be a sort of ``reference metric'' for us, and we will consider the $\mc S_\eta$'s like
perturbations of it.

Moreover, we will use equivalently the notations $\mc S_\eta(X,Y)\,=\,\langle X,Y\rangle_\eta$, and we will denote by
$\|\,\cdot\,\|_\eta$ the induced norm. We will also write $X\perp_\eta Y$ if $X$ and $Y$ are orthogonal with respect to
$\mc S_\eta$; equally, given two subspaces $\mc E_1,\mc E_2\,\subset\,\mc H$, we write $\mc E_1\oplus_\eta\mc E_2$ if they
are orthogonal with respect to $\mc S_\eta$. For a linear operator $\mc P$ defined on $\mc H$, we  will set
$\|\mc P\|_{\mc L(\eta)}$ its operator norm with respect to the scalar product $\mc S_\eta$; for $\eta=0$ we will use
the notations $\|\,\cdot\,\|_{\mc L(0)}$ and $\|\,\cdot\,\|_{\mc L(\mc H)}$ in an equivalent way.
Finally, the adjoint of $\mc P$ with respect to $\mc S_\eta$ will be denoted by $\mc P^{*(\eta)}$.

In the same time, we will consider a one-parameter family of operators $\bigl(\mc B_\eta\bigr)_{\eta\in[0,1]}$, and we will see each
$\mc B_\eta$ as a perturbation of a self-adjoint operator $\mc B_0$ (recall that we mean self-adjoint with respect to $\mc S_0$).

From the original statement (see Theorem \ref{th:RAGE} above), we immediately infer the following one-parameter
variant of the RAGE theorem.
\begin{prop} \label{p:RAGE_eta}
Let $(\mc H,\mc S_0)$ be a Hilbert space, and let $\bigl(\mc S_\eta\bigr)_{\eta\in[0,1]}$ be a one-parameter family of scalar products
on $\mc H$, and suppose that they induces equivalent metrics, independently of $\eta$. \\
Let $\bigl(\mc B_\eta\bigr)_{\eta\in[0,1]}$ be a family of operators on $\mc H$ such that $\mc B_\eta$ is self-adjoint
with respect to the inner product $\mc S_\eta$ for all $\eta\in[0,1]$. \\
Let $\Pi_{\rm cont,\eta}$ the orthogonal (with respect to $\mc S_\eta$) projection onto $\mc H_{{\rm cont},\eta}$, where we defined
$$
\mc{H}\;=\;\mc H_{\rm cont, \eta}\;\oplus_\eta\;\oline{{\rm Eigen}\,(\mc{B}_\eta)}\,.
$$

Then, for any family of compact operators $\bigl(\mc K_\eta\bigr)_\eta$ on $\mc H$, one has
$$
\lim_{\eta\ra0}\,\lim_{T\ra+\infty}\,\left\|\frac{1}{T}\int^T_0e^{-it\mc B_\eta}\,\mc K_\eta\,
\Pi_{\rm cont,\eta}\,e^{it\mc B_\eta}\,dt\right\|_{\mc L(\eta)}\,=\,0\,.
$$
\end{prop}
As a matter of fact, by Theorem \ref{th:RAGE} the limit in $T\ra+\infty$ is $0$ at any $\eta$ fixed.

Remark that, for simplicity, we assumed that all the scalar products $\mc S_\eta$ are equivalent to each other. However,
such a hypothesis is not really needed at this level: it is enough to suppose that each operator $\mc K_\eta$ on $\mc H$
is compact with respect to the topology induced by $\mc S_\eta$.

\medbreak
In view of the application to the Navier-Stokes-Korteweg system, we are interested now in linking the parameters $\eta$ and $T$ together.
In this case, in order to prove a result in the same spirit of the RAGE theorem we need some additional hypotheses.

More precisely, we suppose that both $\bigl(\mc S_\eta\bigr)_\eta$ and $\bigl(\mc B_\eta\bigr)_\eta$
are defined by use of a family of automorphisms $\bigl(\Lambda_\eta\bigr)_\eta$ of $\mc H$ (again, we mean here that they are
bounded with respect to the reference metric $\mc S_0$). Let us make an important remark.
\begin{rem} \label{r:bounded-holom}
We will always suppose that the family of automorphisms $\bigl(\Lambda_\eta\bigr)_\eta$
is (real) \emph{bounded-holomorphic} in the sense of \cite{K}, Chapter VII (see Section 1).
This will be important to have series expansions in $\eta$ for $\Lambda_\eta$ and its inverse $\Lambda^{-1}_\eta$
(see also \cite{K}, Chapter VII, Paragraph 6.2).

Note however that the situation we consider in Subsection \ref{ss:appl} will be much simpler: we will have
$\Lambda_\eta\,=\,1+\eta\Delta$, and everything will be explicit.
\end{rem}

We aim at proving the following statement.
\begin{thm} \label{th:RAGE_eps}
Let $(\mc H,\mc S_0)$ be a Hilbert space, and $\mc B_0\in\mc L(\mc H)$ be a self-adjoint operator. \\
Let $\bigl(\Lambda_\eta\bigr)_{\eta\in[0,1]}$ be a bounded-holomorphic family of automorphisms
of $\mc H$, with $\Lambda_0=\Id$, such that each $\Lambda_\eta$ is self-adjoint and such that the monotonicity property
$$
\Lambda_{\eta_1}\,\leq\,\Lambda_{\eta_2}\qquad\qquad\qquad\forall\quad0\leq\eta_1\leq\eta_2\leq1
$$
(in the sense of self-adjoint operators) is verified. \\
For any $\eta\in[0,1]$, let $\mc S_\eta$ be the scalar product on $\mc H$ induced by $\Lambda_\eta$: for all
$X,Y\in\mc H$, we set $\mc S_\eta(X,Y)\,:=\,\langle X,\Lambda_\eta Y\rangle_0$. \\
Define also $\mc B_\eta\,:=\,\mc B_0\circ\Lambda_\eta$, and suppose that $\sigma_p(\mc B_\eta)\,=\,\{0\}$ for all $\eta$. \\
Let $\mc H_{\rm cont,\eta}$ the orthogonal complement of ${\rm Ker}\,\mc B_\eta$ in $\mc H$ with respect to $\mc S_\eta$:
\begin{equation} \label{eq:decomp_H}
\mc{H}\;=\;\mc H_{\rm cont, \eta}\;\oplus_\eta\;{\rm Ker}\,\mc B_\eta\,,
\end{equation}
and let $\Pi_{\rm cont,\eta}$ be the orthogonal (with respect to $\mc S_\eta$) projection onto $\mc H_{\rm cont,\eta}$. \\
Let us now take $\eta\,=\,\sigma(\veps)$, where $\sigma:[0,1]\longrightarrow[0,1]$ is a continuous increasing function
such that $\sigma(0)=0$ and $\sigma(1)=1$.

Then, for any compact operator $\mc K$ on $\mc H$ and any $T>0$ fixed, defining
$\mc K_{\sigma(\veps)}\,=\,\Lambda_{\sigma(\veps)}^{-1}\mc K$, one has that
$$
\lim_{\veps\ra0}\,\left\|\frac{1}{T}\int^T_0\exp\left(-i\,\frac{t}{\veps}\,\mc B_{\sigma(\veps)}\right)\,\mc K_{\sigma(\veps)}\,
\Pi_{{\rm cont},\sigma(\veps)}\,\exp\left(i\,\frac{t}{\veps}\,\mc B_{\sigma(\veps)}\right)\,dt\right\|_{\mc L(\veps)}\,=\,0\,.
$$
\end{thm}

Before proving the theorem, let us make some comments.
\begin{rem} \label{r:RAGE_param}
 \begin{itemize}
\item[(i)] Notice that, by definitions of $\mc S_\eta$ and $\mc B_\eta$, it immediately follows that each
operator $\mc B_\eta$ is self-adjoint with respect to the scalar product $\mc S_\eta$. Then, the orthogonal decomposition
\eqref{eq:decomp_H} and the definition of the semigroup $\exp\bigl(it\mc B_\eta\bigr)$ make sense.
\item[(ii)] Decomposition \eqref{eq:decomp_H} is based on the hypothesis $\sigma_p(\mc B_\eta)\,=\,\{0\}$ for all $\eta$.
Such a spectral condition is important for stating Lemma \ref{l:ker-ort} and deriving Corollary \ref{c:projection},
which will be used in the proof.
\item[(iii)] On the other hand, the hypothesis $\sigma_p(\mc B_\eta)\,=\,\{0\}$ for all $\eta$ looks quite strong, but
it actually applies to the problem we want to deal with (see Proposition \ref{p:A-spec_a}).
We will not pursue here the issue of weakening this condition; moreover, in Proposition \ref{p:spectral}
we will give a sufficient condition in order to guarantee it (again, such a condition applies to our case, see also Remark
\ref{r:spectral}).
\item[(iv)] The monotonocity of the family of automorphisms $\bigl(\Lambda_\eta\bigr)_\eta$ implies
an analogous property for the scalar products $\bigl(\mc S_\eta\bigr)_\eta$; moreover, since $\Lambda_1$ is in particular
continuous on $(\mc H,\mc S_0)$, we have also $\|\,\cdot\,\|_1\,\leq\,C\,\|\,\cdot\,\|_0$. Then, the metrics (and so the topologies)
induced by the $\mc S_\eta$'s are all equivalent: hence, saying that an operator $\mc K$ is compact, without
any other specification, makes sense in this context.
\item[(v)] The fact that the compact operators $\mc K_\eta$ depend on $\eta$ is important for us, because in the end we want to obtain
an analogue of Corollary \ref{c:RAGE} (the compact operator has to be self-adjoint with respect to each scalar product
we consider). However, working with $\mc K$ (independent of $\eta$) would not have been really usefull:
in the proof we will need to compute its adjoint with respect to $\mc S_\eta$, and there a dependence on
$\eta$ would arise in any case.
\item[(vi)] We also remark the following points. On the one hand, the fact that the $\mc K_\eta$'s are perturbations of a
fixed compact operator $\mc K$ allows us to reduce the proof to the case of an operator of rank $1$ (as in the original
RAGE theorem, see \cite{Cyc-Fr-K-S}): indeed, we need that the approximation by finite rank operators is, in some sense,
uniform in $\eta$.
On the other hand, in the proof we will exploit also the particular form $\mc K_\eta=\Lambda_\eta^{-1}\mc K$ of the perturbations:
it allows us to ``play'' with the special definition of the scalar products $\mc S_\eta$.
Notice that such a hypothesis is well-adapted to the case we want to consider (see Subsection \ref{ss:appl}).
\end{itemize}
\end{rem}

This having been pointed out, some preliminary results are in order.

\begin{lemma} \label{l:ker-ort}
Under the hypotheses of Theorem \ref{th:RAGE_eps}, for all $\eta\in[0,1]$ one has the equality
${\rm Ker}\,\mc B_\eta\,=\,\Lambda_\eta^{-1}\,{\rm Ker}\,\mc B_0$.
In particular, $\mc H_{\rm cont, \eta}\,\equiv\,\mc H_{\rm cont,0}$ for all $\eta$.
\end{lemma}

\begin{proof}%[Proof of Lemma \ref{l:ker-ort}]
 Let $X\in{\rm Ker}\,\mc B_0$. Then, by definition of $\mc B_\eta=\mc B_0\circ\Lambda_\eta$, one immediately has
 $\mc B_\eta\Lambda_\eta^{-1}X=0$, and hence
$\Lambda_\eta^{-1}{\rm Ker}\,\mc B_0\subset{\rm Ker}\,\mc B_\eta$.

On the other hand, if $Y\in{\rm Ker}\,\mc B_\eta$, the element $X:=\Lambda_\eta Y$ belongs to ${\rm Ker}\,\mc B_0$.
So $Y=\Lambda_\eta^{-1}X$, which proves the other inclusion ${\rm Ker}\,\mc B_\eta\subset\Lambda_\eta^{-1}{\rm Ker}\,\mc B_0$.

Let us now work with the orthogonal complements of the kernels.

Fix $E\in\mc H_{\rm cont,\eta}$: we want to prove $\langle E,X\rangle_0=0$ for all $X\in{\rm Ker}\,\mc B_0$.
In fact, from writing $X=\Lambda_\eta Y$, with $Y\in{\rm Ker}\,\mc B_\eta$, one infers
$$
\langle E,X\rangle_0\,=\,\langle E,\Lambda_\eta Y\rangle_0\,=\,\langle E,Y\rangle_\eta\,=\,0\,.
$$
Then $\mc H_{\rm cont,\eta}\subset\mc H_{\rm cont,0}$.

The reverse inclusion is obtained in a totally analogous way.
\end{proof}

Notice that, a priori, the previous proposition does not tell us anything about the orthogonal projections onto these subspaces.
For instance, if $\Pi_{K,\eta}$ denotes the orthogonal (with respect to $\mc S_\eta$) projection onto ${\rm Ker}\,\mc B_\eta$,
we cannot infer that $\Pi_{K,\eta}=\Lambda_\eta^{-1}\Pi_{K,0}$.

Nonetheless, we can state the following corollary.
\begin{coroll} \label{c:projection}
For all $\eta\in[0,1]$, we have
%the equality $\Pi_{\rm cont,\eta}\,=\,\Pi_{\rm cont,0}\,+\,\eta\,\mc R_\eta$,
%where the remainder operators $\mc R_\eta$ satisfy
$$
\Pi_{\rm cont,\eta}\,=\,\Pi_{\rm cont,0}\,+\,\eta\,\mc R_\eta\;,\qquad\qquad\mbox{ with }\qquad
\sup_{\eta\in[0,1]}\left\|\mc R_\eta\right\|_{\mc L(\mc H)}\,\leq\,C\,.
$$
\end{coroll}

\begin{proof}%[Proof of Corollary \ref{c:projection}]
First of all, since $\mc H_{\rm cont,\eta}\equiv\mc H_{\rm cont,0}$ by Lemma \ref{l:ker-ort}, we infer
$$\Pi_{\rm cont,\eta}\circ\Pi_{\rm cont,0}\,=\,\Pi_{\rm cont,0}\,. $$

Now, any $X\in\mc H$ can be decomposed into $X\,=\,\Pi_{\rm cont,0}X\,+\,\Pi_{K,0}X$. Hence, from the previous equality we get
$$
\Pi_{\rm cont,\eta}X\,=\,\Pi_{\rm cont,0}X\,+\,\Pi_{\rm cont,\eta}\Pi_{K,0}X\,.
$$
Then, we have just to understand the action of $\Pi_{\rm cont,\eta}$ on ${\rm Ker}\,\mc B_0$.

Let $Z\in{\rm Ker}\,B_0$. By Lemma \ref{l:ker-ort} we know that $Z_\eta:=\Lambda_\eta^{-1}Z\in{\rm Ker}\,B_\eta$. On the other hand,
by hypothesis on the family $\bigl(\Lambda_\eta\bigr)_\eta$, we can write
$\Lambda_\eta\,=\,\Id\,+\,\eta\,\mc{D}_\eta$, for a suitable bounded family of self-adjoint operators
$\bigl(\mc{D}_\eta\bigr)_\eta\,\subset\,\mc L(\mc H)$. Then one gathers
\begin{eqnarray*}
\Pi_{\rm cont,\eta}Z & = & \Pi_{\rm cont,\eta}\Lambda_\eta Z_\eta\;=\;\Pi_{\rm cont,\eta}Z_\eta\,+\,
\eta\,\Pi_{\rm cont,\eta}\,\mc{D}_\eta\,Z_\eta \\
& = & \eta\,\Pi_{\rm cont,\eta}\,\mc{D}_\eta\,Z_\eta\,,
\end{eqnarray*}
where the last equality follows from the fact that $Z_\eta\in{\rm Ker}\,B_\eta$.

To complete the proof of the corollary, we have just to show that
\begin{equation} \label{est:Pi_eta}
\left\|\Pi_{\rm cont,\eta}\,\mc{D}_\eta\,Z_\eta\right\|_0\,\leq\,C\,\|Z\|_0\,,
\end{equation}
for a constant $C>0$ independent of $\eta$.

We already know that $\sup_{\eta}\left\|\mc D_\eta\right\|_{\mc L(\mc H)}\,\leq\,C$. So, let us estimate
$\left\|\Pi_{\rm cont,\eta}\right\|_{\mc L(\mc H)}$. For all $Y\in\mc H$, we have
\begin{eqnarray*}
\left\|\Pi_{\rm cont,\eta}\,Y\right\|_0^2 & = & \langle\Pi_{\rm cont,\eta}\,Y\,,\,\Pi_{\rm cont,\eta}\,Y\rangle_0
\;=\;\langle\Pi_{\rm cont,\eta}\,Y\,,\,\Lambda_\eta^{-1}\,\Pi_{\rm cont,\eta}\,Y\rangle_\eta \\
& \leq & \left\|\Pi_{\rm cont,\eta}\,Y\right\|_\eta^2\,\left\|\Lambda_\eta^{-1}\right\|_{\mc L(\eta)}\,.
\end{eqnarray*}
For the former term, we use that $\Pi_{\rm cont,\eta}$ is an orthogonal projection with respect to the scalar product $\mc S_\eta$,
so its $\eta$-norm is bounded by $1$: using then the monotonicity property of the $\Lambda_\eta$'s and the continuity
of $\Lambda_1$ with respect to $\mc S_0$, we finally get
$$
\left\|\Pi_{\rm cont,\eta}\,Y\right\|_\eta^2\,\leq\,\left\|Y\right\|_\eta^2\,\leq\,C\,\left\|Y\right\|_0^2\,.
$$
For the latter term, we argue exactly as above: for all $Y\in\mc H$,
\begin{eqnarray*}
\left\|\Lambda^{-1}_{\eta}\,Y\right\|_\eta^2 & = & \langle\Lambda^{-1}_{\eta}\,Y\,,\,\Lambda^{-1}_{\eta}\,Y\rangle_\eta\;=\;
\langle\Lambda^{-1}_{\eta}\,Y\,,\,Y\rangle_0 \\
& \leq & C\,\|Y\|_0^2\;\leq\;C\,\|Y\|_\eta^2\,,
\end{eqnarray*}
where the last estimate comes from the monotonicity hypothesis.
Combining these two last inequalities together, we easily deduce \eqref{est:Pi_eta}, from which the corollary follows.
\end{proof}

We need also the following simple lemma.
\begin{lemma} \label{l:conv_measures}
Under the hypotheses of Theorem \ref{th:RAGE_eps}, let us fix a $X\in\mc H$ and consider the spectral measure
$\mu_\eta$ associated to the element $\Pi_{\rm cont,\eta}\,\Lambda^{-1}_\eta\,\Pi_{\rm cont,0}\,X$.

Then one has $\mu_\eta(\R)\,\longrightarrow\,\mu_0(\R)$ for $\eta\ra0$.
\end{lemma}

\begin{proof}%[Proof of Lemma \ref{l:conv_measures}]
By definition of spectral measure and the spectral theorem, one has
\begin{eqnarray*}
\mu_\eta(\R) & = & \int_\R d\mu_\eta(x)\;=\;\left\|\Pi_{\rm cont,\eta}\,\Lambda^{-1}_\eta\,\Pi_{\rm cont,0}\,X\right\|^2_{\eta} \\
& = & \langle\Pi_{\rm cont,\eta}\,\Lambda^{-1}_\eta\,\Pi_{\rm cont,0}\,X\,,\,\Lambda^{-1}_\eta\,\Pi_{\rm cont,0}\,X\rangle_\eta\,.
\end{eqnarray*}
Now, by definition of $\mc S_\eta$ and Corollary \ref{c:projection}, we have
$$
\mu_\eta(\R)\,=\,\langle\Pi_{\rm cont,\eta}\,\Lambda^{-1}_\eta\,\Pi_{\rm cont,0}\,X\,,\,\Pi_{\rm cont,0}\,X\rangle_0\,=\,
\langle\Lambda^{-1}_\eta\,\Pi_{\rm cont,0}\,X\,,\,\Pi_{\rm cont,0}\,X\rangle_0\,+\,O(\eta)\,,
$$
where we used also that $\Pi_{\rm cont,0}$ is self-adjoint with respect to $\mc S_0$ and $\Pi_{\rm cont,0}^2\,=\,\Pi_{\rm cont,0}$.
At this point, thanks to the hypothesis over the family $\bigl(\Lambda_\eta\bigr)_\eta$, we can write
$\Lambda^{-1}_\eta\,=\,\Id\,+\,\eta\,\wtilde{\mc D}_\eta$, for a suitable bounded family of self-adjoint operators
$\bigl(\wtilde{\mc D}_\eta\bigr)_\eta\subset\mc L(\mc H)$ (see also Chapter VII of \cite{K}, in particular Paragraph 3.2). Then
the previous relation becomes
$$
\mu_\eta(\R)\,=\,\langle\Pi_{\rm cont,0}\,X\,,\,\Pi_{\rm cont,0}\,X\rangle_0\,+\,O(\eta)\,=\,
\mu_0(\R)\,+\,O(\eta)\,,
$$
and this proves the claim of the lemma.
\end{proof}

We can finally prove Theorem \ref{th:RAGE_eps}. We will follow the main lines of the proof given in \cite{Cyc-Fr-K-S} (see Theorem 5.8,
Chapter 5).
\begin{proof}[Proof of Theorem \ref{th:RAGE_eps}]
First of all, we notice that, up to perform the change of variable $\tau=t/\veps$, our claim is equivalent to show that
$$
\lim_{\veps\ra0}\,\left\|\frac{\veps}{T}\int^{T/\veps}_0\exp\left(-i\,t\,\mc B_{\sigma(\veps)}\right)\,\mc K_{\sigma(\veps)}\,
\Pi_{{\rm cont},\sigma(\veps)}\,\exp\left(i\,t\,\mc B_{\sigma(\veps)}\right)\,dt\right\|_{\mc L(\veps)}\,=\,0\,.
$$
For notation convenience, for the moment we keep writing $\eta$ instead of $\sigma(\veps)$.

Since a compact operator can be approximated (in the norm topology) by finite rank operators, and each finite rank operator
can be written as a finite sum of operators of rank $1$, it is enough to restrict to the case of ${\rm rk}\,\mc K=1$.

Recall here point (vi) of Remark \ref{r:RAGE_param}: approximating $\mc K$ gives the ``same approximation'' for all $\mc K_\eta$
(up to the isomorphism $\Lambda_\eta$).
We are not able to exploit this reduction to rank $1$ operators if the approximation itself depended on the particular
compact operator $\mc K_\eta$, with no relations between them.

Since ${\rm rk}\,\mc K=1$, we can represent $\mc K$ with respect to the reference scalar product $\mc S_0$ in the form
$\mc K\,\vphi\,=\,\langle X\,,\,\vphi\rangle_0\;Y$,
%with respect to the reference scalar product $\mc S_0$,
for suitable $X,Y\in\mc H$. Hence, by definitions of $\mc K_\eta$
and $\mc S_\eta$, we have
\begin{eqnarray*}
\mc K_\eta\,\vphi\;=\;\Lambda_\eta^{-1}\,\mc K\,\vphi & = & \langle X\,,\,\vphi\rangle_0\;\Lambda_\eta^{-1}\,Y\;=\;
\langle\Lambda_\eta^{-1}\,X\,,\,\vphi\rangle_\eta\;\Lambda_\eta^{-1}\,Y \\
& = & \langle X_\eta\,,\,\vphi\rangle_\eta\;Y_\eta\,,
\end{eqnarray*}
where we have denoted $X_\eta=\Lambda_\eta^{-1}\,X$ and $Y_\eta=\Lambda_\eta^{-1}\,Y$. Therefore, its adjoint
$\mc K_\eta^{*(\eta)}$ (with respect to the scalar product $\mc S_\eta$) is given by
$$
\mc K_\eta^{*(\eta)}\,\vphi\;=\;\langle Y_\eta\,,\,\vphi\rangle_\eta\;X_\eta\,.
$$

Now, as in \cite{Cyc-Fr-K-S}, for any $\veps\in[0,1]$ fixed and denoting again $\eta=\sigma(\veps)$, we define the operator
\begin{eqnarray*}
Q_\veps(T) & := & \frac{\veps}{T}\int^{T/\veps}_0e^{-it\mc B_\eta}\,\mc K_\eta\,\Pi_{\rm cont,\eta}\,e^{it\mc B_\eta}\,dt \\
& = & \frac{\veps}{T}\,\int^{T/\veps}_0\langle e^{-it\mc B_\eta}\,\Pi_{\rm cont,\eta}\,X_\eta\,,\;\cdot\;\rangle_\eta\,
e^{-it\mc B_\eta}\,Y_\eta\,dt
\end{eqnarray*}
and its adjoint (again, with respect to $\mc S_\eta$)
$$
Q^{*(\eta)}_\veps(T)\;=\;\frac{\veps}{T}\int^{T/\veps}_0\langle e^{-it\mc B_\eta}\,\Pi_{\rm cont,\eta}\,Y_\eta\,,\;\cdot\;\rangle_\eta\,
e^{-it\mc B_\eta}\,X_\eta\,dt\,.
$$
Then, for all $\vphi\in\mc H$, the following identity holds true:
\begin{eqnarray*}
& & \hspace{-0.7cm}
Q_\veps(T)\,Q^{*(\eta)}_\veps(T)\,\vphi\;=\;\frac{\veps}{T}\int^{T/\veps}_0\langle e^{-it\mc B_\eta}\,\Pi_{\rm cont,\eta}\,X_\eta\,,\,
Q^{*(\eta)}_\veps(T)\,\vphi\rangle_\eta\,e^{-it\mc B_\eta}\,Y_\eta\,dt \\
& & \qquad=\;\frac{\veps^2}{T^2}\int^{T/\veps}_0\int^{T/\veps}_0\langle e^{-it\mc B_\eta}\,\Pi_{\rm cont,\eta}\,X_\eta\,,\,
e^{-is\mc B_\eta}\,\Pi_{\rm cont,\eta}\,X_\eta\rangle_\eta\,\langle e^{is\mc B_\eta}\,Y_\eta\,,\,\vphi\rangle_\eta\,e^{-it\mc B_\eta}\,Y_\eta\,ds\,dt\,.
\end{eqnarray*}

Therefore, we can write
\begin{eqnarray*}
& & \hspace{-1.2cm}
\left\|\frac{\veps}{T}\int^{T/\veps}_0e^{-it\mc B_\eta}\,\mc K_\eta\,\Pi_{\rm cont,\eta}\,e^{it\mc B_\eta}\,dt\right\|^2_{\mc L(\eta)}\;=\;
\left\|Q_\veps(T)\right\|^2_{\mc L(\eta)}\;=\;\left\|Q_\veps(T)\;Q^{*(\eta)}_\veps(T)\right\|_{\mc L(\eta)} \\
& & \qquad\qquad\qquad
\leq\;\frac{\veps^2}{T^2}\,\left\|Y_\eta\right\|^2_\eta\,\int^{T/\veps}_0\int^{T/\veps}_0\left|\langle e^{-it\mc B_\eta}\,\Pi_{\rm cont,\eta}\,X_\eta\,,\,
e^{-is\mc B_\eta}\,\Pi_{\rm cont,\eta}\,X_\eta\rangle_\eta\right|\,ds\,dt\,.
\end{eqnarray*}
By definitions and the continuity of the map $\eta\,\mapsto\,\Lambda_\eta$, we infer
$$
\left\|Y_\eta\right\|^2_\eta\,=\,\langle\Lambda_\eta^{-1}Y\,,\,\Lambda_\eta^{-1}Y\rangle_\eta\,=\,
\langle\Lambda_\eta^{-1}Y\,,\,Y\rangle_0\,\leq\,C\,\|Y\|_0^2\,,
$$
and applying the Cauchy-Schwarz inequality we arrive at
\begin{eqnarray*}
& & \hspace{-1cm}
\left\|\frac{\veps}{T}\int^{T/\veps}_0e^{-it\mc B_\eta}\,\mc K_\eta\,\Pi_{\rm cont,\eta}\,e^{it\mc B_\eta}\,dt\right\|^2_{\mc L(\eta)} \\
& & \qquad\qquad\leq\;C\,\|Y\|_0^2\,\left(\frac{\veps^2}{T^2}\int^{T/\veps}_0\int^{T/\veps}_0\left|\langle\Pi_{\rm cont,\eta}\,X_\eta\,,\,
e^{i(t-s)\mc B_\eta}\,\Pi_{\rm cont,\eta}\,X_\eta\rangle_\eta\right|^2\,ds\,dt\right)^{\!1/2}\,.
\end{eqnarray*}

We now focus on the integral term on the right-hand side of the previous inequality.
Notice that, since $\Lambda_\eta^{-1}\,\Pi_{K,0}X\in{\rm Ker}\,\mc B_\eta$, one can write
\begin{equation} \label{eq:def_meas}
\Pi_{\rm cont,\eta}\,X_\eta\,=\,\Pi_{\rm cont,\eta}\,\Lambda_\eta^{-1}\,X\,=\,
\Pi_{\rm cont,\eta}\,\Lambda^{-1}_\eta\,\Pi_{\rm cont,0}\,X\,.
\end{equation}
Then, coming back to the notation $\eta=\sigma(\veps)$, let us consider the quantity
$$
\mc J_\veps\,:=\,\frac{\veps^2}{T^2}\int^{T/\veps}_0\!\!\int^{T/\veps}_0\left|
\langle\Pi_{{\rm cont},\sigma(\veps)}\,\Lambda^{-1}_{\sigma(\veps)}\,\Pi_{\rm cont,0}\,X\,,\,
e^{i(t-s)\mc B_{\sigma(\veps)}}\,\Pi_{{\rm cont},\sigma(\veps)}\,\Lambda^{-1}_{\sigma(\veps)}\,
\Pi_{\rm cont,0}\,X\rangle_{\sigma(\veps)}\right|^2\,ds\,dt\,.
$$

We denote by $\mu_{\sigma(\veps)}$ the spectral measure associated to
$\Pi_{{\rm cont},\sigma(\veps)}\,\Lambda^{-1}_{\sigma(\veps)}\,\Pi_{\rm cont,0}X$.
Therefore, repeating the computations in \cite{Cyc-Fr-K-S}, it follows that
\begin{eqnarray*}
\mc J_\veps & = &
\frac{\veps^2}{T^2}\int^{T/\veps}_0\int^{T/\veps}_0\left(\int_\R\int_\R\exp\bigl(i\,(t-s)\,(x-y)\bigr)\,
d\mu_{\sigma(\veps)}(x)\,d\mu_{\sigma(\veps)}(y)\right)ds\,dt \\
& \leq & \int_\R\int_\R\left(\frac{\sin\bigl((x-y)T/(2\veps)\bigr)}{(x-y)T/(2\veps)}\right)^2\,
d\mu_{\sigma(\veps)}(x)\,d\mu_{\sigma(\veps)}(y)\,.
\end{eqnarray*}

At this point, the convergence argument is the same used in the proof of Theorem \ref{th:wiener} above: we easily obtain that
$\mc I_\veps\longrightarrow0$ for $\veps\ra0$, provided we show that the measures $\mu_{\sigma(\veps)}\longrightarrow\mu_0$ in the sense
of setwise convergence.

\medbreak
For notational convenience, we set $\eta=\sigma(\veps)$ once again, and study the convergence for $\eta\ra0$.

We start by noticing that, by definition of $\mc B_\eta\,=\,\mc B_0\circ\Lambda_\eta$ and continuity in $\eta$ of the family $\Lambda_\eta$, one immediately infers that
$\mc B_\eta\,\vphi\,\ra\,\mc B_0\,\vphi$ for all $\vphi\in\mc H$. Then, by Theorem VIII.25 of \cite{R-S_I}, $\mc B_\eta$ converges to $\mc B_0$ in the strong
resolvent sense. In turn, by Theorem VIII.20 this property implies that, for all fixed function $f\in\mc C_b(\R)$, one has the convergence
\begin{equation} \label{eq:conv_f}
f(\mc B_\eta)\vphi\,\longrightarrow\,f(\mc B_0)\vphi\qquad\qquad\mbox{ for all }\quad \vphi\in\mc H\,.
\end{equation}

Let now $f\in\mc C_b(\R)$, and let us consider, for all $\eta\in[0,1]$, the quantity $\lan\mu_\eta\,,\,f\ran_{\mc M\times\mc C}$, where the brackets
$\lan\,\cdot\,,\,\cdot\,\ran_{\mc M\times\mc C}$ represent the duality pair between $\mc C_b(\R)$ and the set $\mc M_{\mbb Ba}(\R)$ of Baire measures over $\R$.
By spectral theorem, we have that
\begin{eqnarray*}
\lan\mu_\eta\,,\,f\ran_{\mc M\times\mc C}\,=\,\int f\,d\mu_\eta & = &
\lan\Pi_{\rm cont,\eta}\,\Lambda^{-1}_\eta\,\Pi_{\rm cont,0}\,X\,,\,f(\mc B_\eta)\,\Lambda^{-1}_\eta\,\Pi_{\rm cont,0}\,X\rangle_\eta \\
& = & \lan\Pi_{\rm cont,0}\,X\,,\,f(\mc B_\eta)\,\Pi_{\rm cont,0}\,X\rangle_0\,+\,O(\eta)\,,
\end{eqnarray*}
where, in the last step, we have used also Corollary \ref{c:projection} and the fact that $\Lambda^{-1}_\eta\,=\,\Id\,+\,\eta\,\wtilde{\mc D}_\eta$, as in the proof to Lemma \ref{l:conv_measures}.
Hence, using \eqref{eq:conv_f} we gather that, in the limit for $\eta$ going to $0$,
$$
\lan\mu_\eta\,,\,f\ran_{\mc M\times\mc C}\,=\,\int f\,d\mu_\eta\,\longrightarrow\,\lan\Pi_{\rm cont,0}\,X\,,\,f(\mc B_0)\,\Pi_{\rm cont,0}\,X\rangle_0\,=\,
\int f\,d\mu_0\,=\,\lan\mu_0\,,\,f\ran_{\mc M\times\mc C}
$$
for all $f\in\mc C_b(\R)$. Said in an equivalent way, we have just proved that $\mu_\eta\Rightarrow\mu_0$.

But then, by Remark \ref{r:ba-bo} below and Lemma \ref{l:conv_measures}, we can apply Corollary \ref{c:weak-open}. Indeed, each $\mu_\eta$ does not charge the points, being $\mu_\eta$
the spectral measure associated to an element projected onto the continuum spectrum of the operator $\mc B_\eta$. So, from Corollary \ref{c:weak-open} we gather the convergence
$\mu_\eta(I)\,\longrightarrow\,\mu_0(I)$ whenever $I\subset\R$ is an open or closed interval. By Proposition \ref{p:base}, this implies that
$\mu_\eta\,\longrightarrow\,\mu_0$ in the setwise topology.

This concludes the proof of Theorem \ref{th:RAGE_eps}.
\end{proof}

\begin{rem} \label{r:semigroups}
We proved the previous theorem by direct computations. Notice that one could also compare the two propagators, related
to $\mc B_0$ and to $\mc B_{\sigma(\veps)}$, and use properties from perturbation theory of semigroups: we refer e.g. to Theorem 2.19 of
\cite{K}, Chapter IX (see also Theorem 13.5.8 of \cite{Hille-P}). However, it seems to us that these results fail to provide
uniform bounds on time intervals $[0,T/\veps]$ when $\veps\ra0$: this is why we preferred to prove estimates ``by hands''.

Alternatively, one could use the Baker-Campbell-Hausdorff formula (see for instance \cite{Magnus} and \cite{Foguel}) in order to write the propagator
$\exp\bigl(it\mc B_{\sigma(\veps)}\bigr)$ as the propagator $\exp\bigl(it\mc B_0\bigr)$ related to the unperturbed operator,
plus a uniformly bounded remainder of order $\sigma(\veps)$.
\end{rem}

Also in this case, we have the analogue of Corollary \ref{c:RAGE}.
\begin{coroll} \label{c:RAGE_eps}
Under the hypotheses of Theorem \ref{th:RAGE_eps}, suppose moreover that $\mc{K}$ is self-adjoint, with $\mc{K}\geq0$.

Then there exist a constant $C>0$ and a function $\mu$, with $\mu(\veps)\ra0$ for $\veps\ra0$, such that:
\begin{itemize}
 \item[1)] for any $Y\in\mc H$ and any $T>0$, one has
 $$
\frac{1}{T}\int^T_0\left\|\mc{K}_{\sigma(\veps)}^{1/2}\,e^{i\,t\,\mc{B}_{\sigma(\veps)}/\veps}\,
\Pi_{{\rm cont},\sigma(\veps)}Y\right\|^2_{\sigma(\veps)}\,dt\,\leq\,C\,\mu(\veps)\,\|Y\|_{0}^2\,;
 $$
 \item[2)] for any $T>0$ and any $X\in L^2\bigl([0,T];\mc H\bigr)$, one has
$$
\frac{1}{T^2}\left\|\mc{K}_{\sigma(\veps)}^{1/2}\,\Pi_{{\rm cont},\sigma(\veps)}\int^t_0e^{i\,(t-\tau)\,\mc{B}_{\sigma(\veps)}/\veps}\,
X(\tau)\,d\tau\right\|^2_{L^2([0,T];(\mc H,\mc S_{\sigma(\veps)}))}\,\leq\,C\,\mu(\veps)\,\left\|X\right\|^2_{L^2([0,T];\mc H)}\,.
 $$
\end{itemize}
\end{coroll}

Let us conclude this part giving a sufficient condition in order to guarantee the spectral property $\sigma_p(\mc B_\eta)\,=\,\{0\}$
at least for $\eta$ close to $0$: we need $0$ to be an \emph{isolated eigenvalue}\footnote{Here, we mean ``isolated'' in the sense
of \cite{K}, Chapter III, Paragraph 6.5: it is an isolated point not just of $\sigma_p$,
but of the whole spectrum of the operator.} of $\mc B_0$.
\begin{prop} \label{p:spectral}
Let $(\mc H,\mc S_0)$ be a Hilbert space. Let $\mc B_0\in\mc L(\mc H)$ be a self-adjoint operator such that
$\sigma_p(\mc B_0)\,=\,\{0\}$, and suppose also that $0$ is an isolated eigenvalue. \\
Let $\bigl(\Lambda_\eta\bigr)_{\eta\in[0,1]}$ be a bounded-holomorphic family of automorphisms
of $\mc H$, with $\Lambda_0=\Id$, such that each $\Lambda_\eta$ is self-adjoint.
For any $\eta\in[0,1]$, define the operator $\mc B_\eta\,:=\,\mc B_0\circ\Lambda_\eta$.

Then $\sigma_p(\mc B_\eta)\,=\,\{0\}$ for $\eta$ small enough.
\end{prop}

\begin{proof}
By hypothesis, we know that there exists $v\in\mc H$, $v\neq0$, such that $\mc B_0\,v\,=\,0$. We want to solve the equation
\begin{equation} \label{eq:to-solve}
\mc B_\eta\,v_\eta\,=\,\lambda_\eta\,v_\eta
\end{equation}
and show that $\lambda_\eta=0$, at least for small $\eta$.

As done above (see Corollary \ref{c:projection}), by hypothesis we can write $\Lambda_\eta\,=\,\Id\,+\,\eta\,\mc{D}_\eta$,
for a bounded family $\bigl(\mc{D}_\eta\bigr)_\eta$ of self-adjoint operators.

Moreover, since $0$ is an isolated eigenvalue of $\mc{B}_0$, by perturbation theory (see \cite{K}, Theorem 3.16 of Chapter IV and
Theorems 1.7 and 1.8 of Chapter VII), for suitably small $\eta$ we also have
$$
\lambda_\eta\,=\,\lambda_0\,+\,\eta\,\wtilde{\lambda}_{\eta}\qquad\qquad\mbox{ and }\qquad\qquad
v_\eta\,=\,v_0\,+\,\eta\,\wtilde{v}_\eta\,,
$$
where the families of remainders $\bigl(\wtilde{\lambda}_\eta\bigr)_\eta\subset\R$ and
$\bigl(\wtilde{v}_\eta\bigr)_\eta\subset\mc H$ are bounded.

We now insert the previous expansions into \eqref{eq:to-solve}, getting
$$
\mc B_0\,v_0\,+\,\eta\,\mc B_0\,\wtilde{v}_\eta\,+\,\eta\,\mc B_0\,\mc D_\eta\,v_0\,+\,
\eta^2\,\mc B_0\,\mc D_\eta\,\wtilde{v}_\eta\,=\,\lambda_0\,v_0\,+\,\eta\,\lambda_0\,\wtilde{v}_\eta\,+\,
\eta\,\wtilde{\lambda}_\eta\,v_0\,+\,\eta^2\,\wtilde{\lambda}_\eta\,\wtilde{v}_\eta\,,
$$
and we compare the terms with the same power of $\eta$.

From the $0$-th order part, we obviously get that $\lambda_0=0$ and $v_0=v$. Then, the equality involving the terms of order $\eta$
reduces to
$$
\mc B_0\,\wtilde{v}_\eta\,+\,\mc B_0\,\mc D_\eta\,v\,=\,\wtilde{\lambda}_\eta\,v\,.
$$
Now, we take the $\mc S_0$-scalar product of both sides with $v$: since $\mc B_0$ is self-adjoint and $v\in{\rm Ker}\,\mc B_0$,
we immediately get
$$
\wtilde{\lambda}_\eta\,\|v\|_0^2\,=\,0\,.
$$
Using that $v\neq0$, we infer that $\wtilde{\lambda}_\eta=0$, which in turn implies $\lambda_\eta\,=\,0$ (for $\eta$ small enough).
\end{proof}

%%%%%%%%%%%%%%%%%%%%%%%%%%%%%%%%%%%%%%%%%%%%%%%%%%%%%%%%%%%%
\subsection{Application to the vanishing capillarity limit} \label{ss:appl}
%%%%%%%%%%%%%%%%%%%%%%%%%%%%%%%%%%%%%%%%%%%%%%%%%%%%%%%%%%%%%%%

Let us apply now the previous results to our case. As pointed out at the beginning of Section \ref{s:general_a}, for a
fixed $\alpha\,\in\,]0,1[\,$ we rewrite system \eqref{eq:NSK+rot} in the form
\begin{equation} \label{eq:ac-waves_a}
\begin{cases}
\veps\,\d_tr_\veps\,+\,\div\,V_\veps\,=\,0 \\[1ex]
\veps\,\d_tV_\veps\,+\,\Bigl(e^3\times V_\veps\,+\,\nabla\bigl(\Id\,-\,\veps^{2\alpha}\,\Delta\bigr)r_\veps\Bigr)\,=\,
\veps\,f_{\veps,\alpha}\,,
\end{cases}
\end{equation}
where the family $\bigl(f_{\veps,\alpha}\bigr)_\veps$ is bounded in $L^2_T\bigl(W^{-1,2}(\Omega)+W^{-1,1}(\Omega)\bigr)$.

Then, we are led to study the family of operators $\bigl(\mc{A}^{(\alpha)}_\veps\bigr)_\veps$, defined by
$$
\mc{A}^{(\alpha)}_\veps\,:\qquad
\bigl(r\;,\;V\bigr)\quad\mapsto\quad\Bigl(\div V\;,\;e^3\times V\,+\nabla\bigl(\Id\,-\,\veps^{2\alpha}\,\Delta\bigr)r\Bigr)\,.
$$
Notice that one has $\mc{A}^{(\alpha)}_0\equiv\mc{A}$ and $\mc{A}^{(\alpha)}_1\equiv\mc{A}_0$,
where $\mc{A}$ is defined in Paragraph \ref{sss:propagator} and $\mc{A}_0$ in formula \eqref{eq:A_0}.

As one can expect, one has the following result about the point spectrum of each operator.
\begin{prop} \label{p:A-spec_a}
For any $0\leq\veps\leq1$, the point spectrum $\sigma_p\left(\mc{A}^{(\alpha)}_\veps\right)$ contains only $0$.
In particular, ${\rm Eigen}\,\mc{A}^{(\alpha)}_\veps\;\equiv\;{\rm Ker}\,\mc{A}^{(\alpha)}_\veps$.
\end{prop}

\begin{proof}
The same computations performed in the proof of Proposition \ref{p:A-spec_0} give us
$$
\lambda^2\,=\,-\,\frac{1}{2}\left(1+\bigl(1+\veps^{2\alpha}\zeta\bigr)\zeta\pm
\sqrt{\left(1+\bigl(1+\veps^{2\alpha}\zeta\bigr)\zeta\right)^2\,-\,4\,
k^2\,(1+\veps^{2\alpha}\zeta)}\right)\,,
$$
where we recall that we have set $\zeta(\xi^h,k)\,=\,\left|\xi^h\right|^2+k^2$.

As before, to have $\lambda$ in the discrete spectrum of $\mc{A}_0$, we need to delete its dependence on $\xi^h$:
since $1+\veps^{2\alpha}\zeta>0$, the only way to do it is to have $k=0$, for which $\lambda=0$.
\end{proof}

\begin{rem} \label{r:spectral}
Notice that simple computations show also that $0$ is an isolated eigenvalue of the operator $\mc{A}^{(\alpha)}_0$
(here we have to use that in the space $H_M$ the frequencies are bounded). Then, one could alternatively apply
Proposition \ref{p:spectral}.
\end{rem}

As done in Section \ref{s:alpha=0}, it is easy to find a family of scalar products
$\bigl(S^{(\alpha)}_\veps\bigr)_\veps$ on the space $H_M$ such that, for each $\veps$, $S^{(\alpha)}_\veps$ is a
symmetrizer for the operator $\mc{A}^{(\alpha)}_\veps$. Indeed, it is enough to define $S^{(\alpha)}_\veps$  by
a formula anlogous to \eqref{eq:scalar-prod}, where the operator $\Id-\Delta$ is replaced by
$\Id-\veps^{2\alpha}\Delta$ in the first term on the right-hand side of the equality (see equation
\eqref{eq:sp_a} below).
Notice that $S^{(\alpha)}_0$ coincides with the usual $L^2$ scalar product, while $S^{(\alpha)}_1$ is exactly the inner
product defined by formula \eqref{eq:scalar-prod}.

Let us point out that each $\mc{A}^{(\alpha)}_\veps$ can be obtained composing the acoustic propagator $\mc{A}$
with an automorphism $\Lambda^{(\alpha)}_\veps$ of the Hilbert space $H_M$:
$$
\mc{A}^{(\alpha)}_\veps\,=\,\mc{A}\,\circ\,\Lambda^{(\alpha)}_\veps\;,\qquad\qquad
\Lambda^{(\alpha)}_\veps\bigl(r\,,\,V\bigr)\,:=\,\Bigl(\bigl(\Id\,-\,\veps^{2\alpha}\,\Delta\bigr)r\,,\,V\Bigr)\,.
$$
The same can be said also about the scalar products $S^{(\alpha)}_\veps$: namely,
\begin{eqnarray}
\langle (r_1,V_1)\,,\,(r_2,V_2)\rangle_{S^{(\alpha)}_\veps} & := & \langle r_1\,,\,
(\Id-\veps^{2\alpha}\Delta)r_2\rangle_{L^2}\,+\,\langle V_1\,,\,V_2\rangle_{L^2} \label{eq:sp_a} \\
& = & \langle (r_1,V_1)\,,\,\Lambda^{(\alpha)}_\veps\,(r_2,V_2)\rangle_{S^{(\alpha)}_0}\,. \nonumber
\end{eqnarray}

Now, we define the operator $\mc K_{M,\theta}$ as in Paragraph \ref{sss:RAGE}:
$$
\mc{K}_{M,\theta}(r,V)\,:=\,P_M\bigl(\theta\,P_M(r,V)\bigr)\,,
$$
where $P_M\,:\,L^2(\Omega)\times L^2(\Omega)\,\longrightarrow\,H_M$ is the orthogonal projection onto the space $H_M$, defined
by \eqref{eq:def-H_M}, and $\theta\in\mc{D}(\Omega)$ is such that $0\leq\theta\leq1$. Recall that $\mc K_{M,\theta}$ is compact,
self-adjoint and positive.

Following what we have done before, we want to apply Theorem \ref{th:RAGE_eps} to
$$
\mc H\,=\,H_M\;,\quad \Lambda_{\sigma(\veps)}\,=\,\Lambda^{(\alpha)}_\veps\;,\quad \mc{B}_0\,=\,i\,\mc{A}^{(\alpha)}_0\;,\quad
\mc{K}\,=\,\mc{K}_{M,\theta}\;,\quad\Pi_{{\rm cont},\sigma(\veps)}\,=\,Q_{\veps}^{\perp}
$$
(obviously, $\sigma(\veps)\,=\,\veps^{2\alpha}$ here),
where $Q_\veps$ and $Q_\veps^\perp$ denote the orthogonal projections (orthogonal with respect to the scalar product
$\mc S^{(\alpha)}_\veps$) onto respectively ${\rm Ker}\,\mc{A}^{(\alpha)}_\veps$ and
$\bigl({\rm Ker}\,\mc{A}^{(\alpha)}_\veps\bigr)^\perp$.

We apply operator $P_M$ to the system for acoustic waves \eqref{eq:ac-waves_a}:
adopting the same notations as in the previous sections, it becomes
\begin{equation} \label{eq:acoust-eps}
\veps\,\frac{d}{dt}\bigl(r_{\veps,M}\,,\,V_{\veps,M}\bigr)\,+\,\mc{A}^{(\alpha)}_\veps\bigl(r_{\veps,M}\,,\,V_{\veps,M}\bigr)\,=\,
\veps\,\bigl(0\,,\,f_{\veps,\alpha,M}\bigr)\,,
\end{equation}
where uniform bounds give a control analogous to \eqref{est:f_eps-M} also for $\bigl(0\,,\,f_{\veps,\alpha,M}\bigr)$. Notice
that, all the scalar products being equivalent on $H_M$, it is enough to have the bound on the $\mc S^{(\alpha)}_0$ norm.

By use of Duhamel's formula, solutions to the previous acoustic equation can be written as
$$ %\begin{equation} \label{eq:acoust-eps}
\bigl(r_{\veps,M}\,,\,V_{\veps,M}\bigr)(t)\,=\,e^{i\,t\,\mc{B}_{\sigma(\veps)}/\veps}\bigl(r_{\veps,M}\,,\,V_{\veps,M}\bigr)(0)\,+\,
\int^t_0e^{i\,(t-\tau)\,\mc{B}_{\sigma(\veps)}/\veps}\,\bigl(0\,,\,f_{\veps,\alpha,M}\bigr)\,d\tau\,.
$$ %\end{equation}
Again, by definition we have
$$ %\begin{eqnarray*}
\left\|\left(\left(\Lambda^{(\alpha)}_\veps\right)^{-1}\circ\mc{K}_{M,\theta}\right)^{1/2}\,Q^\perp_\veps
\bigl(r_{\veps,M}\,,\,V_{\veps,M}\bigr)\right\|^2_{\sigma(\veps)}\,=\,
\int_\Omega\theta\,\left|Q^\perp_\veps\bigl(r_{\veps,M}\,,\,V_{\veps,M}\bigr)\right|^2\,dx\,. % \\
%& = & \int_\Omega\theta\,\left|Q^\perp_0\bigl(r_{\veps,M}\,,\,V_{\veps,M}\bigr)\right|^2\,dx\,+\,O(\veps)\,.
$$ %\end{eqnarray*}
Therefore, a straightforward application of Corollary \ref{c:RAGE_eps} implies that, for any $T>0$ fixed
and for $\veps$ going to $0$,
\begin{equation} \label{conv:ker-ort_eps}
Q^\perp_\veps\bigl(r_{\veps,M}\,,\,V_{\veps,M}\bigr)\,\longrightarrow\,0\qquad\mbox{ in }\quad L^2\bigl([0,T]\times K\bigr)
\end{equation}
for any fixed $M>0$ and any compact $K\subset\Omega$.

On the other hand, applying operator $Q_\veps$ to equation \eqref{eq:acoust-eps},  we infer that, for any fixed $M>0$,
the family $\bigl(\d_tQ_\veps\bigl(r_{\veps,M}\,,\,V_{\veps,M}\bigr)\bigr)_\veps$ is bounded (uniformly in $\veps$)
in the space $L^2_T(H_M)$. Moreover, as $H_M\hookrightarrow H^m$ for any $m\in\N$, we infer also that
it is compactly embedded in $L^2(K)$ for any $M>0$ and any compact subset $K\subset\Omega$. Hence, as in the previous sections,
Ascoli-Arzel\`a theorem implies that, for $\veps\ra0$,
\begin{equation} \label{conv:ker_eps}
Q_\veps\bigl(r_{\veps,M}\,,\,V_{\veps,M}\bigr)\,\longrightarrow\,\bigl(r_M\,,\,u_M\bigr)\qquad\mbox{ in }\quad
L^2\bigl([0,T]\times K\bigr)\,.
\end{equation}

Thanks to relations \eqref{conv:ker-ort_eps} and \eqref{conv:ker_eps}, the analogue of Proposition \ref{p:strong} still holds true:
namely, we have the strong convergence (up to extraction of subsequences)
$$
r_\veps\,\longrightarrow\,r\qquad\mbox{ and }\qquad
\rho_\veps^{3/2}\,u_\veps\;\longrightarrow\;u\qquad\qquad\mbox{ in }\qquad L^2\bigl([0,T];L^2_{loc}(\Omega)\bigr)\,,
$$
where $r$ and $u$ are the limits which have been identified in Subsection \ref{ss:constraint} and which have to satisfy
the constraints given in Proposition \ref{p:weak-limit}.

The previous strong convergence properties allow us to pass to the limit in the non-linear terms.
Then, the analysis of the limit system can be performed as in Paragraph \ref{sss:limit-system}.

This concludes the proof of Theorem \ref{th:sing-lim} in the remaining cases $0<\alpha<1$.

%%%%%%%%%%%%%%%%%%%%%%%%%%%%%%%%%%%%%%%%%%%%%%%%%%%%%%%%%%%%%%%%%%%%%%%%%%%%%%%%%%%%%%%%%%%%%%%%%%%%%%%%%%
%%%%%%%%%%%%%%%%%%%%%%%%%%%%%%%%%%%%%%%%%%%%%%%%%%%%%%%%%%%%%%%%%%%%%%%%%%%%%%%%%%%%%%%%%%%%%%%%%%%%%%
\appendix
%%%%%%%%%%%%%%%%%%%%%%%%%%%%%%%%%%%%%%%%%%%%%%%%%%%%%%%%%%%%%%%%%%%%%%%%%%%%%%%%%%%%%%%%%%%%%%%%%%%%%%%%%%%%
%%%%%%%%%%%%%%%%%%%%%%%%%%%%%%%%%%%%%%%%%%%%%%%%%%%%%%%%%%%%%%%%%%%%%%%%%%%%%%%%%%%%%%%%%%%%%%%%%%%5

\section{Appendix -- A primer on Littlewood-Paley theory} \label{app:LP}

Let us recall here the main ideas of Littlewood-Paley theory, which we exploited in the previous analysis.
We refer e.g. to \cite{B-C-D} (Chapter 2) and \cite{M-2008} (Chapters 4 and 5) for details.

For simplicity of exposition, let us deal with the $\R^d$ case; however, the construction can be adapted to the $d$-dimensional
torus $\mbb{T}^d$, and then also to the case of $\R^{d_1}\times\mbb{T}^{d_2}$.

\medbreak
First of all, let us introduce the so called ``Littlewood-Paley decomposition'', based on a non-homogeneous dyadic partition of unity with
respect to the Fourier variable. 

We, fix a smooth radial function
$\chi$ supported in the ball $B(0,2),$ 
equal to $1$ in a neighborhood of $B(0,1)$
and such that $r\mapsto\chi(r\,e)$ is nonincreasing
over $\R_+$ for all unitary vectors $e\in\R^d$. Set
$\varphi\left(\xi\right)=\chi\left(\xi\right)-\chi\left(2\xi\right)$ and
$\vphi_j(\xi):=\vphi(2^{-j}\xi)$ for all $j\geq0$.

The dyadic blocks $(\Delta_j)_{j\in\Z}$
 are defined by\footnote{Throughout we agree  that  $f(D)$ stands for 
the pseudo-differential operator $u\mapsto\mc{F}^{-1}(f\,\mc{F}u)$.} 
$$
\Delta_j:=0\ \hbox{ if }\ j\leq-2,\quad\Delta_{-1}:=\chi(D)\qquad\mbox{ and }\qquad
\Delta_j:=\varphi(2^{-j}D)\; \mbox{ if }\;  j\geq0\,.
$$
Throughout the paper we will use freely the following classical property:
for any $u\in\mc{S}',$ the equality $u=\sum_{j}\Delta_ju$ holds true in $\mc{S}'$.

Let us also mention the so-called \emph{Bernstein's inequalities}, which explain
the way derivatives act on spectrally localized functions.
  \begin{lemma} \label{l:bern}
Let  $0<r<R$.   A
constant $C$ exists so that, for any nonnegative integer $k$, any couple $(p,q)$ 
in $[1,+\infty]^2$ with  $p\leq q$ 
and any function $u\in L^p$,  we  have, for all $\lambda>0$,
$$
\displaylines{
{\rm supp}\, \widehat u \subset   B(0,\lambda R)\quad
\Longrightarrow\quad
\|\nabla^k u\|_{L^q}\, \leq\,
 C^{k+1}\,\lambda^{k+d\left(\frac{1}{p}-\frac{1}{q}\right)}\,\|u\|_{L^p}\;;\cr
{\rm supp}\, \widehat u \subset \{\xi\in\R^d\,|\, r\lambda\leq|\xi|\leq R\lambda\}
\quad\Longrightarrow\quad C^{-k-1}\,\lambda^k\|u\|_{L^p}\,
\leq\,
\|\nabla^k u\|_{L^p}\,
\leq\,
C^{k+1} \, \lambda^k\|u\|_{L^p}\,.
}$$
\end{lemma}   

By use of Littlewood-Paley decomposition, we can define the class of Besov spaces.
\begin{defin} \label{d:B}
  Let $s\in\R$ and $1\leq p,r\leq+\infty$. The \emph{non-homogeneous Besov space}
$B^{s}_{p,r}$ is defined as the subset of tempered distributions $u$ for which
$$
\|u\|_{B^{s}_{p,r}}\,:=\,
\left\|\left(2^{js}\,\|\Delta_ju\|_{L^p}\right)_{j\in\N}\right\|_{\ell^r}\,<\,+\infty\,.
$$
\end{defin}

Besov spaces are interpolation spaces between the Sobolev ones. In fact, for any $k\in\N$ and $p\in[1,+\infty]$
we have the following chain of continuous embeddings:
$$
 B^k_{p,1}\hookrightarrow W^{k,p}\hookrightarrow B^k_{p,\infty}\,,
$$
where  $W^{k,p}$ denotes the classical Sobolev space of $L^p$ functions with all the derivatives up to the order $k$ in $L^p$.
Moreover, for all $s\in\R$ we have the equivalence $B^s_{2,2}\equiv H^s$, with
$$
\|f\|_{H^s}\,\sim\,\left(\sum_{j\geq-1}2^{2 j s}\,\|\Delta_jf\|^2_{L^2}\right)^{1/2}\,.
$$

Let us now collect some bounds which are straightforward consequences of Bernstein's inequalities.
The statements are not optimal: we limit to present the properties we used in our analysis.
\begin{lemma} \label{l:density}
\begin{itemize}
 \item[(i)] For $1\leq p\leq2$, one  has $\|f\|_{L^2}\,\leq\,C\bigl(\|f\|_{L^p}\,+\,\|\nabla f\|_{L^2}\bigr)$.
\item[(ii)] For any $0<\delta\leq1/2$ %, there exists $\sigma=(1/2)-\delta$ (and then $\sigma\in[0,1/2[\;$) such that,
and any $1\leq p\leq +\infty$, one has
$$
\|f\|_{L^\infty}\,\leq\,C\left(\|f\|_{L^p}\,+\,\|\nabla f\|^{(1/2)-\delta}_{L^2}\,
\left\|\nabla^2f\right\|^{(1/2)+\delta}_{L^2}\,\right).
$$
\item[(iii)] Let $1\leq p\leq2$ such that $1/p\,<\,1/d\,+\,1/2$. For any $j\in\N$, there exists a constant $C_j$, depending just on
$j$, $d$ and $p$, such that
$$
\left\|\left(\Id\,-\,S_j\right)f\right\|_{L^2}\,\leq\,C_j\,\left\|\nabla f\right\|_{B^0_{p,\infty}}\,.
$$
Moreover, denoting $\beta\,:=\,1\,-\,d(1/p\,-\,1/2)>0$, we have the explicit formula
$$
C_j\,=\,\left(\frac{1}{1-2^{-2\beta}}\right)^{\!\!1/2}\;2^{-\beta(j-1)}\,.
$$
In particular, if $\nabla f=\nabla f_1 + \nabla f_2$, with $\nabla f_1\in B^0_{2,\infty}$ and $\nabla f_2\in B^0_{p,\infty}$, then
$$
\left\|\left(\Id\,-\,S_j\right)f\right\|_{L^2}\,\leq\,\wtilde{C}_j\,\left(\left\|\nabla f_1\right\|_{B^0_{2,\infty}}\,+\,
\left\|\nabla f_2\right\|_{B^0_{p,\infty}}\right)\,,
$$
for a new constant $\wtilde{C}_j$ still going to $0$ for $j\ra+\infty$.
\end{itemize}
\end{lemma}

\begin{proof}
For the first inequality, it is enough to write $f\,=\,\Delta_{-1}f\,+\,(\Id-\Delta_{-1})f$. The former term can be controlled
by $\|f\|_{L^p}$ by Bernstein's inequalities; for the latter, instead we can write
\begin{eqnarray*}
\|(\Id-\Delta_{-1})f\|_{L^2} & \leq & \sum_{k\geq0}\|\Delta_k(\Id-\Delta_{-1})f\|_{L^2} \\
& \leq & C\,\sum_{k\geq0}2^{-k}\,\|\Delta_k(\Id-\Delta_{-1})\nabla f\|_{L^2}\;\leq\;C\,\|\nabla f\|_{L^2}\,,
\end{eqnarray*}
where we used again Bernstein's inequalities and the characterization $L^2\equiv B^0_{2,2}$. 

In order to prove the second estimate, we proceed exactly as before. Again, Bernstein's inequalities allow us to bound low frequencies
by $\|f\|_{L^p}$. Next we have:
\begin{eqnarray*}
\|(\Id-\Delta_{-1})f\|_{L^\infty} & \leq & C\,\sum_{k\geq0}2^{3k/2}\,\|\Delta_k(\Id-\Delta_{-1})f\|_{L^2} \\
& \leq & C\,\sum_{k\geq0}2^{-\delta k}\,\left\||D|^{\delta+3/2}\Delta_k(\Id-\Delta_{-1})f\right\|_{L^2}
\end{eqnarray*}
(for any $0<\delta<1/2$), where we denoted $|D|$ the Fourier multiplier having symbol equal to $|\xi|$.
By interpolation we can write
$$
\left\||D|^{\delta+3/2}\Delta_k(\Id-\Delta_{-1})f\right\|_{L^2}\,\leq\,C\,\left\|\Delta_k(\Id-\Delta_{-1})\nabla f\right\|^{\sigma}_{L^2}\,
\left\|\Delta_k(\Id-\Delta_{-1})\nabla^2f\right\|^{1-\sigma}_{L^2}\,,
$$
for $\sigma\,\in\,]0,1[\,$ (actually, $\sigma=(1/2)-\delta$), and this immediately gives the conclusion.

Let us finally prove the third claim. By spectral localization we can write
\begin{eqnarray*}
\left\|\left(\Id-S_j\right)f\right\|^2_{L^2} & \leq & \sum_{k\geq j-1}\|\Delta_kf\|^2_{L^2}\;\leq\;
\sum_{k\geq j-1}2^{-2k}\|\nabla\Delta_kf\|^2_{L^2} \\
& \leq & \sum_{k\geq j-1}2^{2kd\left(1/p\,-\,1/2\right)}\,2^{-2k}\,\|\nabla\Delta_kf\|^2_{L^p}\,.
\end{eqnarray*}
Keeping in mind that, by hypothesis, $d\left(1/p\,-\,1/2\right)-1=-\beta<0$, we infer the desired inequality
and the explicit expression for $C_j$.
\end{proof}

Finally, let us recall that one can rather work with homogeneous dyadic blocks $(\dot\Delta_j)_{j\in\Z}$, with
$$
\dot\Delta_j:=\varphi(2^{-j}D)\qquad \mbox{ for all }\quad  j\in\Z\,,
$$
and introduce the homogeneous Besov spaces $\dot{B}^s_{p,r}$, defined by the condition
$$
\|u\|_{\dot{B}^{s}_{p,r}}\,:=\,
\left\|\left(2^{js}\,\|\dot\Delta_ju\|_{L^p}\right)_{\!j\in\Z}\,\right\|_{\ell^r}\,<\,+\infty\,.
$$
We do not enter into the details here; we just limit ourselves to recall refined embeddings of homogeneous Besov spaces into Lebesgue
spaces (see Theorem 2.40 of \cite{B-C-D}).
\begin{prop} \label{p:emb_hom-besov}
For any $2\leq p<+\infty$, one has the continuous embeddings $\dot{B}^0_{p,2}\,\hra\,L^p$ and $L^{p'}\,\hra\,\dot{B}^0_{p',2}$.
\end{prop}

\section{Appendix -- Some results from measure theory} \label{app:measure}

We collect here some definitions and results from measure theory. If not otherwise specified, we refer to book \cite{Bog} for a complete treatement of the subject.

First of all, we give the definition of \emph{weak convergence} of measures (see Chapter 8 of \cite{Bog}).
\begin{defin} \label{d:measure-weak}
Let $X$ be a topological space, and let us endow $X$ with the Baire $\sigma$-algebra $\mbb Ba(X)$. Let us denote by $\mc C_b(X)$ the set of all bounded real-valued continuous functions
over $X$. \\
Let $\bigl(\mu_{\eta}\bigr)_{\eta}$ be a one-parameter family of positive measures on a measurable space $\bigl(X,\mbb Ba(X)\bigr)$. 
We say that it converges \emph{weakly} to some Baire measure $\mu$, and we write $\mu_\eta\Rightarrow\mu$, if one has
$$
\lim_{\eta\ra0}\int_Xf\,d\mu_\eta\,=\,\int_Xf\,d\mu
$$
for every $f\in\mc C_b(X)$.
\end{defin}

\begin{rem} \label{r:ba-bo}
We recall that, if $X$ is a metric space, the set $\mc M_{\mbb Ba}(X)$ of Baire measures and the set $\mc M_{\mbb Bo}(X)$ of Borel measures on $X$ coincide.

Moreover, by Theorem 7.1.7 of \cite{Bog}, if $X$ is complete and separable, every Borel measure is a Radon measure. This is the case, for instance, if $X=\R$.
\end{rem}

We need also some auxiliary results. The first one (see Theorem 4.7.25 of \cite{Bog}) provides us with a useful characterization of
compactness in the topology of setwise convergence. We have not used this statement in our proof, but we include it for the sake of completeness.
\begin{prop} \label{p:Bog_setwise}
Let $\mc{M}(X,\Sigma)$ be the family of all bounded countably additive measures on a measurable space $(X,\Sigma)$.
For every set $\mc F\subset\mc{M}(X,\Sigma)$, the following conditions are equivalent.
\begin{enumerate}[(i)]
\item $\mc F$ has compact closure in the topology of the setwise convergence.
\item $\mc F$ has compact closure in the topology of convergence on every bounded $\Sigma$-measurable function.
\item $\mc F$ is bounded in the norm of the total variation, and there exists a non-negative measure $\nu\in\mc{M}(X,\Sigma)$ such that
the family $\mc F$ is uniformly $\nu$-continuous: for every $\veps>0$, there exists a $\delta>0$ such that, for all $A\in\Sigma$ verifying $\nu(A)\leq\delta$, one has
$|\mu(A)|\leq\veps$ for all $\mu\in\mc F$. In this case, all measures in $\mc F$ are absolutely continuous with respect to $\nu$ and the family $\bigl\{d\mu/d\nu\,\bigl|\,\mu\in\mc F\bigr\}$
of densities is compact in the weak topology of $L^1(\nu)$.
\end{enumerate}
\end{prop}

The second result (see Exercice 4.7.132 at page 321-322 of \cite{Bog}) allows us to prove convergence when considering both families of measures and measurable functions depending
on some parameter.
\begin{prop} \label{p:Areskin}
Let $(X,\Sigma)$ be a measurable space, and let $\bigl(\mu_n\bigr)_n$ be a sequence of bounded countably additive positive measures on $X$.
Finally, let $\bigl(f_n\bigr)_n$ a sequence of $\Sigma$-measurable functions. Assume that the $\mu_n$'s converge to some measure $\mu$ in the sense of setwise convergence.

\begin{enumerate}[A)]
 \item If the functions $f_n$ converge $\mu$-almost everywhere to some $\Sigma$-measurable function $f$, then, for all $\theta>0$, one has
\begin{equation} \label{eq:conv_meas}
\lim_{n\ra+\infty}\mu_n\left\{x\in X\;\bigl|\quad|f_n(x)\,-\,f(x)|\,\geq\,\theta\right\}\,=\,0\,.
\end{equation}
\item Suppose that, for all $\theta>0$ fixed, property \eqref{eq:conv_meas} holds true and that the functions $f_n$ are uniformly bounded. Then
$$
\lim_{n\ra+\infty}\int_Xf_n\,d\mu_n\,=\,\int_Xf\,d_\mu\,.
$$
\end{enumerate}
\end{prop}

Now, we present further auxiliary results, which were needed in the proof of Theorem \ref{th:RAGE_eps}. We start by quoting Theorem 8.10.56 of \cite{Bog}.
\begin{prop} \label{p:base}
Let $X$ be a Hausdorff space and let $\mc U_0$ be a basis for its topology. Assume that $\mc U_0$ is stable for countable unions.

Let a sequence $\bigl(\mu_n\bigr)_n$ of Radon measures converge on every set $U\in\mc U_0$. Then $\bigl(\mu_n\bigr)_n$ converges on every Borel set.
\end{prop}

We also mention Theorem 8.2.3 of \cite{Bog}.
\begin{prop} \label{p:weak-open}
Let $X$ be a topological space; let $\bigl(\mu_\alpha\bigr)_\alpha\,\subset\,\mc{M}_{\mbb Ba}(X)$ be a family of positive Baire measures, and $\mu\in\mc M_{\mbb Ba}(X)$ be a given positive Baire
measure.

Then $\mu_\alpha\Rightarrow\mu$ if and only if $\;\;\lim_\alpha\mu_\alpha(X)\,=\,\mu(X)\;$ and one of the following two conditions holds true:
\begin{itemize}
 \item for every functionally closed set $F$, one has
$\;\limsup_\alpha\mu_\alpha(F)\,\leq\,\mu(F)$;
 \item for every functionally open set $U$, one has
$\;\liminf_\alpha\mu_\alpha(U)\,\geq\,\mu(U)$.
\end{itemize}

\end{prop}

From the previous statement, we immediately deduce the next result, which should be compared with Theorem 8.2.7 of \cite{Bog} (given for probability measures).
\begin{coroll} \label{c:weak-open}
Let $X$ be a topological space; let $\bigl(\mu_\alpha\bigr)_\alpha\,\subset\,\mc{M}_{\mbb Ba}(X)$ be a family of positive Baire measures, and $\mu\in\mc M_{\mbb Ba}(X)$ be a given positive Baire
measure.

Then $\mu_\alpha\Rightarrow\mu$ if and only if $\;\lim_\alpha\mu_\alpha(X)\,=\,\mu(X)$ and $\;\lim_\alpha\mu_\alpha(E)\,=\,\mu(E)\;$ for all set $E\in\mbb Ba(X)$ satisfying the following property:
there exist a functionally open set $U$ and a functionally closed set $F$ such that $U\subset E\subset F$ and $\mu(F\setminus U)=0$.
\end{coroll}

%%%%%%%%%%%%%%%%%%%%%%%%%%%%%%%%%%%%%%%%%%%%%%%%%%%%%%%%%%%%%%%%%%%%%%%%%%%%%%%%%%%%%%%%%%%%
%%%%%%%%%%%%%%%%%%%%%%%%%%%%%%%%%%%%%%%%%%%%%%%%%%%%%%%%%%%%%%%%%%%%%%%%%%%%%%%%%%%%%%%%%%%%


{\small
\begin{thebibliography}{xxx}

\bibitem{B-C-D} H. Bahouri, J.-Y. Chemin and R. Danchin: 
{\it ``Fourier Analysis and Nonlinear Partial Differential Equations''}.
Grundlehren der Mathematischen Wissenschaften (Fundamental Principles of Mathematical Sciences),
{\bf 343}, Springer, Heidelberg (2011).

\bibitem{Bog} V. I. Bogachev:
{\it ``Measure theory. Vol. I, II''}.
Springer-Verlag, Berlin (2007).

\bibitem{B-D_2003} D. Bresch, B. Desjardins:
{\it Existence of global weak solution for a $2$D viscous shallow water equation and convergence to the quasi-geostrophic model}.
Comm. Math. Phys., {\bf 238} (2003), n. 1-2, 211-223.

\bibitem{B-D_2006} D. Bresch, B. Desjardins:
{\it On the construction of approximate solutions for the $2$D viscous shallow water model and for compressible Navier-Stokes models}.
J. Math. Pures Appl. (9), {\bf 86} (2006), n. 4, 362-368.

\bibitem{B-D-GV} D. Bresch, B. Desjardins, D. G\'erard-Varet:
{\it Rotating fluids in a cylinder}.
Discrete Cont. Dyn. Syst., {\bf 11} (2004), n. 1, 47-82.

\bibitem{B-D-L} D. Bresch, B. Desjardins, C.-K. Lin:
{\it On some compressible fluid models: Korteweg, lubrication, and shallow water systems}.
Comm. Partial Differential Equations, {\bf 28} (2003), n. 3-4, 843-868.

\bibitem{B-D-Zat} D. Bresch, B. Desjardins, E. Zatorska:
{\it Two-velocity hydrodynamics in fluid mechanics: Part II. Existence of global $\kappa$-entropy solutions
to compressible Navier-Stokes systems with  degenerate viscosities}.
J. Math. Pures Appl., {\bf 104} (2015), n. 4, 801-836.

\bibitem{Carles-Danchin-Saut} R. Carles, R. Danchin, J.-C. Saut:
{\it Madelung, Gross-Pitaevskii and Korteweg}.
Nonlinearity, {\bf 25} (2012), n. 10, 2843-2873.

\bibitem{C-D-G-G} J.-Y. Chemin, B. Desjardins, I. Gallagher, E. Grenier:
{\it ``Mathematical geophysics. An introduction to rotating fluids and the Navier-Stokes equations''}.
Oxford Lecture Series in Mathematics and its Applications, {\bf 32}, Oxford University Press, Oxford (2006).

\bibitem{Cyc-Fr-K-S} H. L. Cycon, R. G. Froese, W. Kirsch, B. Simon:
{\it ``Schr\"odinger operators with application to quantum mechanics and global geometry''}.
Text and Monographs in Physics, Springer-Verlag, Berlin (1987).

\bibitem{F-G-GV-N} E. Feireisl, I. Gallagher, D. G\'erard-Varet, A. Novotn\'y:
{\it Multi-scale analysis of compressible viscous and rotating fluids}.
Comm. Math. Phys., {\bf 314} (2012), n. 3, 641-670.

\bibitem{F-G-N} E. Feireisl, I. Gallagher, A. Novotn\'y:
{\it A singular limit for compressible rotating fluids}.
SIAM J. Math. Anal., {\bf 44} (2012), n. 1, 192-205.

\bibitem{F-N} E. Feireisl, A. Novotn\'y:
{\it ``Singular limits in thermodynamics of viscous fluids''}.
Advances in Mathematical Fluid Mechanics, Birkh\"auser Verlag, Basel (2009).

\bibitem{Foguel} S. R. Foguel:
{\it A perturbation theorem for scalar operators}.
Comm. Pure Appl. Math., {\bf 11} (1958), 293-295.

\bibitem{G_2008} I. Gallagher:
{\it A mathematical review of the analysis of the betaplane model and equatorial waves}.
Discrete Contin. Dyn. Syst. Ser. S, {\bf 1} (2008), n. 3, 461-480.

\bibitem{G-SR_2006} I. Gallagher, L. Saint-Raymond:
{\it Weak convergence results for inhomogeneous rotating fluid equations}.
J. Anal. Math., {\bf 99} (2006), 1-34.

\bibitem{G-SR_Mem} I. Gallagher, L. Saint-Raymond:
{\it ``Mathematical study of the betaplane model: equatorial waves and convergence results''}.
M\'em. Soc. Math. Fr., {\bf 107} (2006).

\bibitem{Hille-P} E. Hille, R. S. Phillips:
{\it ``Functional analysis and semigroups''}.
Americal Mathematical Society Colloquium Publications, vol. XXXI, American Mathematical Society, Providence, R. I. (1974).

\bibitem{J} A. J\"ungel:
{\it Global weak solutions to compressible Navier-Stokes equations for quantum fields}.
SIAM J. Math. Anal., {\bf 42} (2010), n. 3, 1025-1045.

\bibitem{J-L-W} A. J\"ungel, C.-K. Lin, K.-C. Wu:
{\it An asymptotic limit of a Navier-Stokes system with capillary effects}.
Comm. Math. Phys., {\bf 329} (2014), n. 2, 725-744.

\bibitem{K} T. Kato:
{\it ``Perturbation Theory for Linear Operators''},
Classics in Mathematics, Springer-Verlag, Berlin (1995).

\bibitem{Li-Marc} H. Li, P. Marcati:
{\it Existence and asymptotic behavior of multi-dimensional quantum hydrodynamic model for semiconductors}.
Comm. Math. Phys., {\bf 245} (2004), n. 2, 215-247.

\bibitem{Li-Xin} J. Li, Z. Xin:
{\it Global existence of weak solutions to the barotropic compressible Navier-Stokes flows with degenerate viscosities}.
Submitted (2015), \texttt{http://arxiv.org/abs/1504.06826}.

\bibitem{Magnus} W. Magnus:
{\it On the exponential solution of differential equations for a linear operator}.
Comm. Pure Appl. Math., {\bf 7} (1954), 649-673.

\bibitem{M-2008} G. M\'etivier:
{\it ``Para-differential calculus and applications to the Cauchy problem for nonlinear systems''}.
Centro di Ricerca Matematica ``Ennio De Giorgi'' (CRM) Series, {\bf 5}, Edizioni della Normale, Pisa (2008).

\bibitem{Mu-Po-Za} P. B. Mucha, M. Pokorn\'y, E. Zatorska:
{\it Approximate solutions to model of two-component reactive flow}.
Discrete Contin. Dyn. Syst. Ser. S, {\bf 7} (2014), n. 5, 1079-1099.

\bibitem{Mu-Po-Za_2015} P. B. Mucha, M. Pokorn\'y, E. Zatorska:
{\it Heat-conducting, compressible mixtures with multicomponent diffusion: construction of a weak solution}.
SIAM J. Math. Anal., {\bf 47} (2015), n. 5, 3747-3797.

\bibitem{Ped} J. Pedlosky:
{\it ``Geophysical fluid dynamics''}.
Springer-Verlag, New-York (1987).

\bibitem{R-S_I} M. Reed, B. Simon:
{\it ``Methods of modern mathematical physics. I: functional analysis''}.
Academic Press, New York (1980).

\bibitem{R-S_III} M. Reed, B. Simon:
{\it ``Methods of modern mathematical physics. III: scattering theory''}.
Academic Press, New York - London (1979).

\bibitem{V-Yu} A. F. Vasseur, C. Yu:
{\it Existence of global weak solutions for $3$D degenerate compressible Navier-Stokes equations}.
Invent. Math., {\bf 206} (2016), n. 3, 935-974.

\bibitem{V-Yu_2} A. F. Vasseur, C. Yu:
{\it Global weak solutions to compressible quantum Navier-Stokes equations with damping}.
SIAM J. Math. Anal., {\bf 48} (2016), n. 2, 1489-1511.

\end{thebibliography}
}
\end{document}